\documentclass[11pt]{article}
\usepackage{amssymb,amsmath}
\usepackage{enumerate,latexsym}

\usepackage{graphicx}
\usepackage{color}
\usepackage{latexsym}

\usepackage{soul}
\newfont{\bb}{msbm10 at 11pt}
\newfont{\bbsmall}{msbm8 at 8pt}
\newcommand{\TM}{{\cal T}(M)}
\newcommand{\TS}{{\cal T}(\S)}

\def\R{\mathbb{R}}
\def\N{\mathbb{N}}
\def\B{\mathbb{B}}
\def\D{\mathbb{D}}
\def\E{\mathbb{E}}
\def\Z{\mathbb{Z}}

\newcommand{\C}{\mathbb{C}}
\newcommand{\Q}{\mathbb{Q}}

\newcommand{\HH}{\mathbb{H}}

\newcommand{\esf}{\mathbb{S}}
\newcommand{\Te}{\mathbb{T}}

\newcommand{\Hip}{\mathbb{H}}

\newcommand{\Jac}{\mbox{\rm Jac}}

\newcommand{\Div}{\mbox{\rm div}}
\newcommand{\Ric}{\mbox{\rm Ric}}
\newcommand{\rth}{\R^3}

\def\cL{\mathcal{L}}
\def\cF{\mathcal{F}}
\def\cC{\mathcal{C}}
\def\cP{\mathcal{P}}
\def\cD{\mathcal{D}}
\def\cH{\mathcal{H}}
\def\cM{\mathcal{M}}
\def\cS{\mathcal{S}}
\def\cR{\mathcal{R}}
\newcommand{\wt}{\widetilde}
\newcommand{\wh}{\widehat}
\newcommand{\ov}{\overline}
\newcommand{\Int}{\mbox{\rm Int}}

\def\a{{\alpha}}
\def\lc{{\cal L}}
\def\t{{\theta}}

\def\g{{\gamma}}
\def\G{{\Gamma}}
\def\l{{\lambda}}

\def\de{{\delta}}
\def\be{{\beta}}
\def\ve{{\varepsilon}}

\def\centerbmp#1#2#3{\vskip#2\relax\centerline{\hbox to#1{\special
    {bmp:#3 x=#1, y=#2}\hfil}}}

\newtheorem{theorem}{Theorem}[section]
\newtheorem{lemma}[theorem]{Lemma}
\newtheorem{proposition}[theorem]{Proposition}

\newtheorem{remark}[theorem]{Remark}
\newtheorem{corollary}[theorem]{Corollary}
\newtheorem{definition}[theorem]{Definition}
\newtheorem{conjecture}[theorem]{Conjecture}

\newenvironment{proof}
{\smallskip\noindent{\it Proof.}\hskip \labelsep}
                          {\hfill\penalty10000\raisebox{-.09em}{$\Box$}
                          \par\medskip}

\newcommand{\ben}{\begin{enumerate}}
\newcommand{\bit}{\begin{itemize}}
\newcommand{\een}{\end{enumerate}}
\newcommand{\eit}{\end{itemize}}
\renewcommand{\S}{\Sigma}

\newcommand{\su}{{\rm SU}(2)}

\renewcommand{\sl}{\wt{\mathrm{SL}}(2,\R)}
\newcommand{\sol}{\mathrm{Sol}_3}

\newcommand{\ed}{\end{document})}

\begin{document}

\begin{title}
{Constant mean curvature surfaces}
\end{title}

\begin{author}
{William H. Meeks III,    \and Joaqu\'\i n P\' erez,
   \and Giuseppe Tinaglia}
\end{author}
\maketitle

\nocite{me21,gny1,dhkw1}

\begin{abstract}
In this article we survey recent developments in the
theory of constant mean curvature surfaces in homogeneous 3-manifolds, as well as
some related aspects on existence and descriptive results for $H$-laminations and
CMC foliations of Riemannian $n$-manifolds.

\vspace{.15cm} \noindent{\it Mathematics Subject Classification:}
Primary 53A10,
 Secondary 49Q05, 53C42.

\vspace{.1cm} \noindent{\it Key words and phrases:}
Minimal surface, constant mean curvature,
 minimal lamination,  $H$-lamination, CMC foliation, locally simply connected,
Jacobi function, stability, index
of stability, Shiffman function,
curvature estimates, parking garage structure,
local picture on the scale of topology, Willmore energy.
\end{abstract}

\tableofcontents

\section{Introduction.}
\label{secintrod}

We present here a survey of some   recent  developments  in the
theory of constant mean curvature surfaces in homogeneous 3-manifolds
and some related topics on the geometry and  existence of $H$-laminations and
CMC foliations of Riemannian $n$-manifolds.  For the most part, the results
presented in this manuscript are  related to work of the authors.
However, in Section~\ref{SectionReferee}, we  include a brief discussion of some outstanding
results described below:\ben
\item The solution of the Lawson Conjecture {(the Clifford Torus is the unique embedded minimal torus
up to congruencies in the 3-sphere $\esf^3$)}
by Brendle~\cite{bren1,bren2}. More generally,
Brendle~\cite{bren3} proved that Alexandrov embedded constant mean curvature tori in $\esf^3$
are rotational (also see Andrews and Li~\cite{anli1}).
\item The result of Marques and Neves~\cite{mane1,mane2} that a closed embedded  minimal
surface in $\esf^3$ of positive genus has area at least $2\pi^2$, which  is a key tool
in their proof of the Willmore Conjecture (the Clifford Torus is the unique minimizer of the Willmore energy among tori in $\esf^3$).
\item The classification of properly embedded minimal annuli in $\esf^2 \times \R$
 by Hauswirth, Kilian and Schmidt~\cite{hauks1}, from which it follows that such annuli intersect
each level set sphere $\esf^2 \times \{t\}$ in a circle.
\een

We begin by pointing out two
theorems in the classical setting of $\rth$.
The first theorem concerns the classification of properly embedded
minimal planar domains in {3-dimensional Euclidean space}
$\rth$:

\begin{theorem}
\label{thm1}
The plane, the helicoid, the catenoid and the one-parameter family
$\{{\cal R}_t\}_{t>0}$
of  Riemann minimal examples are the only complete,
properly embedded, minimal planar domains in $\rth$.
\end{theorem}
The proof of Theorem~\ref{thm1} depends primarily on work of Colding and
Minicozzi~\cite{cm23,cm35}, Collin~\cite{col1}, L\'opez and Ros~\cite{lor1},
Meeks, P\'erez and Ros~\cite{mpr6} and Meeks and
Rosenberg~\cite{mr8}. The second theorem concerns the classification of complete,
simply connected surfaces embedded in $\rth$ with non-zero constant mean curvature:

\begin{theorem}
\label{thm2}
Complete, simply connected surfaces  embedded in
$\rth$ with non-zero constant mean curvature are compact, and thus, by the
classical results of Hopf~\cite{hf1} or Alexandrov~\cite{aa1}, are round spheres.
\end{theorem}

Theorem~\ref{thm2} was   proven by Meeks and Tinaglia~\cite{mt7} and depends on
obtaining curvature and radius estimates for embedded disks of non-zero
constant mean curvature. We will  cover in some detail the proof of this result
 and will explain how they lead to a deeper understanding of the
geometry of complete constant mean curvature surfaces embedded in Riemannian 3-manifolds.

In the setting of homogeneous 3-manifolds $X$,
we will cover  results on the uniqueness of constant mean curvature spheres, as described in the next problem.
\par
\vspace{.2cm}
\noindent{\em {\sc Hopf Uniqueness Problem:} If $S_1,S_2$ are immersed spheres
in $X$ with the same constant mean curvature, does there exist an isometry $I$ of $X$ with $I(S_1)=S_2$? }
\par
\vspace{.2cm}
This uniqueness question
gets its name from  Hopf~\cite{hf1}, who proved that an immersed sphere in $\rth$ of
constant mean
curvature $H$ is a round sphere of radius $1/|H|$. This  problem is further motivated  by the result
of Abresch and Rosenberg~\cite{AbRo1,AbRo2}
that constant mean curvature spheres in homogeneous 3-manifolds $X$ with a four-dimensional isometry group
are spheres of revolution, from which it can be shown that a positive answer to
the Hopf Uniqueness Problem holds in this special setting.
More recently, the combined results of Daniel and Mira~\cite{dm2} and of Meeks~\cite{me34}
gave a positive solution to the Hopf Uniqueness Problem in the
case that $X$ is isometric to the solvable Lie group
Sol$_3$ equipped with one of its most symmetric left invariant metrics.
In Section~\ref{sec:Hopf}, we will cover in some detail the approach of Meeks, Mira, P\'erez and Ros
in~\cite{mmpr1,mmpr4,mmpr2} to solving
the Hopf Uniqueness Problem
in the  remainder of the possible homogeneous geometries for $X$.
Their approach  includes classification theorems for the moduli space
 $\cM_X$ of immersed constant mean curvature spheres in $X$
in terms of the Cheeger constant of the universal cover of $X$.

Another fundamental
problem that we will cover is the
{\it Calabi-Yau problem} for complete, constant mean curvature
surfaces in locally homogeneous 3-manifolds $X$, especially in the classical case $X=\rth$.
This problem in the case that the surface is {\em embedded} asks the following question.

\par
\vspace{.2cm}
\noindent {\em {\sc Embedded Calabi-Yau Problem:} Does there exist a complete, non-compact surface of
fixed constant mean curvature that is embedded in a given compact subdomain $\Omega$ of $X$? }
\par
\vspace{.2cm}
\noindent Some versions of the Embedded Calabi-Yau Problem
also restrict the topology of the surface and/or assert that
such a surface can be chosen to be proper
in the interior of  $\Omega$  and/or  weaken the condition
that the surface be contained in a compact domain to that of being non-proper in  $X$.

We will also discuss the theory of constant mean curvature $H$-laminations and CMC foliations
of Riemannian $n$-manifolds. By {\em CMC foliation}, we mean a  transversely oriented, codimension-one
foliation $\cF$ of a   Riemannian $n$-manifold $X$ (not necessarily orientable), such that
all of the leaves of $\cF$ are two-sided hypersurfaces of constant mean curvature, and where the value of
the constant mean curvature can vary from leaf to leaf. Notable results on this subject include the
following ones by Meeks, P\'erez and Ros: the Stable Limit Leaf Theorem~\cite{mpr18}, the
Local Removable Singularity for $H$-laminations~\cite{mpr21,mpr10}, the Dynamics Theorem
for properly embedded minimal surfaces~\cite{mpr20},
curvature estimates for CMC foliations of Riemannian
3-manifolds~\cite{mpr21} and the
application of these results to  classify
the CMC foliations of $\rth$ and $\esf^3$ with a closed
countable set of singularities~\cite{mpr21}. In this final section we will give an
outline of the  proof by Meeks and P\'erez~\cite{mpe13}  that a smooth closed $n$-manifold
$X$ admits  a smooth CMC foliation for some Riemannian metric
if and only if its Euler characteristic vanishes, a result that was proved previously
when $X$ is orientable by Oshikiri~\cite{osh3}.

Henceforth for clarity of  exposition, we will call an oriented surface
$M$ immersed in a Riemannian 3-manifold $X$ an {\it $H$-surface} if it
is {\em connected},  {\it embedded} and it
has {\it non-negative constant mean curvature $H$}; our convention of mean curvature
gives that a sphere $\esf^2$ in $\rth$ of radius 1
has $H=1$ when oriented by the inward pointing unit normal
to the ball that it bounds.
If we say that $M$ is an {\it immersed} $H$-surface in $X$,
then that indicates that the surface might {\em not} be embedded.
We will  call an
$H$-surface an {\em $H$-disk} if the surface is homeomorphic
to a closed unit disk in the Euclidean plane.

We now elaborate further on the results mentioned so far
and on the organization of the paper.
The theory of $H$-surfaces in $\R^3$
has its
roots in the calculus of variations developed by Euler and Lagrange
in the 18-th century and in later investigations by, among others, Delaunay, Enneper, Scherk,
Schwarz, Riemann and Weierstrass in the 19-th century. During the
years, many great mathematicians have contributed to this theory: besides the above mentioned
names that belong to the 19-th
century, we find fundamental contributions by  Bernstein, Courant,
Douglas, Hopf, Morrey, Morse, Rad\'o and Shiffman in the first half of the
last century.
Several global questions and conjectures that arose in this
classical subject
have only recently been addressed.

The next two classification results give  solutions to
 long standing conjectures. Concerning the first one, several mathematicians pointed
out to us that Osserman was the first to ask the question
about whether the plane and the helicoid were the only
simply connected, complete 0-surfaces; Osserman described this
question as potentially the most beautiful extension and explanation
of Bernstein's Theorem. For a complete outline of the proof of the second
result below, including Riemann's original proof of the classification of minimal surfaces
foliated by circles and lines in parallel planes, see the historical account
by the first two authors presented in~\cite{mpe16}.

\begin{theorem}
\label{helicoid}
A complete, simply connected $H$-surface in $\rth$
is a plane, a sphere
or a helicoid.
\end{theorem}

\begin{theorem}[Meeks, P\'erez and Ros~\cite{mpr6}]
\label{classthm} Up to scaling and rigid motion, any connected,
properly embedded, minimal planar domain in $\rth$ is a plane, a
helicoid, a catenoid or one of the Riemann minimal
examples. In particular, for every such surface there exists
a foliation of $\R^3$ by parallel planes, each of which intersects
the surface transversely in a connected curve which is a circle or a
line.
\end{theorem}

To understand the context and implications of the next theorem,
first note that every simply connected
homogenous 3-manifold $X$ that
is not isometric  to $\esf^2 (\kappa)\times \R$, where $\kappa$ is the
 non-zero Gaussian curvature of $\esf^2$, is isometric to a {\em metric Lie group}, i.e., a Lie group
equipped with a left invariant metric; see~\cite{mpe11} for a proof of this
fact. In particular, if $X$ is compact, simply connected and homogeneous,
then it is isometric to the Lie group
\begin{equation}
\label{eq:su2}
\su=\{ A\in \mathcal{M}_2(\C )\ | \ A^t\overline{A}=I_2, \ \det (A)=1\}
\end{equation}
with a left invariant metric. When $X$ is
homogenous and diffeomorphic to $\esf^2\times \R$ or more generally when $X$ has a four-dimensional
isometry group,
Abresch and Rosenberg~\cite{AbRo1,AbRo2}  proved that
for every  $H\geq0$, there exists a unique immersed $H$-sphere in $X$
and this sphere is embedded when $X$ is
diffeomorphic to $\esf^2\times \R$;
they  obtained these results  by first proving that
every such
sphere is a surface of revolution and then, using this symmetry property,
they classified the examples. In the classical
setting of $X=\rth$,  Hopf~\cite{hf1} earlier proved that
an immersed  $H$-sphere is a round sphere of radius $1/H$.  Motivated by these results,
the uniqueness up to ambient isometry question for immersed $H$-spheres in $X$ became known as the
previously mentioned Hopf Uniqueness
Problem in homogeneous 3-manifolds; since spheres are simply connected and lift to the universal cover
of $X$,
we henceforth will only consider this uniqueness problem with the additional condition that the
homogeneous 3-manifold $X$ be simply connected.

The next theorem by
Meeks, Mira, P\'erez and Ros~\cite{mmpr4} gives a complete solution to
the Hopf Uniqueness Problem and to the classification of immersed $H$-spheres
when the homogeneous manifold $X$ is diffeomorphic to $\esf^3$. These authors
are confident that they also
have a proof of the  classification  for the moduli space of
immersed $H$-spheres in a general simply connected, homogeneous 3-manifold $X$,
and this is work in progress in~\cite{mmpr1}.
Their proposed classification result depends
on their characterization in~\cite{mmpr2} of the Cheeger constant of $X$ as being twice the value of
the infimum of the mean curvatures of immersed closed $H$-surfaces in the space.
See Theorem~\ref{main5} in Section~\ref{sec:Hopf} for a more complete version of the next theorem
and for further explanations.

\begin{theorem}[Compact case of the Hopf Uniqueness Problem] \label{main3}
Let  $X$ be $\su$ equipped with a left invariant metric and let $\cM_X$ be the moduli space
of immersed constant mean curvature spheres in $X$ identified up to left translations.
Then for every  $H\in[0,\infty)$ there exists an oriented immersed  $H$-sphere $S_H$ in $X$
and  $S_H$  is the unique immersed $H$-sphere in $X$   up to
left translations. Hence, $\cM_X$ is naturally parameterized by the interval $[0,\infty)$
of all possible mean curvature values.
\end{theorem}

The proofs of Theorems~\ref{helicoid}, \ref{classthm} and~\ref{main3} depend on a series of new results
and theory that have been developed over the past decade. The
purpose of this article is two-fold. The first goal is to explain these
results and the history behind them in a manner accessible to a graduate student interested
in Differential Geometry or Geometric Analysis, and the second goal is to explain how these results
and theory transcend their application to the proofs of Theorems~\ref{helicoid},
\ref{classthm} and \ref{main3} and enhance the understanding of the theory,
giving rise to new theorems and conjectures. Since much of
this material for minimal surfaces is well-documented in the survey~\cite{mpe2} and book~\cite{mpe10} by
the first two authors, we will focus somewhat more of our attention here
on the case when $H>0$ and we refer the interested reader to~\cite{mpe16,mpe2,mpe10} for further
background on the minimal surface results that we mention here.

Before proceeding, we make a few general comments on the proof of Theorem~\ref{helicoid} that we feel can
suggest to the reader a visual idea of what is going on. The most natural motivation for
understanding this theorem, Theorem~\ref{classthm} and
other results presented in this survey is to try to answer the
following heuristic question:
\par
\vspace{.2cm}
{\it What are the possible shapes of surfaces
which satisfy a variational principle and have a given topology?}
\par
\vspace{.2cm}
\noindent For instance, if the variational equation expresses the critical
points of the area functional with respect to compactly supported volume preserving
variations, and the requested topology is the
simplest one of a disk, then Theorem~\ref{helicoid} says that the
possible shapes for complete non-compact examples are the trivial one given by
a plane and (after a rotation)
an infinite double spiral staircase, which is a visual description of a vertical helicoid;
in particular there are no non-compact examples which are not minimal.

A more precise description of the double spiral staircase nature of a
vertical helicoid is that this surface is the union of two infinite-sheeted
multi-valued graphs, which are glued along a vertical axis. Crucial in the proof of
Theorem~\ref{helicoid} are local and global results of Colding and Minicozzi
on  0-disks~\cite{cm21,cm22,cm23,cm35}, global results of Meeks and
Rosenberg for complete 0-disks~\cite{mr8}, and generalizations
of them by Meeks and Tinaglia to the $(H=1)$-setting~\cite{mt8,mt7,mt13,mt9}.
The local results of Colding and Minicozzi
describe  the structure of compact embedded  minimal disks (with boundary)
as essentially being modeled by the plane or the helicoid, i.e., either
they are graphs or pairs of finitely sheeted multi-valued graphs glued along
an ``axis''. In the case of 1-disks,  the recent work of Meeks and Tinaglia
demonstrates that 1-disks are modeled only by graphs away
from their boundary curves, in other words,
there exist curvature estimates for $1$-disks at points at any fixed positive
intrinsic distance  from their boundary curves.

\begin{theorem}[Curvature Estimates, Meeks, Tinaglia~\cite{mt7}]
\label{cest} Given $\delta$, $\cH>0$,
there exists a  $K(\delta,\cH)\geq\sqrt{2}\cH$ such that
any   $H$-disk $M$  in $\rth$ with
 $H\geq \cH$ satisfies
\[
\sup_{\large \{p\in {M} \, \mid \,
d_{M}(p,\partial M)\geq \delta\}} |A_{M}|(p)\leq  K(\delta,\cH),
\]
where $|A_{M}|$ is the norm of the second fundamental form
and $d_{M}$ is the intrinsic distance function of $M$.
\end{theorem}

We wish to emphasize  that the curvature estimates for
$H$-disks given in
Theorem~\ref{cest} depend {only} on the fixed {\em lower} positive bound
$\cH$ for their mean curvature, and we next explain a
simple but important consequence of this observation.
Recall that the {\em radius} of a compact Riemannian surface with
boundary is the maximum intrinsic distance of points in the surface to its boundary;
we claim that the  radius of a  1-disk
must be less than  ${K(1,1)}$,  where $K(1,1)$ is the constant given in the
above theorem with $\delta=1, \cH=1$. To see this, let $\S$ be a  1-disk
and let $\wh{\S}=\frac{1}{K(1,1)}\cdot \S$ be the homothetic scaling
of $\S$ by the factor $\frac{1}{K(1,1)}$.  Note that the mean curvature
of  $\wh{\S}$ is $K(1,1)\geq\sqrt 2>1$ and thus, the classical inequality
Trace$(A)\leq \sqrt{2}|A|$ valid for every symmetric $2\times 2$ real matrix $A$ implies that
\begin{equation}
\label{eq:a}
\inf_{ p\in \wh{\S}}|A_{\wh{\S}}|(p) \geq\sqrt{2}\cdot  K(1,1)>K(1,1).
\end{equation}
 Therefore, the radius of $\wh{\S}$ must be less than 1, otherwise $ \{p\in {\wh \S} \, \mid \,
d_{\wh\S}(p,\partial \wh\S)\geq 1\}\neq\O$,
and then Theorem~\ref{cest} with $\delta=1, \cH=1$ would give
\[
\sup_{\large \{p\in \wh{\S} \, \mid \,
d_{\wh \S}(p,\partial \wh \S)\geq 1\}} |A_{\wh\S}|(p)\leq  K(1,1),
\]
contradicting (\ref{eq:a}). This contradiction implies that the radius of $\S=K(1,1)\cdot \wh{\S}$
is less than  ${K(1,1)}$, which proves our claim.
With these considerations in mind, it is perhaps not too surprising
that the proof of the  curvature estimates in Theorem~\ref{cest}
is intertwined with the
proof of the following result on the existence of radius estimates for $(H>0)$-disks.

Notice that the next theorem implies that there do not exist
complete $(H>0)$-planes in $\rth$, since such planes contain topological disks
of arbitrarily large radius, which resolves the $H>0$ case in the proof of Theorem~\ref{helicoid}.

\begin{theorem}[Radius Estimates, Meeks, Tinaglia~\cite{mt7}]
\label{rest}
There exists an ${\mathcal R}\geq \pi$ such that any $H$-disk in $\rth$ with
$H>0$ has radius less than ${{\mathcal R}}/{H}$.
\end{theorem}

Another important result in the proof of Theorem~\ref{helicoid}, as well as
in the proofs of Theorems~\ref{classthm}, \ref{cest} and \ref{rest},
concerns global aspects of limits of $H$-disks and genus-zero $H$-surfaces, which were first
described by Colding and Minicozzi in their Lamination Theorem for 0-Disks and
more recently  by Meeks and Tinaglia in their Lamination Theorem for $H$-Disks, see
Theorem~\ref{thmlimitlaminCM} below. A last key ingredient in the proofs
of the aforementioned theorems is the following
chord-arc  result that allows one  to relate
intrinsic and extrinsic  distances on an $H$-disk at points far from
its boundary and at the same time near to points where the surface is not too flat;
this chord-arc result implies that a complete simply connected $H$-surface
must be properly embedded in $\rth$.  The proof of the next theorem by
Meeks and Tinaglia~\cite{mt8} depends
on results in~\cite{mt7,mt13,mt9} and the strategy of their proof follows and generalizes
the  proof of a  similar chord-arc estimate for 0-disks
by Colding and Minicozzi in~\cite{cm35}.

We will denote by $d_{\Sigma }$, $B_{\Sigma }(p,r)$ respectively the intrinsic distance function and the open intrinsic
ball of radius $r>0$ centered at a point $p$ in a Riemannian surface $\Sigma $.


\begin{theorem}[Chord-arc property  for $H$-disks]
\label{main2}
There exists a $C>1$ so that the following holds.
Suppose that $\S \subset \R^3$
 is an $H$-disk, $ \vec{0}\in \S$ and $R>r_0>0$.
If 
${B}_\S(\vec 0,CR)\subset
\S-\partial \S$ and
$\sup_{B_\S(\vec 0,(1-\frac{\sqrt 2}2)r_0)}|A_\S|>r_0^{-1}$,
then
$$ \frac{1}{3}d_{\Sigma }
(x,\vec{0})\leq \frac{1}{2}\| x\| +r_0, \;
\mbox{\rm for all } x\in B_\S(\vec 0,R).$$
\end{theorem}

Our survey is organized as follows. We present the main definitions
and background material in the  introductory Section~\ref{BR}. In
that section we also briefly describe geometrically, as well as
analytically, some of the important classical examples of proper
 $H$-surfaces in $\rth$ that we will need later on; understanding
these key examples is crucial in obtaining a feeling for this subject
(as in many other branches of mathematics), as well as in
making important theoretical advances and
asking the right questions. Before going further,
the  reader will probably benefit by taking a few minutes to view and
identify the computer graphics images of these surfaces that
appear near the end of {Section}~\ref{subsecexamples}, and to read the
brief historical and
descriptive comments related to the individual images.

In Section~\ref{sec:flux} we cover conservation laws that pair Killing fields in
a Riemannian $n$-manifold $X$  with elements in the homology group $H_{n-2} (M)$ of any
 $H$-hypersurface $M$ in $X$. In our setting
of Riemannian 3-manifolds $X$, these conservation laws are interpreted as scalar fluxes induced
by a Killing field $K$ across 1-cycles $\g$ on
an $H$-surface $M$, and these fluxes only depend on the
homology class of the 1-cycle.  These flux invariants
play an important role in describing both local and global aspects of the
geometry of  $H$-surfaces in homogeneous 3-manifolds, as we will illustrate in later sections.

In Section~\ref{subsecends} we summarize
a number of results concerning proper
$H$-surfaces in $\rth$ of finite genus
as seen in the light of the  recent
contributions of Colding and Minicozzi~\cite{cm35}
and Meeks and Tinaglia~\cite{mt7} that demonstrate that complete
$H$-surfaces in $\rth$ of finite topology are proper.
Also we  briefly explain here the recent classification of proper  0-surfaces that
are planar domains of infinite topology by Meeks, P\' erez and Ros~\cite{mpr6},
as well as the description by these authors of
the asymptotic behavior of proper 0-surfaces with  finite genus and an infinite number of ends.
In particular, we will explain how Theorems~\ref{helicoid}
and \ref{classthm} follow from the results described
in this section. At the end of Section~\ref{subsecends} we explain in
some detail the
analytic construction by Meeks and P\' erez~\cite{mpe3}
of certain proper 0-annuli with boundary $E_{a,b}$, where $(a,b)\in [0,\infty)\times \R $,
that are models for ends of $0$-annuli in $\R^3$ with infinite total curvature;
in other words, every complete injective 0-immersion
$\psi\colon \esf^1\times [0,\infty) \to \rth$ with infinite total curvature
is, after a rigid motion, asymptotic to exactly one of the annuli $E_{a,b}$.
These special embedded 0-annular ends $E_{a,b}$ are called {\it canonical
ends} and their geometry is related to the distinct flux vectors of their
boundary curves, after making certain geometric normalizations.

In Section~\ref{SectionReferee} we cover recent results
of Brendle~\cite{bren2} on his solution of the Lawson Conjecture,
the classification of complete embedded minimal annuli in $\esf^2 \times\R$, all of which are
periodic and which intersect the level set spheres $\esf^2\times\{t\}$ in round circles
by Hauswirth, Kilian and Schmidt~\cite{hks1}, and
the area estimate from below by $2\pi^2$ for closed embedded minimal surfaces of positive
genus in $\esf^3$ by Marques and Neves~\cite{mane1,mane2},
 which led them to a proof of  the Willmore conjecture.

In Sections~\ref{sec4} and \ref{section6}
we study limits of sequences of $H$-surfaces. Depending on whether or
 not such a sequence has uniform local bounds for the area and/or for the
  second fundamental form, new objects can appear in the limit. For
   instance, in presence of local uniform bounds for the area and
   second fundamental form of the surfaces in the sequence, the
   classical Arzel\`{a}-Ascoli theorem implies subsequential
convergence to an $H$-surface. When the sequence has local uniform
bounds for the second fundamental form but it fails to have local uniform bounds for the area,
then (weak) $H$-laminations appear in the limit; this last notion will be studied in Section~\ref{section6}.
The reader not familiar with the subject of weak $H$-laminations
should think about a
geodesic $\g $ on a Riemannian surface.  If $\g $ is complete and
embedded (a one-to-one immersion), then its closure is a geodesic
 lamination ${\cal L}$ of the surface.
When $\g $ has no accumulation points, then it is proper and it is the unique leaf of ${\cal L}$.
Otherwise, there pass complete, embedded, pairwise disjoint geodesics through the
accumulation points, and these geodesics together with $\g $ form
the leaves of the geodesic lamination
${\cal L}$. A similar result is true for a complete
$H$-surface of locally  bounded curvature (i.e., whose norm of the second fundamental form is
bounded in compact extrinsic balls) in a
Riemannian 3-manifold~\cite{mr13}. However, when $H>0$, two leaves of the resulting lamination might
intersect non-transversely at some point $p$ where the unit normal
vectors to the leaves point in opposite directions, and in this case
we call this structure a {\em weak} $H$-lamination;
still it holds that nearby such a point
$p$ and on the mean convex side of each of the two intersecting leaves,
there is a lamination structure (no intersections). In Section~\ref{section6} we also
cover the Stable Limit Leaf Theorem of Meeks, P\' erez and Ros~\cite{mpr19,mpr18} and
the Limit Lamination Theorem for $0$-surfaces of
Finite Genus by Colding and Minicozzi~\cite{cm25}.

In Section~\ref{ttttmr} we explain some further local and global
results by Colding and Minicozzi  in~\cite{cm35}, where among
other things they prove that complete 0-surfaces of
finite topology in $\rth$ are proper. We explain here the results
of Meeks and Rosenberg~\cite{mr13} on
generalizations of the work of Colding and Minicozzi in~\cite{cm35}
to the Riemannian 3-manifold setting.

In Section~\ref{LRST}, we examine how the theoretical results
in the previous sections
lead to deep global results in the classical theory in $\R^3$, as well as to
a general understanding of the local geometry of any
complete $H$-surface $M$ in any homogeneously regular 3-manifold
(see Definition~\ref{defhomgreg} below for the concept of homogeneously regular
3-manifold). This local
description is given in two local picture theorems by Meeks, P\'erez
and Ros~\cite{mpr20,mpr14}, each of which describes
the local extrinsic geometry of $M$
 near points of concentrated curvature (the {\it Local Picture Theorem on
the Scale of Curvature}) or of concentrated topology
(the {\it Local Picture Theorem on the Scale  of Topology}). In order to understand the second
local picture theorem, we develop in this section
 the important notion of a {\it minimal
  parking garage structure} on
$\rth$, which is one of the possible limiting pictures in the
topological setting. Crucial in these local pictures is a local result that calculates
the rate of growth of the norm of the second fundamental form of an
$H$-lamination in a punctured ball of a Riemannian 3-manifold when approaching
a singularity of the lamination occurring at the center of the ball
(the {\it Local Removable Singularity Theorem}). Global applications of the Local Removable
Singularity Theorem to the classical theory are also discussed here; the most
important of these applications are the
{\it Quadratic Curvature Decay Theorem}
and the {\it Dynamics Theorem} for proper 0-surfaces in $\R^3$
by Meeks, P\'erez and Ros~\cite{mpr10}.

In Sections~\ref{sec:dynamics} and \ref{sec:MT} we cover some results of
Meeks and Tinaglia mentioned previously, as well as their Dynamics
and Minimal Elements Theorems
for complete strongly Alexandrov embedded $1$-surfaces in $\rth$ from~\cite{mt4}.
This  Minimal Elements Theorem
is needed in the proofs of  the curvature and radius estimates
stated previously in Theorems~\ref{cest}
and \ref{rest}.

In Section~\ref{seccy}, we briefly discuss what are usually referred to as
 the {\it Calabi-Yau problems} for
complete $H$-surfaces in $\rth$ and in homogeneous 3-manifolds.
These problems arose from questions
asked by Calabi~\cite{ca1} and Yau (see page 212 in~\cite{che4}
and problem 91
in~\cite{yau1}) concerning the existence of complete,
immersed 0-surfaces that are constrained to lie in a given
region of $\rth$, such as in a bounded domain. Various aspects of
the Calabi-Yau problems constitute
 an active field of research with an
interesting mix of positive and negative results. We include here a
few recent fundamental advances on this problem that are not covered adequately in
previous sections of the survey.  We end this section with the fundamental
existence Conjecture~{\ref{CY1}}
on the embedded Calabi-Yau problem for
complete 0-surfaces.

In Section~\ref{sec:Hopf} we discuss recent results on the Hopf Uniqueness Problem,
as the aforementioned Theorem~\ref{main3}.
Section~\ref{sec:CMC}  is devoted to material on the existence and geometry of
CMC foliations of Riemannian $n$-manifolds. This section includes results
by Meeks, P\'erez and Ros on the classification of
CMC foliations of $\rth$ or $\esf^3$ with a countable number of singularities given
in Theorem~\ref{thmspheres} in the general setting of weak CMC foliations
and their  curvature estimates given in Theorem~\ref{thm5.7} for weak CMC foliations of Riemannian
3-manifolds, as well as an existence theorem for Riemannian metrics together with
CMC foliations in compact $n$-dimensional manifolds $X$ with
the property that the Euler characteristic of
$X$ is zero (Meeks and P\'erez~\cite{mpe13}).

The final Section~\ref{sec:conj} of this survey is devoted to a discussion of
some of the outstanding conjectures on the geometry of $H$-surfaces in locally homogeneous
3-manifolds.

\vspace{.2cm} \noindent {\sc Acknowledgments:}  The authors would
like to thank  Matthias Weber for contributing the beautiful computer
graphics images of classical minimal surfaces to our {Section}~\ref{subsecexamples} of examples
of $H$-surfaces.
We also thank the referee for valuable comments on this article.

\vspace{.2cm} \noindent
First author's financial support: This material is based upon work for
the NSF under Award No. DMS-1309236. Any opinions, findings, and
conclusions or recommendations expressed in this publication are those
of the authors and do not necessarily reflect the views of the NSF.
Second  author's financial support: Research partially supported by a
MINECO/FEDER grant no. MTM2014-52368-P.
Third author's financial support:  Research
partially supported by
EPSRC grant no. EP/M024512/1.

\section{Basic results in theory of  $H$-surfaces in $\rth$.}
\label{BR}

We will devote this section to giving a fast tour through
the foundations of the theory, providing enough material for the reader
to understand the results to be explained in future sections.
While our exposition here emphasizes $H$-surfaces in the classical $\rth$ setting, we
will sometimes mention how the concept of  $H$-surface generalizes to the
Riemannian 3-manifold setting.
In the sequel, $B(p,r)$ will denote the open ball centered at a point
$p\in \R^3$ with radius $r>0$.

\subsection{Equivalent definitions of $H$-surfaces.}
One can define an $H$-surface from different points of view. The
equivalences between these starting points give  insight into the
richness of the classical theory of $H$-surfaces in $\rth$ and its
connections with other branches of mathematics.

Throughout the paper, all surfaces will be assumed to be
orientable unless otherwise stated. Consider the Gauss
map $N\colon M\to \esf ^2$ of a surface $M\subset \R^3$.
Then, the  tangent space $T_pM$ of~$M$ at
$p\in M$ can be identified as a subspace of $\R^3$ under parallel translation
with the tangent space $T_{N(p)}\esf^2$ to the unit sphere at $N(p)$. Hence,
one can view the
differential $A_M(p)=-dN_p$ as an endomorphism of $T_pM$, called the
{\it shape operator}. $A_M(p)$ is a symmetric linear transformation,
whose orthogonal eigenvectors are called the {\it principal directions} of
$M$ at $p$, and the corresponding eigenvalues are the {\it principal
curvatures} of $M$ at $p$. Since the (possibly non-constant)
 mean curvature function $H$ of $M$ equals the arithmetic mean of such principal curvatures
(or the average normal curvature), then we can write
\begin{equation}
\label{eq:AM}
A_M(p)=-dN_p=\left( \begin{array} {cc}
{H}+a & b \\ b &
{H} -a\end{array}\right)
\end{equation}
in an orthonormal tangent basis (here $H,a,b$ depend on $p$).

Note that by the Cauchy-Riemann equations,
when $H$ is identically zero, then the Gauss
map of $M$ is anticonformal when the sphere $\esf^2$ is
taken with its outward pointing normal, and it is
conformal when the sphere $\esf^2$ is taken with inward pointing normal,
which is the orientation induced  by stereographic
projection of $\esf^2$ from its north pole
$(0,0,1)$ to $\C\cup\{\infty\}$, and we denote this
meromorphic function by $g\colon M\to \C \cup \{ \infty \} $.

\begin{definition}
{\em
The formula $\langle A,B\rangle =\mbox{Trace}(AB)$ endows the space of $2\times 2$ real
symmetric matrices with a positive definite inner product, with associated
norm $|A|=\sqrt{\sum _{i,j}a_{ij}^2}$ if $A=(a_{ij})_{i,j}$.
The norm of the second fundamental form $|A_M|(p)$ of $M$
at the point $p$ is the norm of the matrix given by (\ref{eq:AM}), or equivalently,
$|A_M|(p)=\sqrt{\l_1^2 +\l_2^2}$, where $\l_1,\l_2$
are the principal curvatures of $M$ at $p$.

}
\end{definition}

\begin{definition} \label{2.2}
{\em
\ben \item A surface $M\subset \R^3$ is {\it minimal} if and only if its
mean curvature vanishes identically.
\item A surface $M\subset \R^3$ is an {\em $H$-surface} if and only if has
constant mean curvature $H\in \R $, which we will always assume is non-negative
after appropriately orienting $M$.
\een}
\end{definition}

Often, it is useful to identify a Riemannian surface
$M$ with its image under an isometric
embedding. Since minimality is a local concept, the notion
of minimality can be
applied to an isometrically immersed surface $\psi \colon M\to \R^3$.
Recall the well-known vector-valued formula
\[
\Delta \psi = 2H N,
\]
where $\Delta$ is the Riemannian Laplacian on $M$, and
$H\colon M\to \R$ is the mean curvature function of $M$ with respect to
the Gauss map $N$. In particular, the coordinate functions of
an immersed 0-surface are harmonic.

Let $\Omega $ be a subdomain with compact closure in a surface
$M\subset \R^3$. If we
perturb normally the inclusion map $\psi$ on $\Omega $ by a compactly
supported smooth function $u\in C^{\infty }_0(\Omega )$, then
$\psi +tuN$ is again an
immersion whenever $|t|<\ve $,
for some $\ve $ sufficiently small. The mean curvature function $H$
of $M$ relates to the infinitesimal variation of the area functional
$A(t)=\mbox{Area}[(\psi +tuN)(\Omega )]$ for compactly
supported normal variations by means of the {\it first variation of
area}
(see for instance~\cite{ni2}):
\begin{equation}
\label{eq:1vararea}
A'(0)=-2\int _{\Omega }uH\, dA,
\end{equation}
where $dA$ stands for the area element of $M$. Formula
(\ref{eq:1vararea}) implies that compact immersed 0-surfaces are critical points of the
area functional for compactly supported variations.  In fact,
a consequence of the {\it second variation of area}
is that any point in
a 0-surface has a
neighborhood with least-area relative to its boundary. This property
justifies the word ``minimal'' for these surfaces.

Another consequence of (\ref{eq:1vararea}) is that
when $M$ is a compact $H$-surface with boundary (now $H\in [0,\infty )$
is a constant), then $M$ is a critical point of the
area functional for compactly supported variations that infinitesimally
preserve the volume, i.e., for functions $u\in C^{\infty }_0(M)$ with
$\int_M u\, dA=0$.
This fact can be generalized to a
Riemannian 3-manifold $X$ and
explains why a compact smooth domain $W$ in
$X$ whose boundary surface area is critical
with respect to the areas of the boundaries of nearby smooth domains with
the same volume as $W$, must have boundary $M=\partial W$ with
constant mean curvature.  When such a domain  $W$
in $X$ has least area with respect to the boundaries of all smooth
compact subdomains in $X$ with volume $V$, then $\Omega $ is called a
{\em solution to the isoperimetric problem in $X$ for the volume $V$. }

The above discussion
establishes 0-surfaces
as the $2$-dimensional analog to geodesics in Riemannian geometry,
and connects the theory of $H$-surfaces with one of the most
important classical
branches of mathematics:
the calculus of variations. Coming back to our isometric immersion $\psi \colon
M\to \R^3$, another well-known functional in
the calculus of variations besides the area functional~$A$ is the {\it Dirichlet energy,}
\[
E=\int _{\Omega }|\nabla \psi |^2 dA,
\]
where again $\Omega \subset M$ is a subdomain with compact closure.
These
functionals are related by the inequality $E\geq 2A$, with equality
if and only if $\psi $  is conformal. This conformality condition is
not restrictive, as follows from the existence of local  isothermal or conformal
coordinates for any 2-dimensional Riemannian manifold, modeled on
domains of $\C $.

From a physical point of view, the mean curvature function of a
homogeneous membrane (surface) separating two media is equal, up to a non-zero
multiplicative
constant, to the
difference between the pressures at the two sides of the surface.
When this pressure
difference is zero, then the membrane has zero mean curvature.
Therefore, soap films
in space are physical realizations of the ideal concept of a
0-surface and soap bubbles are physical realizations of the ideal concept of an
($H>0$)-surface.

We now summarize these various properties for 0- and $H$-surfaces.

\begin{definition}
\label{defmin}
 {\rm
Let $\psi =(x_1,x_2,x_3)\colon M\to \R^3$ be an isometric immersion of a
Riemannian surface
into space and we identify $M$ with its image. Then, $M$ is {\it minimal},
or equivalently an immersed 0-surface, if and only if
any of the following equivalent properties hold:
\ben[1.]
\item The mean curvature function of $M$ vanishes identically.
\item The coordinate function $x_i$ is a harmonic function on~$M$ for each $i$.
In other words, $\Delta x_i = 0$, where $\Delta$ is the Riemannian Laplacian on $M$.
\item  $M$ is
a critical point of the {\it area functional} for all compactly
supported variations.
\item Every point $p\in M$ has a
neighborhood  $D_p$  with least area relative to its boundary.
\item $M$ is a critical point
of the Dirichlet
energy for all compactly supported variations, or equivalently
if any point $p\in M$ has a
neighborhood  $D_p$ with least energy relative to its boundary.
\item Every point $p\in M$ has a
neighborhood $D_p$ that is equal to the unique idealized soap
film with boundary $\partial D_p$.
\item The stereographically
projected Gauss map $g\colon M\to \C \cup \{ \infty \} $ is
meromorphic with respect to the underlying
 Riemann surface structure on $M$.
\een
}
\end{definition}

\begin{definition}
\label{defH}
 {\rm
 Let $\psi =(x_1,x_2,x_3)\colon M\to \R^3$ be an injective isometric immersion of a
  Riemannian surface
into space and we identify $M$ with its image. Then, $M$ is an {\it $H$-surface}
for some $H\geq 0$ if and only if
any of the following equivalent properties hold:
\ben[1.]
\item The mean curvature function of $M$ is constant.
\item  $M$ is a critical point of the {\it area functional} for all compactly
supported volume preserving normal variations.
\item Every point $p\in M$ has a
neighborhood $D_p$ that is equal to an idealized soap
bubble with boundary $\partial D_p$, i.e., considering $\partial D_p$
to be a wire, then $D_p$ is realizable by a soap bubble
bounding $\partial D_p$ where the air
pressure has a constant difference on its opposite sides.
\item Given a point $p\in M$, there exists a small $\ve>0$ such that
the component $D_p$  of $\B(p,\ve)\cap M$ containing $p$,
which is part of  the oriented boundary of a component $W$ of $\B(p,\ve)-D_p$,
satisfies the following constrained area-minimizing property.
For any compact embedded oriented surface $\S\subset \B(p,\ve)$
with $\partial \S=\partial D_p$ that
is homologous in $\B(p,\ve)$ to $D_p$ relative to its boundary and which
lies in the oriented boundary of a component $W_\S$ of
$\B(p,\ve)-\S$ with the same volume as
$W$, then the area of $\Sigma $ is not less than the area of $D_p$.
\een
}
\end{definition}

This concludes our discussion of the equivalent definitions of $H$-surfaces.
Returning to our background discussion, we note that Definition~\ref{defmin} and
the maximum principle for harmonic functions imply  that
no compact immersed 0-surfaces in $\R^3$
without boundary exist. On the contrary, there exist many immersed closed
surfaces with non-zero constant mean
curvature~\cite{kap2,kap3,we1} but by the next classical result
this is not possible for spheres.  We state the next theorem of Hopf in the
3-dimensional space form setting, where his original proof in~\cite{hf1}
can be adapted.

\begin{theorem}[Hopf Theorem] \label{hopf}
An immersed $H$-sphere  in a complete, simply
connected 3-dimensional manifold $\Q^3(c)$ of constant sectional
curvature $c$ is a round sphere.
\end{theorem}

Recall that $H$-surfaces are assumed to be embedded,
whereas {\em immersed} $H$-surfaces need not be.
Round spheres are also  the only  closed $H$-surfaces in $\rth$.
This uniqueness result follows from  the  classical result of Alexandrov below
and its proof is based on the so called {\em Alexandrov reflection principle,} which in turn
is based on the interior and boundary maximum principles for $H$-surfaces
given in Theorems~\ref{thmintmaxprin} and \ref{thmintmaxprinb} below.
Motivated by the importance
of the Alexandrov reflection principle, we will briefly explain Alexandrov's proof
of the next theorem; this proof appears immediately  after
the statements of maximum principles given in Theorems~\ref{thmintmaxprin} and \ref{thmintmaxprinb}.

\begin{theorem}[Alexandrov~\cite{aa1}]
\label{thmAlex}
Round spheres are
the only closed $H$-surfaces in $\rth$.
More generally, if $\psi \colon M\to \rth$ is a closed
immersed $H$-surface  that
extends as the boundary of a compact 3-manifold which is
immersed in $\R^3$, then $\psi (M)$ is a round sphere.
\end{theorem}

In this survey we will
focus on the study of {\it complete}  $H$-surfaces
(possibly with boundary), in the sense
that all geodesics in them can  be indefinitely
extended up to the boundary of the surface. Note that with respect to the
intrinsic
Riemannian distance function between points on a surface, the property
of being ``geodesically complete'' is equivalent to the surface
being a complete metric space. A stronger global hypothesis, whose
relationship with completeness is an active field of research in
0-surface theory, is presented in the following definition.

\begin{definition}
{\rm A map $f\colon X\to Y$ between topological spaces is {\it
proper} if $f^{-1}(C)$ is compact in $X$ for any compact set
$C\subset Y$. {A subset $Y'\subset Y$ is called proper if the
inclusion map $i\colon Y'\to Y$ is proper.}
}
\end{definition}

The Gaussian curvature function $K$ of an immersed surface
$M$ in $ \R^3$ is the product of its principal curvatures,
or equivalently, the determinant of the
shape operator~$A_M$. Thus  $|K|$ is the absolute value of
the Jacobian of the Gauss map~$N\colon M\to \esf^2$. If $M$ is minimal, then
its principal curvatures are oppositely signed and thus, $K$ is
 non-positive. Therefore, after integrating
$K$ on $M$ (note that this integral may be $-\infty $ or a
 non-positive number), we obtain the same quantity as when
computing the negative of the spherical area of $M$ through its
Gauss map, counting multiplicities. This quantity is called the {\it
total curvature} of the immersed 0-surface:
\begin{equation}
\label{eq:curvtot} C(M)=\int_MK\, dA = -{\rm Area}(N\colon
M\rightarrow \esf^2).
\end{equation}

\subsection{Weierstrass representation.}
\label{subsecWeiers}
Recall that the Gauss map  of an immersed 0-surface
$M$ can be viewed as a meromorphic
function $g\colon M\to \C \cup \{ \infty \} $
on the underlying Riemann surface. Furthermore,
the harmonicity of the
third coordinate function $x_3$ of $M$ lets us define (at least locally)
its harmonic conjugate function $x_3^*$; hence, the so-called {\it
height differential\/}\;  $dh=dx_3+idx_3^*$
is a holomorphic differential on~$M$. The pair $(g,dh)$ is usually
referred to as the {\it Weierstrass data} of the immersed 0-surface,
and the 0-immersion $\psi \colon M\to \R^3$ can be expressed up to
translation by $\psi (p_0)$, $p_0\in M$, solely in terms of this data as
\begin{equation}
\label{eq:repW}
 \psi (p)=\mbox{Re}\int _{p_0}^p\left( \frac{1}{2}\left(
 \frac{1}{g}-g\right)
,\frac{i}{2}\left( \frac{1}{g}+g\right) ,1\right) dh.
\end{equation}
The pair $(g,dh)$ satisfies certain compatibility conditions, stated in assertions
{\it i), ii)} of Theorem~\ref{thm3.1}
below. The key point is that this procedure has the
following converse, which gives a cookbook-type recipe for
analytically defining any immersed 0-surface.

\begin{theorem}[Osserman~\cite{os3}]
\label{thm3.1}
Let $M$ be a Riemann surface, $g\colon M\to \C \cup
\{ \infty \} $ a meromorphic function and $dh$ a holomorphic
one-form on $M$. Assume that:
\begin{enumerate}[i)]
    \item The zeros of $dh$ coincide with the
    poles and zeros of $g$, with the same order.
    \item For any closed curve $\g \subset M$,
    \begin{equation}
\label{eq:periodproblem}
    \overline{\int _{\g }g\, dh}=\int _{\g }\frac{dh}{g},\qquad
    \mbox{\rm Re} \int _{\g }dh=0,
    \end{equation}
\end{enumerate}
where $\overline{z}$ denotes the complex conjugate of $z \in \C$. Then,
the map $\psi \colon M\to \R^3$ given by {\rm (\ref{eq:repW})} is
a conformal 0-immersion with Weierstrass data $(g,dh)$.
\end{theorem}

All local geometric invariants of an immersed 0-surface $M$ can be expressed
in terms of its Weierstrass data. For instance, the first and second
fundamental
 forms are
respectively (see~\cite{hk2,os1}):
\begin{equation}
\label{eq:I,II}
 ds^2=\left( \frac{1}{2}(|g|+|g|^{-1})|dh|\right) ^2,
  \qquad II(v,v)=\mbox{Re} \left(
\frac{dg}{g}(v)\cdot dh(v)\right) ,
\end{equation}
where $v$ is a tangent vector to $M$, and the Gaussian curvature is
\begin{equation}
\label{eq:K} K=-\left( \frac{4\left| dg/g\right| }{(|g|+
|g|^{-1})^2|dh|}\right) ^2.
\end{equation}

If $(g,dh)$ is the Weierstrass data of an immersed 0-surface $\psi \colon
M\to \R^3$, then for each $\l >0$ the pair $(\l g,dh)$ satisfies
condition {\it i)} of Theorem~\ref{thm3.1} and the second equation
in (\ref{eq:periodproblem}). The first equation in
(\ref{eq:periodproblem}) holds for this new Weierstrass data if and
only if
\[
\int _{\g }g\, dh=\int _{\g }\frac{dh}{g}=0
\]
for all homology
classes $\g $ in $M$, a condition that can be stated in terms
of the notion of flux, which we now define.
Given an immersed 0-surface $M$ with Weierstrass data $(g,dh)$, the
{\it flux vector} along a closed curve $\gamma \subset M$ is defined as
\begin{equation}
\label{eq:flux} F( \gamma) = \int_\gamma \mbox{Rot}_{90^\circ }(\g ')
= \mbox{Im} \int _{\g }\left( \frac{1}{2}\left(
\frac{1}{g}-g\right), \frac{i}{2}\left( \frac{1}{g}+g\right)
,1\right) dh\in \R^3,
\end{equation}
where $\mbox{Rot}_{90^\circ }$ denotes the rotation by angle $\pi /2$ in
the tangent plane of $M$ at any point.

\subsection{Some interesting examples of complete $H$-surfaces.}
\label{subsecexamples}

Throughout the
presentation of the  examples in this section, we will
freely use Collin's Theorem~\cite{col1} that states that
proper finite topology 0-surfaces in $\R^3$ with more than one
end have finite total curvature and  Theorem~\ref{thmCM} on the
properness of complete $H$-surfaces of finite topology; see Section~\ref{subsecends}
for further discussion of these important and deep results.

The most familiar examples of 1-surfaces in $\rth$ are spheres of radius one and cylinders
of radius $1/2$, both of which are surfaces of revolution. \vspace{.15cm}

\noindent
{\bf The Delaunay surfaces $\cD_t$, $t\in (0,\frac{\pi}{2}]$}.
In 1841,  Delaunay~\cite{de1} classified the immersed 1-surfaces of revolution in $\rth$.
We will call the embedded ones (unduloids)
{\it Delaunay surfaces,} see Figure~\ref{DelaunayFigure}.
The $\cD _t$, $t\in (0,\frac{\pi}{2}] $, form a
one-parameter family of proper 1-surfaces of
revolution that are invariant under a translation
along the  revolution axis. $\cD _t$ is a cylinder for $t=\frac{\pi }{2}$, ,
whereas with $t\to 0$, then $\cD _t$ converges
to a chain of tangential spheres with radius
$1$. In fact, the parameter
$t\in (0,\frac{\pi}{2}] $ can be viewed as the length
of the CMC flux vector of $\cD_t$,
computed on any of its circles; see Definition~\ref{def:flux} below where
the CMC flux is defined.
\begin{figure}
\begin{center}
\includegraphics[width=2.4in]{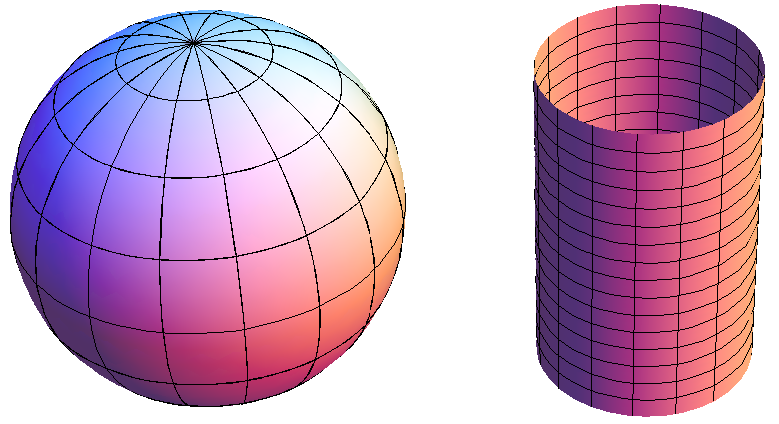}
\caption{A sphere and a cylinder.} 
 \end{center}
\end{figure}
\begin{figure}
\begin{center}
\includegraphics[width=4.5cm]{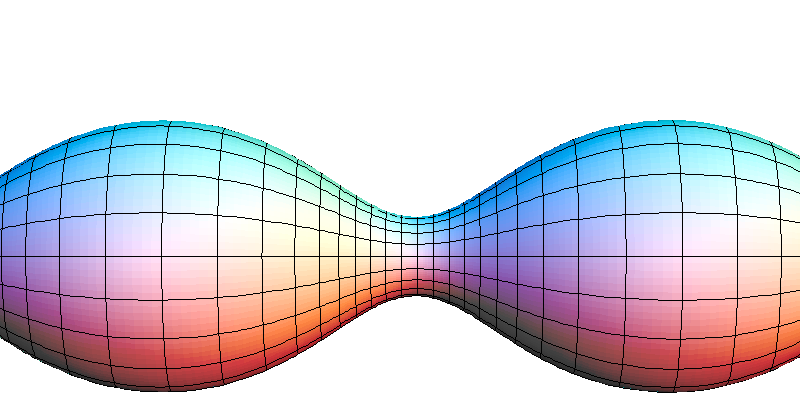}
\qquad \includegraphics[width=1.8in]{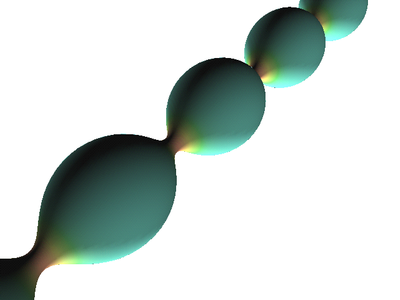}
\caption{Two different Delaunay surfaces.} \label{DelaunayFigure}
 \end{center}
\end{figure}
\vspace{.2cm}

We will next use the Weierstrass representation for introducing
some of the most celebrated complete immersed 0-surfaces.
\vspace{.15cm}

\noindent
{\bf The catenoid.} $M=\C -\{ 0\} $, $g(z)=z$,
 $dh=\frac{dz}{z}$, see Figure~\ref{cat-hel-Enn} Left.
 In 1741, Euler~\cite{eul} discovered that when a catenary $x_1=\cosh x_3$
is rotated around the $x_3$-axis, one obtains a surface
which minimizes area among surfaces of revolution
after prescribing boundary values for the generating curves. This surface was
called the {\it alysseid} or since Plateau's time,
the catenoid. In 1776, Meusnier verified that
the catenoid is locally a solution of Lagrange's equation,
which just means that it locally minimizes
area relative to local boundaries. This surface has
genus zero, two ends and total curvature $-4\pi $.
Together with the
plane, the catenoid is the only 0-surface of revolution
(Bonnet~\cite{Bonnet2}) and the unique complete 0-surface with
genus zero, finite topology and more than one end
(L\' opez and Ros~\cite{lor1}). Also,
the catenoid is  characterized as being the unique complete 0-surface with
finite topology and two ends
(Schoen~\cite{sc1}).
\begin{figure}
\begin{center}
  \includegraphics[width=8.6cm]{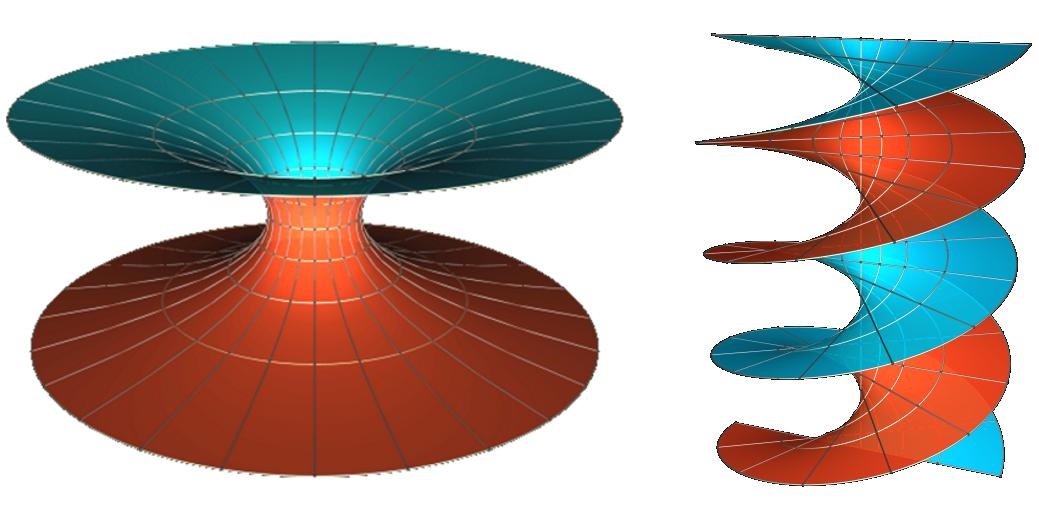}\\
  \caption{Left: The catenoid. Right: The helicoid.}
 \label{cat-hel-Enn}
 \end{center}
\end{figure}

\par
\vspace{.2cm}
\noindent
{\bf The helicoid.} $M=\C $, $g(z)=e^z$,
$dh=i\, dz$, see Figure~\ref{cat-hel-Enn} Right.
This surface was first proved to be minimal by Meusnier in 1776~\cite{meu1}.
When viewed in $\R^3$, the helicoid has genus zero, one end and
infinite total curvature. Together with the plane, the helicoid is
the only ruled 0-surface (Catalan~\cite{catalan1}) and the
unique simply connected, complete  0-surface (Meeks
and Rosenberg~\cite{mr8}, see also~\cite{bb1}). The vertical helicoid
can also
be viewed as a genus-zero surface with two ends in
a quotient of $\R^3$ by a
vertical translation or by a screw motion. The catenoid
and the helicoid are {\it
conjugate} 0-surfaces, in the sense of the following definition.

\begin{definition}
\label{defconjugate}
{\rm Two immersed 0-surfaces in $\R^3$ are said to be {\it conjugate} if
the coordinate functions of one of them are locally the harmonic conjugates
of the coordinate functions of the other one.
}
\end{definition}

\begin{remark}
{\rm
  There is also a notion of conjugate surface for $(H>0)$-surfaces;
see~\cite{mt2} for further discussion on the more general notion
of  associate surfaces to an $H$-surface.
}
\end{remark}
Note that in the case of the helicoid and catenoid, we consider
the catenoid to be
defined on its universal cover $e^z \colon \mathbb{C} \rightarrow
\mathbb{C}-\{0\}$ in order for the harmonic conjugate of $x_3$ to be
well-defined.  Equivalently, both surfaces share the Gauss map $e^z$
and their height differentials differ by multiplication by
$i=\sqrt{-1}$.

\vspace{.15cm}
\par
\noindent
{\bf The Meeks minimal M\"obius strip.} $M=\C - \{ 0 \}$,  $g(z)
= z^2 \left( \frac{z+1}{z-1} \right)$, $dh = i \left( \frac{z^2
-1}{z^2}\right) \, dz$, see  Figure~\ref{figmobius-benthel} Left.
Found by Meeks~\cite{me7}, the
0-surface defined by this Weierstrass data double covers a
complete, immersed 0-surface $M_1\subset \R^3$ which is
topologically a M\"obius strip. This is the unique complete,
minimally immersed surface in $\R^3$ of finite total curvature $-6 \pi$. It
contains a unique closed geodesic which is a planar circle, and also
contains a line bisecting the circle.
\vspace{.15cm}
\begin{figure}
\begin{center}
  \includegraphics[width=10.5cm]{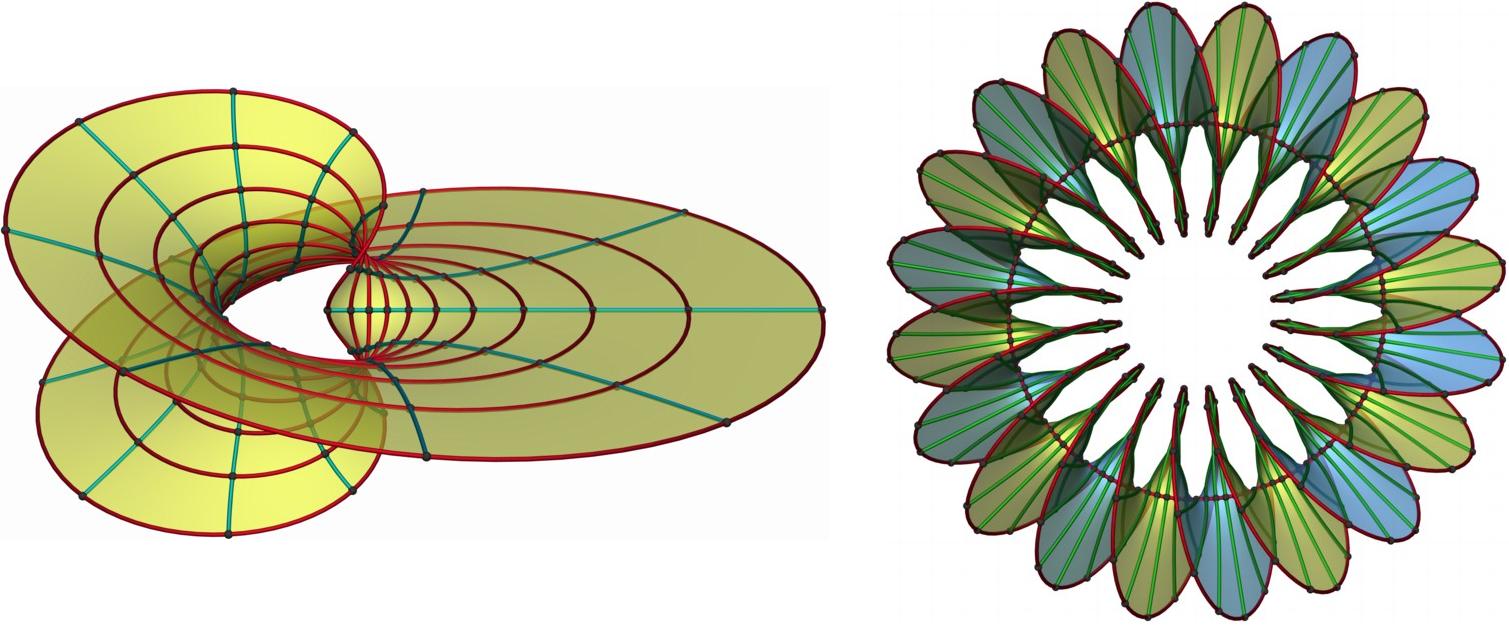}\\
  \caption{Left: The Meeks minimal M\"obius strip.
  Right: A bent helicoid near the
 circle $\esf ^1$, which is viewed from above along the $x_3$-axis.
 Images courtesy of M. Weber.}
\label{figmobius-benthel}
\end{center}
\end{figure}

\par
\noindent
{\bf The bent helicoids.}
$M=\C -\{ 0\} $, $g(z)=-z \frac{z^n+i}{iz^n+i}$, $dh = \frac{z^{n}+z^{-n}}{2z}dz$,
see  Figure~\ref{figmobius-benthel} Right. Discovered by
Meeks and Weber~\cite{mwe1} and independently by Mira~\cite{mira1}, these are complete,
immersed 0-annuli
$\widetilde{A}_n\subset \R^3$
with two non-embedded ends and finite total curvature; each
of the surfaces $\widetilde{A}_n$ contains the unit circle $\esf^1$ in the
$(x_1, x_2)$-plane, and a neighborhood of $\esf^1$ in $\widetilde{A}_n$ contains an embedded
annulus $A_n$ which  approximates, for $n$ large, a highly spinning helicoid whose usual
straight axis has been periodically bent into the
unit circle $\esf ^1$ (thus the name of bent helicoids). Furthermore,
the $A_n$ converge as $n\to \infty $ to the foliation of $\R^3$ minus the $x_3$-axis
by vertical half-planes
with boundary the $x_3$-axis, and with $\esf^1$ as the singular set of
$C^1$-convergence. The method applied by
Meeks, Weber and Mira to find the bent helicoids is the
classical {\it Bj\" orling formula}~\cite{ni2} with an orthogonal
 unit field along $\esf^1$ that spins an arbitrary number $n$ of times around the circle.
This construction also makes sense
when $n$ is half an integer; in the case $n=\frac{1}{2}$, $\widetilde{A}_{1/2}$
is the double cover of the Meeks minimal
M\"{o}bius strip described in the previous example. The bent helicoids $A_n$ play
an important role in proving the
converse of Meeks' $C^{1,1}$-Regularity Theorem (see Meeks and Weber~\cite{mwe1}
and also Theorems~\ref{thmregular} and~\ref{MeeksWeber} below)
for the singular set of convergence in a Colding-Minicozzi limit 0-lamination.

\vspace{.15cm}
\par
\noindent{\bf The singly-periodic Scherk surfaces.} $M=(\C \cup \{
\infty \} ) -\{ \pm e^{\pm i\t /2}\} $, $g(z)=z$, $dh=\frac{iz\,
 dz}{\prod (z\pm e^{\pm i\t /2})}$, for
fixed $\t \in (0,\pi /2]$, see Figure~\ref{scherk-figure} Left
for the case $\t =\pi /2$.
Discovered by Scherk~\cite{sche1} in 1835, these surfaces
denoted by ${\mathcal S}_\theta$ form a
$1$-parameter family of complete genus-zero 0-surfaces in a quotient of $\R^3$
by a translation, and have four annular ends. Viewed in $\R^3$,
each surface ${\mathcal S}_{\theta }$ is invariant
under reflection in the $(x_1,x_3)$ and $(x_2,x_3)$-planes and in
horizontal planes at integer heights, and
can be thought of geometrically as a desingularization of two vertical
planes forming
an angle of $\t $. The special case ${\mathcal S}_{\theta =\pi /2}$ also
contains pairs of orthogonal lines
at planes of half-integer heights, and has implicit equation $\sin z=\sinh x\sinh y$.
Together with the plane and catenoid, the surfaces
${\mathcal S}_{\t }$ are conjectured to be the only
connected, complete,
immersed, 0-surfaces in $\rth$ whose areas in balls of radius $R$ is less than $2
\pi R^2$; see Conjecture~\ref{conjScherk} in Section~\ref{sec:conj} for further discussion
on this open problem.  This conjecture was proved by Meeks and Wolf~\cite{mrw1}
under the additional hypothesis that the surface have an infinite symmetry group.
\begin{figure}
\begin{center}
  \includegraphics[width=7.6cm]{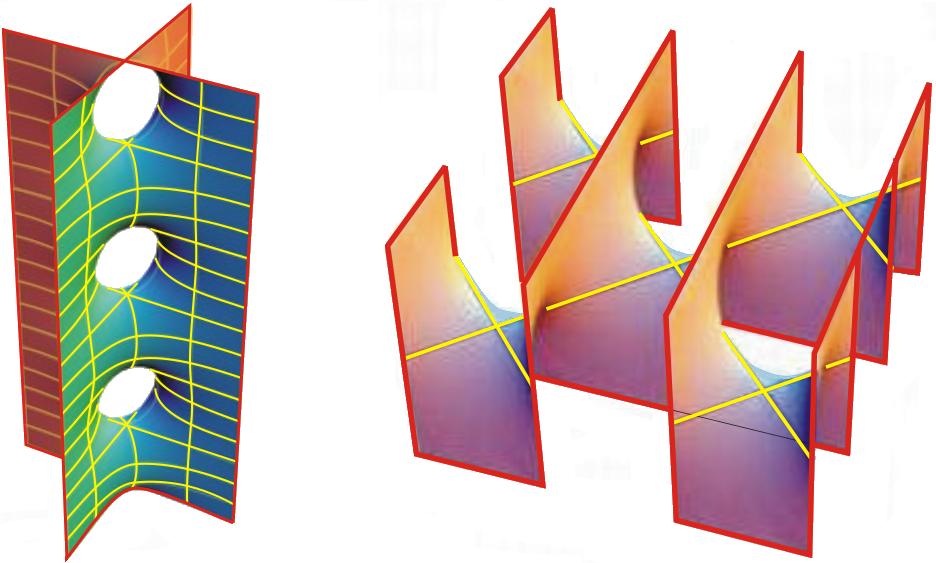}\\
  \caption{Singly-periodic Scherk surface with angle $\t =
\pi/2$ (left), and its conjugate surface,
the doubly-periodic Scherk surface (right). Images courtesy of M. Weber.}
\end{center}
\label{scherk-figure}
\end{figure}

\vspace{.15cm}
\par
\noindent
{\bf The doubly-periodic Scherk surfaces.}
$M=(\C \cup \{ \infty \} ) -\{
\pm e^{\pm i\t /2}\} $, $g(z)=z$, $dh=\frac{z\, dz}{\prod (z\pm e^{\pm
i\t /2})}$, where $\t \in (0,\pi /2]$ (the case $\t =\frac{\pi }{2}$ has implicit equation
$e^z\cos y=\cos x$), see Figure~\ref{scherk-figure} Right.
These surfaces, discovered by Scherk~\cite{sche1} in 1835, are the conjugate
surfaces of the singly-periodic Scherk
surfaces, and can be thought of geometrically as the
 desingularization of two families of equally
spaced vertical parallel half-planes in opposite half-spaces,
 with the half-planes in
the upper family making an angle of $\t $ with the half-planes in
the lower family. These surfaces are doubly-periodic with genus zero
in their corresponding quotient
$\Te ^2\times \R $, and were characterized by Lazard-Holly and
Meeks~\cite{lm2} as being the unique proper 0-surfaces with genus zero in
any $\Te ^2\times \R $. It has been conjectured by Meeks,
P\'erez and Ros (see Conjecture~\ref{conjFP})
that the singly and doubly-periodic
Scherk 0-surfaces are the only complete 0-surfaces in $\rth$ whose Gauss maps miss
four points on $\esf^2$. They also conjecture that the singly and
doubly-periodic Scherk 0-surfaces, together with the catenoid
and helicoid, are the only complete 0-surfaces of
negative curvature (see Conjecture~{\ref{NCC}}).

\vspace{.15cm}
\par
\noindent
{\bf The Riemann minimal examples.} These surfaces come in a
one-parameter family defined in terms of a parameter $\l >0$. Let $M_\lambda =\{ (z,w)\in
(\C\cup \{ \infty \}
)^2\ | \ w^2=z(z-\lambda )(\lambda z+1)\} -\{ (0,0),(\infty ,\infty
)\} $, $g(z,w)=z$, $dh=A_{\lambda }\frac{dz}{w}$, for each $\l >0$,
where $A_{\lambda }$ is a non-zero complex number satisfying
$A_{\lambda }^2\in \R $, see Figure~\ref{Riemann-fig}.
Discovered by Riemann (and posthumously published, Hattendorf and
Riemann~\cite{ri2,ri1}), these examples are invariant under reflection in the
$(x_1,x_3)$-plane and by a
translation $T_{\l }$. The induced
surfaces $M_{\l }/T_{\l }$  in the quotient spaces $\R^3/T_{\l }$
have genus one and two planar ends, see~\cite{mpr6} for a more precise description.
The Riemann minimal examples have the amazing property that every
horizontal plane intersects each of these surfaces in a circle or in a line.
The conjugate minimal surface of the Riemann minimal example for a
given $\l >0$ is the Riemann minimal example for the parameter value
$1/\l $ (the case $\l =1$ gives the only self-conjugate surface in
the family). Meeks, P\'erez and Ros~\cite{mpr6} showed that these surfaces
are the only proper 0-surfaces in $\rth$ of genus
zero and infinite topology. Assuming that Conjecture~\ref{CYconj1} below  holds,
then  these surfaces
are the only complete 0-surfaces in $\rth$ of genus
zero and infinite topology. Also  see~\cite{mpe16} for a complete outline of the
proof of uniqueness of the Riemann minimal examples and historical comments
on Riemann's original proof of the classification of his examples. \vspace{.15cm}
\begin{figure}
\begin{center}
  \includegraphics[width=5.3cm]{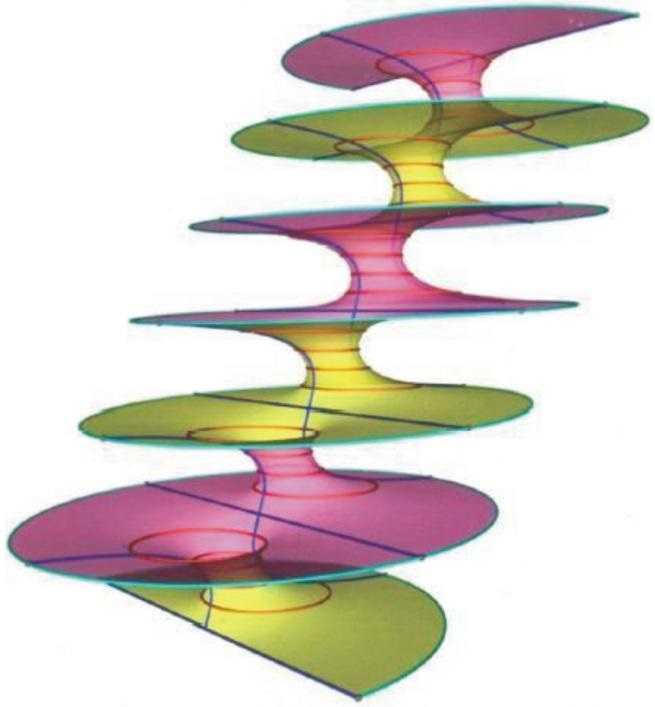}\\
  \caption{A Riemann minimal example.
 Image courtesy of M. Weber.}
\label{Riemann-fig}
\end{center}
\end{figure}

\subsection{Classical maximum principles.}

One of the consequences of the fact that $H$-surfaces can be
viewed locally as solutions of a partial differential
equation is that they satisfy certain
maximum principles. We will state them for $H$-surfaces in $\R^3$,
but they also hold when the ambient space is any Riemannian
3-manifold.

\begin{theorem}[Interior Maximum Principle~\cite{gt1}]
\label{thmintmaxprin}
For $i=1,2$, let $M_i$ be a connected $H_i$-surface in $\R^3$, and
$p$ an interior point to both surfaces. Suppose that $M_2$ lies
on the mean convex side of $M_1$
{near $p$}
$M_1=M_2$ in a neighborhood of~$p$.
\end{theorem}

\begin{theorem}[Boundary Maximum Principle~\cite{gt1}]
\label{thmintmaxprinb}
For $i=1,2$, let $M_i$ be a connected $H_i$-surface with boundary in $\R^3$, and
$p$ a boundary point of both surfaces. Suppose
that $M_2$ lies on the mean convex side of $M_1$ near $p$,
and that the surfaces are locally tangent graphs over the same half disk in
$T_pM_1=T_pM_2$ with tangent boundaries and with the same normal at $p$.
If $H_1\geq H_2$, then $M_1=M_2$ in a neighborhood of~$p$.
\end{theorem}
Before continuing with the exposition and as we announced just before
the statement of Theorem~\ref{thmAlex}, we pause to give a proof of the following
classical result as an application of the previous two theorems.

\begin{theorem}[Alexandrov~\cite{aa1}] Round spheres are
the only closed $H$-surfaces in $\rth$.
\end{theorem}
\begin{proof}
Let $M$ be a closed $H$-surface and let $W$ be the smooth
compact domain in $\rth$ with boundary $M$.
We now explain how to use the Alexandrov reflection principle to prove that for
the family of horizontal planes $\{ P(t)=\{x_3=t \} \} _{t\in \R}$, there
exists a $t_M\in \R$ such that $P(t_{M})$ that is a plane of reflectional symmetry
for $M$ and furthermore  $M-P(t_{M})$ consists of two components, each of which is
a graph over a bounded component of $P(t_{M})-M$.  Assuming this symmetry result we can deduce
that for any unit length vector $a$,  $M$ has a plane of reflective symmetry
with normal vector $a$ and, after a translation and by the compactness of $M$,
$M$ is invariant under the action of the orthogonal group $O(3)$, which implies
that $M$ is a round sphere.

For each $t\in\R$, let $R_t\colon\rth \to \rth$ be reflection in the plane $P(t)$
and consider the closed lower half-space $P(t)^-=\{(x_1,x_2,x_3)\in \rth \mid x_3\leq t\}$
determined by $P(t)$.
Consider the smallest $t_1$ such that $P(t_1)$ intersects $M$.  Then, there exists a small
$\ve>0$ such that the following hold:
\begin{enumerate}[1.]
\item $P(t_1+\ve)^- \cap M$ is a graph over its possibly disconnected projection to
$ P(t_1+\ve)$.
\item $R_{t_1+\ve}(P(t_1+\ve)^-\cap M)\subset W$.
\end{enumerate}
Define $\ve_1=\max \{ \ve '\in [\ve ,\infty )\ | \ \mbox{
items 1 and 2   hold for }\ve '\} $, which exists by compactness of~$M$.
\par
\vspace{.2cm}
{\sc Claim:} {\it The plane $P(t_1+\ve_1)$ is a plane of Alexandrov symmetry for $M$.}
\par
\vspace{.2cm}
\noindent
Observe that the above claim proves the desired
symmetry result for $M$ stated in the first paragraph of this proof,
in other words, $t_M=t_1+\ve_1$. Theorem~\ref{thmintmaxprinb} implies that the claim holds provided
that the plane $P(t_1+\ve)$
is orthogonal to $M$ at some point $p$. So assume now that the plane $P(t_1+\ve)$
is nowhere orthogonal to $M$. We claim that Theorem~\ref{thmintmaxprin}
implies
\begin{equation}
\label{alex1}
R_{t_1+\ve_1}(P(t_1+\ve_1)^- \cap M)-P(t_1+\ve_1)\subset \Int(W).
\end{equation}
Otherwise, since at a point of intersection of
$R_{t_1+\ve_1}(P(t_1+\ve_1)^- \cap M)-P(t_1+\ve_1)$ with
$M=\partial W$, these surfaces have the same normals, then
the interior maximum principle shows that $R_{t_1+\ve_1}(P(t_1+\ve_1)^- \cap M)\subset M$,
which would  imply that
$P(t_1+\ve_1)$ is a plane of symmetry and hence perpendicular
to $M$ at every point of $M\cap P(t_1+\ve _1)$, {which is
contrary to our hypothesis}.
{Now (\ref{alex1}) and the compactness of $M$ ensure that}
for $\delta>0$ sufficiently  small,
$$R_{t_1+\ve_1 +\delta}(P(t_1+\ve_1+\delta)^- \cap M)-P(t_1+\ve_1+\delta)\subset \Int(W),$$
from which we conclude that there exists a $\delta'\in (0,\de )$ such that ${\ve _1}+\delta'$
satisfies
items 1 and 2, which contradicts the definition of $\ve_1$.
This contradiction completes the proof.
\end{proof}

Another beautiful application
of Theorem~\ref{thmintmaxprin} is the following result by Hoffman
and Meeks.

\begin{theorem}[Strong Half-space Theorem~\cite{hm10}]
\label{thmhalf}
 Let $f\colon M\to \R^3$ be a properly immersed, possibly branched, non-planar 0-surface
without boundary. Then, $M$ cannot be contained in a half-space.
More generally, if $M_1,M_2\subset \R^3$ are the images of two properly immersed
0-surfaces in $\rth$, one of which is non-planar, then these surfaces intersect.
\end{theorem}

Here the adjective ``strong" refers to the second statement of Theorem~\ref{thmhalf}.
This second statement follows from the first statement by a previous result of Meeks, Simon
and Yau~\cite{msy1}, where they proved that given two proper, possibly branched
0-surfaces $M_1,M_2$ in $\rth$ that are disjoint, there exists
a proper, stable orientable 0-surface $\Sigma$ in the region of $W$ of $\rth$
between the images of $M_1,M_2$ (see Definition~\ref{defstable} for the notion of stability).
Since $\Sigma$ is stable, then it is a plane by
Theorem~\ref{thmstablecompleteplane} below.
The original proof by Hoffman and Meeks of the first statement of Theorem~\ref{thmhalf}
only uses the interior maximum principle and
a clever argument with catenoids as barriers. Since this  argument is simple and has become
a standard and useful technique for other applications,
we will also include it here for the sake of completeness.

For the following arguments, please refer to {Figure~\ref{halfSP}.}
The first step is to find a smallest open half-space that contains the surface $M$,
which can be assumed to be $\{ z<0\} $. As $M$ is
proper, no points of $\{ z=0\} $ are
accumulation points of $M$. Applying this property to $\vec{0}=(0,0,0)$, one finds a ball
$\B (r)$ centered at $\vec{0}$ of radius $r>0$ which
is disjoint from $M$. Given $a\in (0,r)$,
the vertical catenoid $C_a=\{ x^2+y^2=a^2\cosh ^2(z/a)\} $
has waist circle contained in $\B (r)$
and thus, the upper half-catenoid $C_a^+=C_a\cap \{ x_3\geq 0\} $ can be lowered
some small height $\ve >0$ so that $C:=C_a^+-\varepsilon (0,0,1)$ is contained in
$\B (r)\cup \{ x_3\geq 0\} $. Next one considers the
1-parameter family of half-catenoids
$\{ C(t)\ | \ t>0\} $ obtained after applying to $C$ a
homothety of ratio $t>0$ centered at the center
$O=(0,0,-\ve )$ of $C$.
\begin{figure}
\includegraphics[width=5.8in]{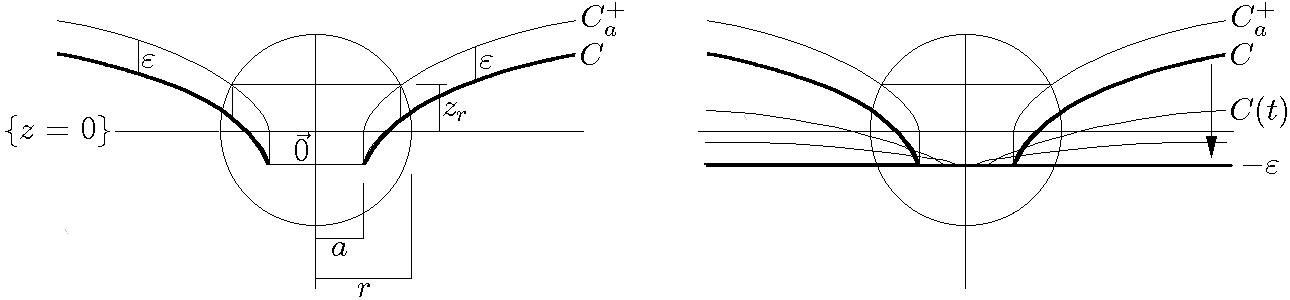}
\vspace{-.2cm}\caption{The $ C(t)$, $t\in (0,1]$, are
homothetic shrinkings of $C$.}
\label{halfSP}
\end{figure}
As $C(t)$ converges as $t\to 0$ to the punctured plane $\{ x_3=-\ve \} -\{ O\} $,
then $M$ must have points above $C(T)$ for some $T>0$
sufficiently small. As $C(1)$ is disjoint from $M$,
then there exists the infimum $t_1$ of the
set $\{ t>0 \ | \ M\cap C(t)=\mbox{\O }\} $. As $M\subset
\{ x_3<0\} $ and the end of $C(t_1)$ is
catenoidal with positive logarithmic growth, then $M$ and
$C(t_1)$ intersect at a common interior point $p$ (note that $\partial C(t_1)\subset
\B (r)$) where $M$ lies at one side of $C(t_1)$ near $p$,
which contradicts the {interior}
maximum principle. \vspace{.2cm}

More generally, one has the following result of Meeks and Rosenberg
based on earlier partial results in~\cite{cmw1,hm10,lr1,mr1,sor1}.

\begin{theorem}[Maximum Principle at Infinity~\cite{mr7}]
\label{thmMaxprin}
Let $M_1,M_2\subset N$ be disjoint, connected,
properly immersed 0-surfaces
with (possibly empty) boundary in a complete flat
3-manifold $N$.
\begin{enumerate}[1.]
\item If $\partial M_1\neq \mbox{\rm \O }$ or
$\partial M_2\neq \mbox{\rm \O }$, then after possibly reindexing,
the distance between $M_1$ and $M_2$ (as subsets of $N$) is equal to
$\inf \{ {d_N}
(p, q) \mid p \in \partial M_1, \, q \in M_2 \} $,
{where $d_N$ denotes Riemannian distance in $N$.}
\item If $\partial M_1=\partial M_2=\mbox{\rm \O}$,
 then $M_1$ and $M_2$
are flat.
\end{enumerate}
\end{theorem}

We now come to a  deep application of the general maximum
principle at infinity.  The next corollary appears in~\cite{mr7} and
a slightly weaker variant of it can be found in Soret~\cite{sor1} when $H=0$.
Actually the results in~\cite{mr7} describe a slightly weaker version of the last statement in
the corollary when $H>0$, but the stronger statement
given below is easily proven with the same methods.

\begin{corollary}[Regular Neighborhood Theorem]
\label{cor.max}
Suppose $M \subset N$ is a proper non-flat $0$-surface in a
complete flat 3-manifold $N$,
with bounded second fundamental form and let $1/R$ {be}
the supremum of the absolute values of the principal curvatures of $M$.  Let
$N_{R}(M)$ be the open subset of the normal bundle of $M$ given by the
normal vectors of length strictly less that $R$. Then,
the corresponding exponential map $\exp \colon N_{R}(M) \rightarrow N$
is a smooth embedding.  In particular:
\begin{enumerate}[1.]
\item $M$ is properly embedded.
\item $M$ has an open, embedded tubular neighborhood of radius~$R$.
\item There exists a constant $C>0$, depending only on $R$ such that
for all balls $B\subset N$ of radius 1,
the area of $M\cap B$ is at most $C$ times the volume of $B$.
\end{enumerate}
Furthermore, under the same hypotheses on $M$ except that it is an $(H>0)$-surface,
and letting $N_{R}^+(M)\subset N_{R}(M)$ be the subset
of normal vectors that have a non-negative inner
product with the mean curvature vector of $M$, then the
restriction $\exp \colon N_{R}^+(M) \rightarrow N$ is a smooth embedding.
In particular, property 2 holds on the mean convex side of $M$,
and thus properties 1 and 3 also hold for $M$.
\end{corollary}
%

\begin{definition}  \label{Strong-Alex}{\rm
Let $N$ be a smooth Riemannian $n$-manifold.
\ben
\item We call a compact, immersed
$H$-hypersurface $f\colon M\to N$
{\em Alexandrov embedded} if there exists an immersion
$F\colon W\to N$ of a compact, mean convex $n$-manifold $W$ with $\partial W = M$,
such that  $F|_M=f$.
\item
We call a proper, immersed
$H$-hypersurface $f\colon M\to N$
{\em strongly Alexandrov embedded} if there exists a proper immersion
$F\colon W\to N$ of a complete, mean convex $n$-manifold $W$ with $\partial W = M$,
such that $F$ is injective on the interior of $W$ and $F|_M=f$.
\een}
\end{definition}

We next include a result which is analogous to Corollary~\ref{cor.max}
and that holds in the $n$-dimensional setting for certain $(H>0)$-hypersurfaces.

\begin{theorem}[One-sided Regular Neighborhood, Meeks, Tinaglia\, \cite{mt3}]
\label{cor*}  \mbox{}\newline
Suppose $N$ is a complete $n$-manifold with absolute
sectional curvature bounded by a constant $S_0>0$. Let $M$ be a strongly
Alexandrov embedded hypersurface with constant mean
curvature $H_0 >0$ and norm of its second fundamental form $|A_M|\leq A_0$ for some $A_0>0$.
Then, the following statements hold.
\begin{enumerate}[1.]
\item There exists a positive number $\tau\in\left(0,\pi/S_0\right)$,
depending on $A_0$, $H_0$, $S_0,$ such
that $M$ has a regular neighborhood $\exp(N^+_{\tau }(M))$
of width $\tau$ on its mean convex side, where we are using
the notation of Corollary~\ref{cor.max}.
\item There exists $C>0$ depending on $A_0$, $H_0$, $S_0,$
such that  the $(n-1)$-dimensional volume of $M$ in balls of
radius $1$ in $N$ is less than $C$. \end{enumerate}
\end{theorem}

\subsection{Second variation of area, index of stability, Jacobi functions and
curvature estimates of stable $H$-surfaces.}
\label{subsecJacobi}
Let $\psi \colon M\to N$ be an isometric immersion of a surface in a Riemannian 3-manifold~$N$.
Assume that $\psi (M)$ is two-sided, i.e. there exists a globally defined unit normal vector
field $\eta$ on $M$. Given a compact smooth domain (possibly with
boundary) $\Omega \subset M$, we will consider variations of $\Omega
$ given by differentiable maps $\Psi \colon
(-\varepsilon,\varepsilon)\times \Omega \rightarrow N$,
$\varepsilon>0$, such that $\Psi (0,p)=\psi (p)$ and $\Psi (t,p)=\psi (p)$ for
$|t|<\varepsilon $ and $p\in M-\Omega $. The {\it variational
vector field} for such a variation $\Psi $ is $\left.
\frac{\partial \Psi }{\partial t}\right| _{t=0}$ and its normal
component is
$u=\langle \left. \frac{\partial \Psi }{\partial t}\right|_{t=0},\eta\rangle$.
Note that, for small $t$, the map
$\psi _t=\Psi _{|t\times \Omega }$ is an immersion. Hence we can associate
to $\Psi $ the {\it area function} Area$(t)=$Area$(\psi_t)$ and the {\it
volume function} Vol$(t)$  given by
\[
{\rm Vol}(t) =\int_{[0,t]\times \Omega } \Jac (\Psi ) \, dV,
\]
where $dV$ is the volume element in $N$. The function Vol$(t)$
measures the signed volume enclosed between $\psi _0=\psi$ and $\psi_t$.

The first variation formula for the area and volume are
\begin{equation}
\label{1varvolume}
\left. \frac{d}{dt}\right| _{t=0}\mbox{\rm Area}(t)=-2\int_M Hu\,
dA,\qquad \left. \frac{d}{dt}\right| _{t=0}\mbox{Vol}(t)= -\int _M
u\, dA,
\end{equation}
where $dA$ is the area element on $M$ for the induced metric by $\psi $. The equations
in (\ref{1varvolume}) imply that $M$ is a critical point
of the functional $\mbox{Area}-2c\, \mbox{Vol}$ (here $c\in \R $) if and only if it
has constant mean curvature $H=c$. In this case, we can consider the {\it Jacobi
operator} on $M$,
\begin{equation}
\label{Jacobiop}
L= \Delta +|A_M|^2+{\rm Ric}(\eta),
\end{equation}
where Ric$(\eta)$
is the Ricci curvature of $N$ along the unit normal vector field of
the immersion. For an $H$-surface $M$, the second variation formula
of the functional $\mbox{Area}-2H\, \mbox{Vol}$ is given by
(see e.g.~\cite{bce1,ni2})
\begin{equation}
\label{stable}
\hspace{-.45cm}
\left.
\frac{d^2}{dt^2}\right| _{t=0}\left[
 \mbox{\rm Area}(t)-2H\ \mbox{Vol}(t)\right] =
 -\int _M uLu\, dA =
 \int_ M \left[ |\nabla u|^2 -  (|A_M|^2+{\rm Ric}(\eta))u^2\right] dA.
\end{equation}
Formula (\ref{stable}) can be viewed as the
bilinear form $Q(u,u)$ associated to the linear elliptic $L^2$-selfadjoint operator
given by the Jacobi operator $L$ defined in (\ref{Jacobiop}).

\begin{remark}{\em
For a normal variation $\psi_t$ of a surface $\psi\colon M\to N$ with associated
normal variational vector field  $u\eta$, $L(u)(p)$ is
equal to $-2H'(t)|_{t=0}$ at $p$, where $H(t)(p)$ is the mean curvature of
the immersed surface $\psi_t (M)$ at the point $\psi_t(p)$.
}
\end{remark}

\begin{definition}
\label{defJacobif}
{\rm A $C^2$-function $u\colon M\to \R $ satisfying $Lu=0$
on $M$ is called a {\it Jacobi function.} We will let ${\mathcal
J}(M)$ denote  the linear space of Jacobi functions on~$M$.}
\end{definition}
Classical elliptic theory implies that given a subdomain
$\Omega \subset M$ with
compact closure, the Dirichlet problem for the Jacobi operator
in $\Omega $ has an
infinite discrete spectrum $\{ \l _k\} _{k\in \N \cup \{ 0\} }$
of eigenvalues with
$\l _k\nearrow +\infty $ as $k$ goes to infinity, and each
eigenspace is a finite
dimensional linear subspace of $C^{\infty }(\Omega )\cap
H_0^1(\Omega )$, where $H_0^1(\Omega )$ denotes the usual Sobolev
space of $L^2$-functions with $L^2$ weak partial derivatives and
trace zero.

\begin{definition} \label{def2.21}
{\rm
 Let $\Omega \subset M$ be a subdomain with compact closure.
  The {\it index of
stability} of $\Omega $ is
the number of negative eigenvalues of the Dirichlet problem
 associated to $L$ in
$\Omega $. The {\it nullity} of $\Omega $ is the dimension
of ${\mathcal J}(\Omega )\cap
H_0^1(\Omega )$. $\Omega $ is called {\it stable} if its index of stability
is zero, and {\it
strictly stable} if both its index and nullity are zero. }
\end{definition}

When $N=\R^3$, (\ref{Jacobiop}) reduces to $L=\Delta -2K$
{($K$ denotes Gaussian curvature)}.
In this case, since the
Gauss map of {an $H$}-graph
defined on a domain in a plane $\Pi $ has
image set contained in an open half-sphere, the inner product of the unit normal vector
with the unit normal to $\Pi $ provides a positive Jacobi function,
from where we conclude that any {$H$}-graph is stable.

Coming back to the general case of a
two-sided $H$-surface $\psi \colon M\to N$ in a Riemannian 3-manifold~$N$,
stability also makes sense in the non-compact setting for $M$, as we next explain.
\begin{definition}
\label{defstable}
{\rm
An $H$-surface $\psi \colon M\to N$ in a
Riemannian 3-manifold~$N$ is called {\it stable}
if any subdomain $\Omega \subset M$ with compact closure
is stable in the sense of Definition~\ref{def2.21}.
Stability is equivalent to the existence of a positive Jacobi function on $M$
(Proposition~1 in Fischer-Colbrie~\cite{fi1}). $M$ is said to have {\it
finite index} if outside of a compact subset it is stable. The {\it
index of stability} of $M$ is the supremum of the indices of stability of
subdomains with compact closure in~$M$.}
\end{definition}

For $H$-surfaces, it is natural to consider
a weaker notion of stability, associated to the isoperimetric
problem.
\begin{definition}
\label{defwstable}
{\rm
We say that an $H$-surface $\psi \colon M\to N$ in a
Riemannian 3-manifold~$N$ is {\it weakly
stable} if
\[
\int_ M \left[ |\nabla u|^2 -  (|A_M|^2+{\rm Ric}(\eta))u^2\right] dA\geq 0,
\]
for every $f\in C^{\infty }_0(M)$ with $\int _M f \, dA=0$.
Sometimes this notion is referred to as
{\it volume preserving stable} in the literature.}
\end{definition}

The Gauss equation allows us to write the Jacobi operator of an $H$-surface in several
interesting forms.
\begin{eqnarray}
L &=& \Delta - 2K +4H^2+\Ric (e_1)+\Ric (e_2) \label{jacobi1} \\
& =&\rule{0cm}{.5cm} \Delta  -K + 2H^2+\frac{1}{2}|A_M |^2+
\frac{1}{2}S \label{jacobi2}\\
&=&\rule{0cm}{.5cm}\Delta  -K+3H^2+\frac{1}{2}S+(H^2-\det (A_M)),
\label{jacobi3}
\end{eqnarray}
where $K$ is the Gaussian curvature of $M$, $e_1,e_2$ is
an orthonormal basis of the tangent plane
of $\psi \colon M\to N$ and $S$ denotes the scalar
curvature of $N$. Note that we take the scalar curvature
function $S$ at a point $p\in N$ to be six times
the average sectional curvature of $N$ at $p$.

By definition, stable surfaces have index zero. The following
theorem explains how restrictive is the property of stability for
complete $H$-surfaces in $\R^3$. In the case $H=0$, the first
statement in it was proved independently by
Fischer-Colbrie and Schoen~\cite{fs1}, do Carmo and Peng~\cite{cp1},
and Pogorelov~\cite{po1} for orientable surfaces. Later,
Ros~\cite{ros9} proved that a complete, non-orientable 0-surface in $\rth$ is never stable.
The second statement has important applications to
the study of regularity properties of $H$-laminations in $\rth$ punctured at the origin.
A short elementary proof of the next result is given in Lemma~6.4 of~\cite{mpr19}.
The case $H=0$ of the second statement in the following result was also obtained
by Colding and Minicozzi (Lemma A.26 in~\cite{cm25}).

\begin{theorem}
\label{thmstablecompleteplane}
If $M\subset\R^3$ is a complete, stable immersed $H$-surface, then $M$ is a plane.
More generally, if $M\subset \rth -\{\vec{0}\}$ is a stable $H$-surface
which is complete outside the origin (in the sense that every divergent path
in $M$ of finite length has as limit point the origin), then $M$ is a plane.
\end{theorem}

A crucial fact in $H$-surface theory is that stable,
immersed $H$-surfaces with boundary in homogeneously regular 3-manifolds
(see Definition~\ref{defhomgreg} below) have curvature estimates
up to their boundary. These curvature estimates were first obtained by
Schoen for two-sided 0-surfaces and later improved by Ros
to the one-sided 0-case, and are a simple consequence of
Theorem~\ref{thmstablecompleteplane} after a rescaling argument.
\begin{definition}
\label{defhomgreg}
{\rm
A Riemannian 3-manifold $N$ is {\it homogeneously regular} if there exists
an $\ve > 0$ such that $\ve$-balls in $N$ are uniformly close to
$\ve$-balls in $\rth$ in the $C^2$-norm. In particular, if $N$ is
compact, then $N$ is homogeneously regular.
}
\end{definition}
\begin{theorem} [Schoen \cite{sc3}, Ros \cite{ros9}]
\label{thmcurvestimstable}
Let $N$ be a homogeneously regular 3-manifold. Then,
there exists a universal constant $c>0$
such that for
any stable immersed $H$-surface $M$ in $N$,
\[
|A_M(p)|\, {d_N}(p,\partial M)^2\leq c\quad \mbox{for all }p\in M,
\]
where {$d_N$}
denotes distance in $N$ and $\partial M$ is the boundary of $M$.
\end{theorem}

Rescaling arguments and results of L\'opez and Ros~\cite{lor2} for complete immersed
$0$-surfaces in $\R^3$ with index of stability 1 demonstrate that
given a homogeneously regular 3-manifold $N$, there exist similar curvature estimates
for {two-sided $H$-surfaces that are weakly stable}
in the sense of Definition~\ref{defwstable}.

We also note that Rosenberg, Souam and Toubiana~\cite{rst1}
have obtained the following version of Theorem~\ref{thmcurvestimstable}
valid for $H$-surfaces in the two-sided case when the ambient 3-manifold has
a bound on its sectional curvature.

\begin{theorem} [Rosenberg, Souam and Toubiana~\cite{rst1}]
\label{thmcurvestimstable1}
Let $N$ be a 3-manifold with a bound $k_0$ on its absolute sectional curvature.
There exists a universal constant $c>0$ (depending on $k_0$) such that for
any stable, two-sided immersed $H$-surface $M$ in $N$,
\[
|A_M(p)|\, {d_N}(p,\partial M)^2\leq c\quad \mbox{for all }p\in M.
\]
\end{theorem}

If we weaken the stability
hypothesis in Theorem~\ref{thmstablecompleteplane} to finite index of
stability and we allow compact boundary, then completeness and orientability also lead to a
well-known family of immersed 0-surfaces.

\begin{theorem}[Fischer-Colbrie~\cite{fi1}]
 \label{thmfiniteindexftc}
Let $M\subset \R^3$ be a
complete, orientable immersed 0-surface in $\R^3$, with (possibly empty) compact
boundary. Then, $M$ has finite index of stability if and only if it has finite
total curvature.
In this case, the
index and nullity of $M$ coincide with the index and nullity of the
meromorphic extension of its Gauss map to the compactification
$\overline{M}$
obtained from $M$
after attaching its ends.
\end{theorem}

In order to make sense of  the last statement in the above theorem, recall
Huber's~\cite{hu1} parabolicity result that implies that if a complete Riemannian
surface with compact boundary has finite total curvature, then it is conformally
a compact Riemann surface, and, as shown by  Osserman~\cite{os3}, a simple application
of Picard's theorem implies the Gauss map extends holomorphically across the
the punctures to the conformal compactification.

\section{The flux of a Killing field.}
\label{sec:flux}

We next describe the notion of the
flux of a 1-cycle on an $H$-surface; see for instance~\cite{kks1,ku2,smyt1}
for further discussion of this invariant.  This generalizes the previous definition of
flux $F(\g)$ of a 1-cycle $\g$ on an immersed 0-surface given in equation~\eqref{eq:flux}
to the $H$-surface setting.

\begin{definition}[CMC Flux] \label{def:flux}
{\em
Let $\gamma$ be a piecewise-smooth 1-cycle in an immersed $H$-surface $M\subset \R^3$.
The  flux vector of $M$ along $\gamma$ is
\begin{equation}
\label{HFlux}
F(\g)=\int_{\gamma}(H\gamma+N)\times {\gamma'},
\end{equation}
where $N$
is the unit normal to $M$ and {$\g'$} is the velocity vector of $\g$
(compare with~\eqref{eq:flux}).
 }
\end{definition}

In the case of a  properly immersed 0-surface $M$ in $\rth$, one can associate for any $t\in \R$
its {\em scalar vertical flux $V_M(t)$ across the plane $\{ x_3=t \},$ } which is the possibly
improper integral
\[
V_M (t)= \int_{\partial (M\cap \{x_3\leq t \})} |\nabla x_3| \;\;\in  (0,\infty],
\]
where $\nabla x_3$ denotes the intrinsic gradient of the third coordinate function of $M$.

\begin{theorem}[Scalar vertical flux, Meeks \cite{me22}]
\label{Meeks-flux}
Let $M$ be a  properly immersed 0-surface in $\rth$. Then, $V_M (t)$
does  not depend on $t\in \R$. Hence,
without ambiguity we define
$V_M\in  (0,\infty]$ as the flux of $\nabla x_3$ across any
horizontal plane and we call $V_M$
the {\em scalar vertical flux of $M$}.
\end{theorem}
We next give a sketch of proof of Theorem~\ref{Meeks-flux}, partly to motivate some other
important theoretical results and techniques in the subject.  We first recall
the notion of a parabolic Riemannian manifold $M$ with boundary, and refer the reader
to Section~7 of the book~\cite{mpe10} for further details.

\begin{definition}
\label{defparabolic}
{\rm
Let $(M^n,g)$ be an $n$-dimensional
Riemannian manifold
with non-empty boundary. $M$ is {\it parabolic} if every bounded
harmonic function on
$M$ is determined by its boundary values.
}
\end{definition}

In dimension $n=2$, the
property of a Riemannian surface with boundary to be parabolic is a conformal one,
and any proper smooth subdomain of a parabolic manifold is also parabolic.
One way to show that  an $n$-dimensional
Riemannian manifold $(M^n,g)$ is parabolic
is to prove that there exists a proper, positive
superharmonic function on it~\cite{gri1}.

Collin, Kusner, Meeks and Rosenberg~\cite{ckmr1}
constructed
ambient functions on certain proper non-compact regions in $\rth$
with the property that they restrict to any minimal surface to
be superharmonic; they called such functions {\em universal superharmonic functions}.
Using the universal superharmonic function
$f(x_1,x_2,x_3)=-x_3^2+\ln(\sqrt{x_1^2+x_2^2})$ defined
on $\{(x_1,x_2,x_3) \mid x_1^2 +x_2^2\geq 1\}$, they proved that the intersection of
a properly immersed 0-surface with boundary in $\rth$ and contained in a half-space
is parabolic. In particular, if $M$ is a  properly immersed 0-surface in $\rth$,
then for any real numbers $t_1<t_2$, the subdomain
\[
M[t_1,t_2]=M\cap \{(x_1,x_2,x_3) \mid t_1\leq x_3\leq t_2\}
\]
is a parabolic surface with boundary contained in the union of the planes $\{x_3=t_i\}$,
$i=1,2$; note that $x_3$ is a bounded harmonic function $h$ on  the parabolic Riemannian
manifold $X=M[t_1,t_2]$ with $\partial X\subset h^{-1}(\{t_1,t_2\})$. In this more
general setting, Meeks proved that the scalar flux of $\nabla h$ across $ h^{-1}(t_1)$ is the same
as the flux across $ h^{-1}(t_2)$, which then proves Theorem~\ref{Meeks-flux} (see~\cite{me22} and also see
the proof of Proposition~4.16 in~\cite{mpr14} for similar calculations). This
completes our sketch of the proof of Theorem~\ref{Meeks-flux}.
\vspace{.3cm}

There is a related notion of flux that generalizes the formula (\ref{HFlux})
and works in the $n$-dimensional
Riemannian manifold setting.  The proof of the next theorem is straightforward
and follows from two applications of the Divergence Theorem; see  the proof below or
the similar calculations in the proof of the conservation laws in Theorem~4.1 in~\cite{kkms1}.

\begin{theorem}[CMC Flux Formula] \label{fluxK}
Let $(X,g)$ be an $n$-dimensional orientable Riemannian manifold,
$M\subset X$ be an orientable hypersurface of constant mean curvature
and $K$ be a Killing field.  Suppose that $\Sigma ,\Sigma'\subset X$ are
$(n-1)$-chains with boundaries $\partial \S=\G \subset M$,
$\partial\Sigma'=\G'\subset M$, such that the $(n-2)$-cycles $\G,\G'$
are homologous in $M$ (i.e., there exists an
$(n-1)$-chain  $M(\G,\G')\subset M$ with boundary $\partial M(\G,\G')=\G -\G'$)
and $\Sigma -\Sigma'+M(\G ,\G')$ is a boundary in $X$  (i.e., there exists an
$n$-chain  $\Omega \subset X$ with boundary $\partial \Omega =\Sigma -\Sigma'+M(\G ,\G')$).
Consider the pairing:
\begin{equation}
\label{fluxformula}
\mbox{\rm Flux}(\G,\S,K) = \int_\G g(\eta_\G, K) +(n-1)H \int_\S g(N_{\S}, K)\in \R,
\end{equation}
where $H\in \R $ is the constant value of the mean curvature of $M$ with
respect to the outward pointing normal vector to $\Omega $,
$N_{\S}$ is the unit normal field to $\S$ that is outward pointing on $\Omega $, and
$\eta_\G$ is the unit normal field to $\G$ in $TM$ that is outward
pointing on $M(\G ,\G')$. Then, $\mbox{\rm Flux}(\G,\S,K)=\mbox{\rm Flux}(\G',\S',K)$.

In particular, if the $n$-th homology group $H_n(X)$ of $X$ vanishes, then
$\mbox{\rm Flux}(\G,\S,K)$ depends only on the homology class of
$[\G]\in H_{n-1}(X)$ and on  the Killing field $K$.
\end{theorem}
\begin{proof}
As $K$ is a Killing vector field, then the bilinear
map $(u,v)\in TX\times TX\mapsto g(\overline{\nabla }_uK,v)$ is skew-symmetric,
where $\overline{\nabla }$ stands for the metric connection of $X$. This
implies that the divergence div$_X(K)$ of $K$ in $X$ vanishes identically
and that the divergence on $M$ of the tangent part $K^T$ of $K$ to $M$
is given by div$_M(K^T)=(n-1)Hg(K,N_M)$,
where $N_M$ is the unit normal vector field of $M$ for which $H$ is the
mean curvature (in particular,
$N_M$ is outward pointing on $\Omega $ along $M(\G,\G')$).

Applying the
divergence theorem to $K$ in $\Omega $, one obtains
\begin{equation}
\label{flux1}
0=\int _{\Omega }\mbox{div}_X(K)=\int_{\Sigma }g(K,N_{\Sigma })-
\int _{\Sigma'}g(K,N_{\Sigma'})+\int _{M(\G ,\G')}g(K,N_M).
\end{equation}
Analogously, the divergence theorem applied to $K^T$  in $M(\G ,\G')$  gives
\begin{equation}
\label{flux2}
(n-1)H\int _{M(\G,\G')}g(K,N_M)=\int _{\G }g(K,\eta _{\G })-\int _{\G'}g(K,\eta _{\G'}).
\end{equation}
Plugging (\ref{flux2}) into (\ref{flux1}) we deduce
that $\mbox{\rm Flux}(\G,\S,K)=\mbox{\rm Flux}(\G',\S',K)$, as desired.
\end{proof}

\begin{remark} \label{remark3.5}{\em
Let $\partial_{x_i}$, $i=1,2,3$  denote the usual constant
Killing fields in $\rth$ endowed with its usual flat metric, let $\g\subset M$ be
a 1-cycle on an immersed  $H$-surface.  Then,  the $i$-th component $F_i$  of the flux vector
$F$ defined in (\ref{HFlux}) is equal to
$\mbox{\rm Flux}(\g,M,\partial_{x_i})$,  and so,
it only depends on the homology class of $\g$ in $M$.  In the $\rth$-setting, there are other
Killing fields generated by one-parameter groups of rotations around
a line and the corresponding fluxes obtained from these additional Killing fields
give rise to other invariants for 1-cycles  on an immersed $H$-surface (torque or momentum).
}
\end{remark}

\section{Classification results for $H$-surfaces of finite genus in $\rth$.}
\label{subsecends}

In this section we will first review some of the main results
in the classical theory of complete $H$-surfaces in $\rth$ in the context of recent
results by Colding and Minicozzi~\cite{cm35}, Meeks, P\'erez and Ros~\cite{mpr9, mpr6}
and Meeks and Tinaglia~\cite{mt7}. After this review, we will present the classification
of the asymptotic behavior of annular ends of 0-surfaces in $\rth$ given by Meeks and P\'erez~\cite{mpe3}.

\subsection{Classification results.} \label{class:results} The next theorem demonstrates that
classification questions
for complete $H$-surfaces in $\rth$ are equivalent to the similar classification questions
for proper $H$-surfaces  under an appropriate
constraint on the global or the local topological properties
of the surface.  As stated, Theorem~\ref{thmCM} below depends
on results in several different papers. The first one of these
is the proof by Colding and Minicozzi~\cite{cm35} that
complete 0-surfaces of finite topology  in $\rth$ are proper.
Using some of the techniques in~\cite{cm35}, Meeks and Rosenberg~\cite{mr13} proved that
the closure of a complete 0-surface with positive injectivity
radius in a Riemannian 3-manifold
has the structure a 0-lamination (leaves are minimal surfaces)
and they used this lamination closure property  to prove that
a complete $0$-surface in $\R^3$ with positive injectivity radius is proper.
Recently, Meeks and Tinaglia~\cite{mt7} have been able
to generalize the results in both of these previous papers to the $H>0$
setting.  Summarizing the results into one statement,
we have the next fundamental theorem.

\begin{theorem}
\label{thmCM}
A complete  $H$-surface in $\rth$ of finite topology or positive injectivity radius
is proper.
\end{theorem}

\begin{remark} \label{remarkENDS}
{\em
The properness conclusion in Theorem~\ref{thmCM} also holds if the $H$-surface has
compact boundary and finite topology or if the surface has
compact boundary with injectivity radius function bounded away from zero
outside of some small neighborhood of its
boundary. In particular, annular ends of a complete $H$-surface in $\rth$ are proper.}
\end{remark}

In fact, by the curvature estimates in Theorem~\ref{cest} for $(H>0)$-disks,
it can be seen that a complete  $(H>0)$-surface
has bounded second fundamental form if and only if it has positive injectivity radius.
We will discuss these results of Meeks and Tinaglia~\cite{mt7} in Section~\ref{sec:MT}.
A fundamental open problem concerning classical $H$-surfaces is
the following one.

\begin{conjecture}[Meeks, P\'erez, Ros, Tinaglia] \label{CYconj1}
A complete  $H$-surface in $\rth$ of finite genus
is proper. More generally, for every such surface $M$, there exists $C_M>0$
such that for any ball $\B(p,R)$ in $\rth$ with radius $R\geq 1$,
$\mbox{\rm Area}(M\cap \B(p,R))\leq C_{M} \, R^3$.
\end{conjecture}

We remark that Meeks, P\'erez and Ros~\cite{mpr9} have obtained the
following partial result on the above conjecture.

\begin{theorem}
\label{thmCYMPR}
A complete  $0$-surface in $\rth$ of finite genus is proper if and only
if it has a countable number of ends.
\end{theorem}

A fundamental problem in classical surface
theory is to describe the behavior of a proper non-compact $H$-surface
$M\subset \R^3$ outside  large  compact sets in space.  This
problem is well-understood if $M$ is minimal with finite total curvature,
 because in this case, each of the ends of $M$ is asymptotic to
 an end of a plane or a catenoid.
In~\cite{col1}, Collin proved that a proper $0$-surface with at
least two ends and  finite topology must
have finite total curvature; hence by the properness of finite
topology $H$-surfaces in $\rth$, we have the next
fundamental result.

\begin{theorem}
\label{thmCM2}
A complete  $0$-surface in $\rth$ of finite topology and at least
two ends has finite total curvature.
In particular, each of its ends is asymptotic to the end of a plane or a catenoid.
\end{theorem}

The next Theorem~\ref{ttmr} by Bernstein and  Breiner in~\cite{bb2} states that if $M$
is a $0$-surface with finite
topology but infinite total curvature (thus $M$ has exactly one end by  Theorem~\ref{thmCM2}),
then $M$ is asymptotic to
a helicoid; this result is based on some of the techniques that
Meeks and Rosenberg used  in the proof of the uniqueness of the helicoid.
The proof of the next theorem
was  found independently by
Meeks and P\'erez~\cite{mpe3} (see Section~\ref{secCTI}) who considered the
asymptotic behaviors of annular ends  in the more general context
 where the 0-surface $M$ has compact boundary.

\begin{theorem} \label{ttmr}
A complete non-flat $0$-surface in $\rth$ of finite topology and one end is asymptotic
to a helicoid.  If the surface also has genus zero, then it is a helicoid.
\end{theorem}

In~\cite{mt11}, Meeks and Tinaglia describe proper
1-surfaces $M_k$ in $\rth$ which are doubly-periodic (invariant by two independent translations)
and contained in an open slab  of
width $\frac{1}{2^{k+1}}$. After stacking these
slabs with their surfaces on top of each other, one obtains a
complete, injectively immersed disconnected surface $M_\infty$
of constant mean curvature 1 that is  properly embedded in the slab
$\{(x_1,x_2,x_3)\in \rth \mid -1<x_3<1\}$ but  $M_\infty$ is not properly embedded
in $\rth$.

In the classical setting, Meeks~\cite{me17} proved that a proper $(H>0)$-annular end
is contained in a solid half-cylinder in $\rth$, and then used this result to prove that
there do not exist any proper $(H>0)$-surfaces in $\rth$ with finite topology and just
one end. Based in part on Meeks' results, Korevaar, Kusner and Solomon~\cite{kks1} then
proved that a proper $(H>0)$-annular end in $\R^3$ is asymptotic to a Delaunay surface, and that
a proper $(H>0)$-surface of finite topology and two ends is a Delaunay surface.
By Theorem~\ref{thmCM}, we have the following result.

\begin{theorem} \label{kks}
Let $M$ be a complete $(H>0)$-surface in $\rth$.
\ben
\item Each annular end of $M$   is asymptotic to the end of a Delaunay surface.
\item If $M$ has finite topology, then it has at least two ends.
\item If $M$ has finite topology and two ends, then it is a Delaunay surface.
\een
\end{theorem}

The first deep classification result for complete 0-surfaces
with finite topology in $\rth$ is the following one due to
Schoen~\cite{sc1}, who proved the following theorem as an
application of the Alexandrov reflection
technique.

\begin{theorem}
\label{schth2}
The catenoid is the unique complete, immersed 0-surface in $\R^3$ with finite total curvature
and two embedded ends.
\end{theorem}

After Schoen's result, L\'opez and Ros~\cite{lor1} used a
deformation argument based on the Weierstrass
representation to prove the following classification theorem.
\begin{theorem}
\label{thmlr}
The only complete 0-surfaces in $\rth$ with finite total curvature
and genus zero are the plane and the catenoid.
\end{theorem}
We next summarize some of the above classification results; in particular,
Theorem~\ref{helicoid} follows from the next theorem.

\begin{theorem}
\label{kks2}
Let $M$ be a complete $H$-surface of finite topology in $\rth$. Then:
\ben[1.]
\item $M$ is proper and has bounded second fundamental form.
\item Each annular end of $M$   is asymptotic to the end of a Delaunay
surface, a plane, a catenoid or a helicoid.
\item If $M$ simply connected, then it is a plane, a sphere, a catenoid or a helicoid.
\item If $M$ has two ends, then it is a catenoid or a Delaunay surface.
\item If $M$ has genus zero and it is a $0$-surface, then it is a plane, a catenoid or a helicoid.
\item If $H>0$, then $M$ has at least two ends.
\een
\end{theorem}

We next explain some of the elements in the proof of Theorem~\ref{classthm}, which
classifies the proper 0-surfaces in $\R^3$ with genus zero; the case of finite topology is covered
by Theorem~\ref{kks2} above and does not need the hypothesis of properness but only the weaker one of
completeness. Unfortunately it is not known at the present moment
if every complete $H$-surface of finite genus in $\rth$ is proper;
see Conjecture~\ref{CYconj1} and Theorem~\ref{thmCYMPR} for a related discussion.
So for this reason, in the next theorem we will assume that the surface is proper.
The results summarized in the next theorem
can be found in~\cite{mpr4}, where it is shown that a proper
finite genus $0$-surface  with infinite topology in $\R^3$
must have two limit ends, and in~\cite{mpr6}; {
The space of ends ${\cal E}(M)$ of a non-compact connected manifold $M$
has the following natural Hausdorff topology. For each
proper domain $\Omega \subset M$ with compact boundary, we define the basis open set
$B(\Omega ) \subset {\cal E}(M)$ to be those equivalence classes in ${\cal E}(M)$
which have representatives contained in $\Omega $. With this topology,
${\cal E}(M)$ is a totally disconnected
compact space. Any isolated point $e \in {\cal E}(M)$ is called a {\it simple end} of
$M$. If $e\in {\cal E}(M)$ is not a simple end (equivalently, if it is a limit point of
${\cal E}(M)$), then $e$ is called a {\it limit end} of $M$.
In the case that $M$ is a proper $0$-surface in $\R^3$ with more than one end,
then Frohman and Meeks~\cite{fme2} showed that  ${\cal E}(M)$ can be equipped with
a linear ordering by the relative heights of the ends over the
$(x_1,x_2)$-plane (after a rotation in $\R^3$).  One
defines the {\it top end} $e_T$ of $M$ as the unique maximal element in
${\cal E}(M)$ for this linear ordering. Analogously, the {\it bottom end}
$e_B$ of $M$ is the unique minimal element in ${\cal E}(M)$. If $e\in {\cal E}(M)$ is neither the
top nor the bottom end of $M$, then it is called a {\it middle end} of $M$.
}

\begin{theorem}
\label{classthm2}
Let $M\subset \rth$
proper 0-surface with infinite topology.
\begin{enumerate}[1.]
\item If $M$ has genus zero, then $M$ is one of the Riemann minimal
examples.

\item $M$ has finite genus greater than zero, then $M$ has two limit
ends and each of its middle ends is planar. Furthermore,
after a homothety and rigid motion, the following properties hold.
\ben[2.a.]
\item $M$ is conformally diffeomorphic to a compact Riemann surface
of genus $g$ minus a countable closed subset
${\cal E}_M=\{ e_n\} _{n\in \Z }\cup \{ e_{\infty},e_{-\infty}\} \subset \overline{M}$,
where $\lim_{n\to -\infty}
e_n =e_{-\infty} $ and $ \lim_{n\to\infty}e_n= e_{\infty} $.

\item There exists a Riemann minimal example $R_t$ so that the
top end of $M$ converges exponentially
to the top end of a translated image of ${\cal R}_t$ in the
following sense: there exists a vector $v^+\in \R ^3$ and
representatives ${\cal R}_t^+$ and $M^+$ of the top ends
of ${\cal R}_t$ and $M$ respectively,
such that $M^+$ can be expressed as the graph over ${\cal R}_t^++v^+$ given by a smooth function
$f$ defined on the half-cylinder $\overline{{\cal R}_t^+}$
obtained after attaching to ${\cal R}_t^+$ its planar ends,
such that $f$ decays exponentially as the height $x_3\to \infty $ on ${\cal R}_t^+$.
In the same way, the bottom end of $M$ is exponentially asymptotic
to ${\cal R}_t^-+v^-$ for a certain translation
vector $v^-\in \R^3$.
\een
\end{enumerate}
\end{theorem}

\subsection{Embedded 0-annular ends with infinite total curvature.}
\label{secCTI}

\label{subsecannular}

In this section we will describe the asymptotic
behavior, conformal structure and analytic
representation of an annular end of any
complete 0-surface $M$ in $\rth$ with compact boundary
and finite topology (hence proper by Remark~\ref{remarkENDS}).
For detailed proofs of the results in this section,
see Meeks and P\'erez~\cite{mpe3}.

Take two numbers $a\in [0,\infty )$, $b\in \R $.
Next we outline how to construct examples $E_{a,b}\subset \R^3$ of complete
0-annuli with compact boundary, conformally parameterized
in $D(\infty ,R)=\{z\in \C \mid |z|\geq R\}$, for some $R>0$,
so that their flux vector along their boundary is $(a,0,-b)$ and their total
Gaussian curvature is infinite. These annuli $E_{\a,b}$
will serve as models for the asymptotic geometry of every complete
embedded minimal end with infinite total curvature and compact boundary.
To define $E_{a,b}$, we will use the Weierstrass representation
$(g,dh)$ where $g$ is the Gauss map and $dh$ the height differential. First
note that after an isometry in $\R^3$ and a possible
change of orientation, we can assume that $b\geq 0$.
We consider three separate cases.
\begin{enumerate}[(C1)]
\item If $a=b=0$, then define  $g(z)=e^{iz}$, $dh=dz$, which produces
the end of a vertical helicoid.
\item If $a\neq 0$ and $b\geq 0$ (i.e., the flux vector is not vertical), we choose
\begin{equation}
\label{eq:gdh1}
g(z)=t\, e^{iz}\frac{z-A}{z}, \qquad dh=\left( 1+
 \frac{B}{z}\right) dz, \qquad z\in D(\infty ,R),
\end{equation}
where $B=\frac{b}{2\pi }$, and the parameters $t>0$ and $A\in \C-\{ 0\} $ are to be
determined (here $R>|A|$). Note that with this choice of $B$, the imaginary part of
$\int _{\{ |z|=R\} }dh$ is $-b$ because we use the orientation of $\{ |z|=R\} $
as the boundary of $D(\infty ,R)$.
\item In the case of vertical flux, i.e., $a=0$ and $b>0 $, we take
\begin{equation}
\label{eq:gdh2} g(z)=e^{iz}\frac{z-A}{z-\overline{A}},\qquad dh=
\left( 1+\frac{B}{z}\right) dz, \quad z\in D(\infty, R),
\end{equation}
where $B=\frac{b}{2\pi }$ and $A\in \C-\{0\}$ is to be determined (again $R> |A|$).
\end{enumerate}

In each of the three cases above, $g$ can rewritten as $g(z)=t\, e^{iz+f(z)}$
where $f(z)$ is a well-defined holomorphic function in $D(\infty ,R)$ with $f(\infty )=0$.
In particular, the differential $\frac{dg}{g}$ extends meromorphically
through the puncture at $\infty $. The same extendability holds for $dh$. These
properties will be collected in the next notion, that was first introduced by
Rosenberg~\cite{rose1} and later studied by Hauswirth, P\'erez and Romon~\cite{hkp1}.
\begin{definition}
\label{deffinitetype}
A complete immersed $0$-surface $M\subset \R^3$
with Weierstrass data $(g,dh)$ is of {\it finite type} if
$M$ is conformally diffeomorphic to a finitely punctured, compact
Riemann surface $\overline{M}$ and after a possible rotation, both
$dg/g$, $dh$ extend meromorphically to $\overline{M}$.
\end{definition}

Coming back to our discussion about the annular minimal ends $E_{a,b}$,
to determine the parameters $t>0$, $A\in \C -\{ 0\} $ that appear in cases (C2), (C3) above,
one studies the period problem for $(g,dh)$. The only period to be
killed is the first equation in (\ref{eq:periodproblem}) along $\{ |z|=R\} $, which
can be explicitly computed in terms of $t,A$. An intermediate value argument gives that
given $B=\frac{b}{2\pi }$, there exist parameters $t>0$, $A\in \C -\{ 0\} $
so that the Weierstrass data given by (\ref{eq:gdh1}), (\ref{eq:gdh2})
solve this corresponding period problem. At the same time, one can calculate
the flux vector $F$ of the resulting 0-immersion along its boundary and prove that
its horizontal component covers all possible values.
This defines for each $a,b\in [0,\infty )$ a complete immersed 0-annulus $E_{a,b}$ with compact boundary,
infinite total curvature and flux vector $(a,0,-b)$. Embeddedness of $E_{a,b}$ will be discussed below.
With the notation above, we will call the end $E_{a,b}$ a {\it canonical end}
(in spite of the name ``canonical end'', we note that the choice of $E_{a,b}$ depends on the
explicit parameters $t,A$ in equations (\ref{eq:gdh1}) and (\ref{eq:gdh2})).

\begin{remark}
{\rm
In case (C3) above, it is easy to
check that the conformal map $z \stackrel{\Phi }{\mapsto
}\overline{z}$ in the parameter domain $D(\infty,R)$ of $E_{0,b}$
satisfies $g\circ \Phi =1/\overline{g}$, $\Phi ^*dh=\overline{dh}$.
Hence, after translating the surface so that the image of the point
$R\in D(\infty ,R)$ lies on the $x_3$-axis, we deduce that $\Phi $
produces an isometry of $E_{0,b}$ which extends to a 180-rotation
of $\rth$ around the $x_3$-axis; in particular,
$E_{0,b}\cap (x_3$-axis)  contains  two infinite rays.
}
\end{remark}

To understand the geometry of the canonical end
$E_{a,b}$ and in particular prove that it is
embedded if $R$ is taken large enough,
it is worth analyzing its multi-valued graph structure,
which is the purpose of Theorem~\ref{thm1.3} below. Before
stating this result, we need some notation.

\begin{definition}
\label{Ngraph}
{\rm
In polar coordinates $(\rho, \t )$ on $\R^2-\{ 0\} $
with $\rho >0$ and
$\t \in \R $, a {\it $k$-valued graph on an annulus of inner radius
$r$ and outer
radius $R$,} is a single-valued graph of a real-valued function $u(\rho ,\t )$
defined over
\begin{equation}
\label{eq:sectormultigraph}
S_{r,R}^{-k,k}=\{
(\rho ,\t )\ | \ r\leq \rho \leq R,\ |\t |\leq k\pi \} ,
\end{equation}
$k$ being a positive
integer (see Figure~\ref{multigrph}).
\begin{figure}
\begin{center}
\includegraphics[height=3.1cm]{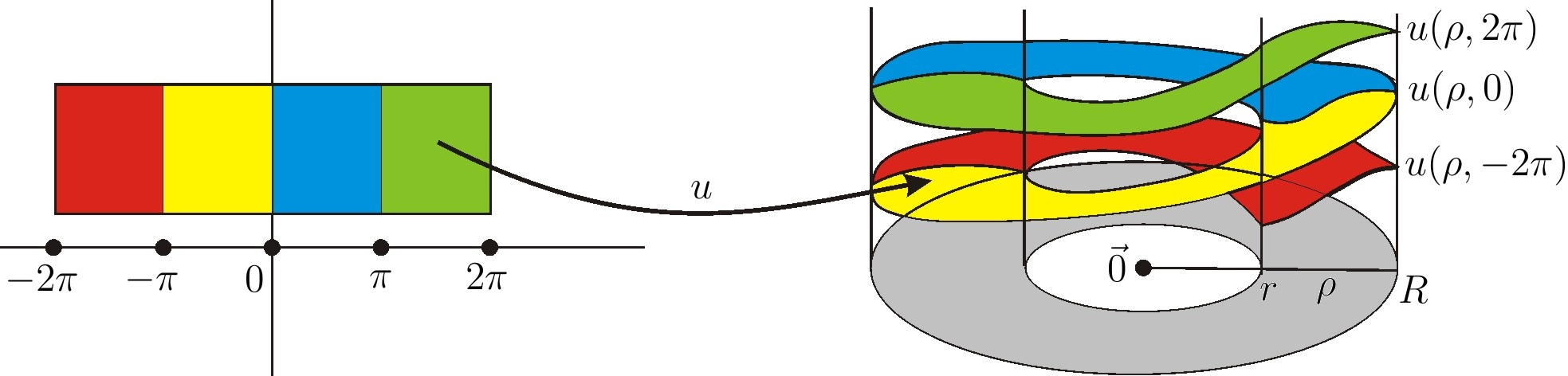}
\caption{A $2$-valued graph with positive separation.}
\label{multigrph}
\end{center}
\end{figure}
The {\it separation} between consecutive sheets is
\begin{equation}
\label{eq:sep}
w(\rho ,\t )=u(\rho ,\t+2\pi )-u(\rho ,\t )\in \R.
\end{equation}
 The surface $\Sigma =\{ (\rho \cos \theta ,\rho \sin \theta ,u(\rho ,\theta ))
\ | \ (\rho ,\theta )\in S_{r,R}^{-k,k}\} $
is clearly embedded if and only if
$w>0$ (or $w<0$). The multi-valued graph $u$ is said
to be an $H$-multi-valued graph if it is an $H$-surface.}
\end{definition}

Note that the separation function $w(\rho ,\theta )$ used in Theorem~\ref{thm1.3} below
refers to the vertical separation between the two disjoint multi-valued graphs
$\Sigma _1,\Sigma _2$ appearing in the next result (versus
the separation used in Definition~\ref{Ngraph}, which measured the
vertical distance between two consecutive sheets of the {\it same} multi-valued graph). We
also use the notation $\widetilde{D}(\infty ,R)=\{ (\rho ,\theta )
\mid \rho \geq R, \theta \in \R \} $ and $C(R)=\{ (x_1,x_2,x_3)\in
\R^3\mid x_1^2+x_2^2\leq R^2\} $.

\begin{theorem}[Asymptotic behavior of $E_{a,b}$]
\label{thm1.3}
Given $a\geq 0$ and $b\in \R$, the canonical end $E=E_{a,b}$ satisfies
the following properties.
\begin{enumerate}[1.]
\item There exists $R=R_E>0$ large such that
    $E_{a,b}-C(R)$ consists of two disjoint multi-valued graphs $\Sigma _1,
\Sigma _2$ over $D(\infty ,R)$ of smooth functions $u_1,u_2\colon
\widetilde{D}(\infty ,R)\to \R $ such that their gradients satisfy
$\nabla u_i(\rho,\theta )\to 0$ as $\rho \to \infty $ and the separation
function $w(\rho,\theta )=u_1(\rho,\theta )-u_2(\rho,\theta )$ between both
multi-valued graphs converges to $\pi $ as $\rho+|\theta |\to \infty $.
Furthermore for $\t $ fixed\,\footnote{This condition expresses the intersection
of $E-C(R_E)$ with a vertical half-plane
bounded by the $x_3$-axis, of polar angle $\theta $, see Figure~\ref{Eab}.} and $i=1,2$,
\begin{equation}
\label{eq:graphr}
 \lim_{\rho \to \infty} \frac{u_i(\rho ,\theta)}{\log
(\log \rho )} =\frac{b}{2\pi }.
\end{equation}
\item The translated surfaces $E_{a,b}+(0,0,-2\pi n-\frac{b}{2\pi
}\log n)$ (resp. $E_{a,b}+(0,0,2\pi n-\frac{b}{2\pi }\log n)$)
converge as $n\to \infty $ to a vertical helicoid $H_T$ (resp.
$H_B$) such that
\begin{equation}
\label{eq:HBHT}
H_B=H_T+(0,a/2,0).
\end{equation}
The last equality together with item 1 imply that for different values of $(a,b)$, the
related canonical ends $E_{a,b}$ are not asymptotic after a rigid motion
and homothety. See Figure~\ref{Eab} for a description of how the flux vector $(a,0,-b)$ of $E_{a,b}$
influences its geometry.
\begin{figure}[h]
\begin{center}
\includegraphics[height=8cm]{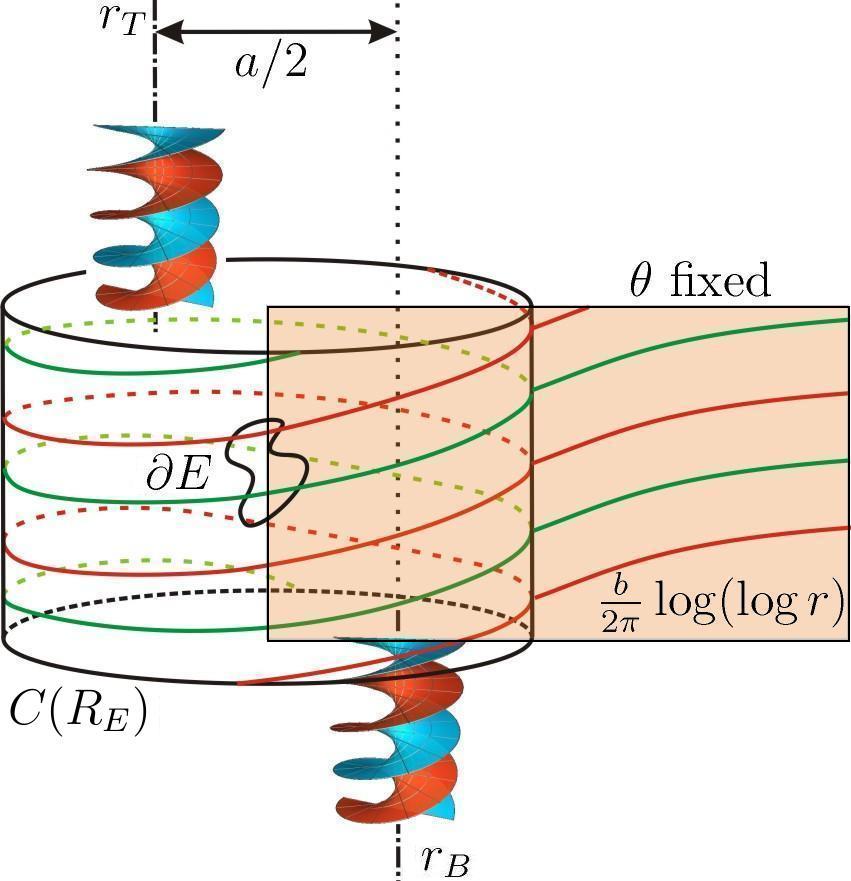}
\caption{A schematic image of the embedded canonical
end $E=E_{a,b}$ with flux vector $(a,0,-b)$ (see
Theorem~\ref{thm1.3}).  The half-lines $r_T,r_B$ refer
to the axes of the vertical helicoids $H_T,H_B$.
}
\label{Eab}
\end{center}
\end{figure}

\end{enumerate}
\end{theorem}
The proof of Theorem~\ref{thm1.3} is based on a careful analysis of the
horizontal and vertical projections of $E=E_{a,b}$ in terms of the
Weierstrass representation. In fact, the explicit expressions of $g,dh$
in equations (\ref{eq:gdh1}), (\ref{eq:gdh2}) is not used, but only that
those choices of $g,dh$ have the common structure
\begin{equation}
\label{eq:WBB}
g(z)=e^{iz+f(z)},\quad dh=\left( 1+\frac{\l }{z-\mu }\right) dz,
\end{equation}
where $\l \in \R $, $\mu \in \C $ and  $f\colon D(\infty ,R)\cup \{ \infty \} \to \C $
is a holomorphic function such that $f(\infty)=0$
(the multiplicative constant $t$ appearing in equation (\ref{eq:gdh1})
can be absorbed by $\mu $ in (\ref{eq:WBB}) after an appropriate change of variables in the
parameter domain).

With the canonical examples at hand, we are now ready to state the main result of this section.
\begin{theorem}
\label{th9.8}
Let $E\subset \rth$ be a complete 0-annulus with
infinite total curvature and compact boundary.
Then, $E$ is conformally diffeomorphic to a punctured disk, its Gaussian
curvature function is bounded, and after replacing $E$ by a subend and
applying a suitable homothety and
rigid motion, we have:
\ben[1.]
\item The Weierstrass data of $E$ is of the form
(\ref{eq:WBB}) defined on $D(\infty ,R)$ for some $R>0$, where $f$
is a holomorphic function in $D(\infty ,R)$ with $f(\infty )=0$, $\l \in \R $ and $\mu
\in \C$.  In particular, $dh$ extends meromorphically across
infinity with a double pole.
\item $E$ is asymptotic to the canonical end $E_{a,b}$
determined by the equality $F=(a,0,-b)$, where $F$ is the flux vector
of $E$ along its boundary. In particular, $E$ is asymptotic to the end of a
helicoid if and only if it has zero flux.
\een
\end{theorem}

\begin{remark}
{\rm
The first statement of Theorem~\ref{ttmr} is now a direct
consequence of Theorem~\ref{th9.8}, since the divergence theorem
insures that the flux vector of a one-ended 0-surface along
a loop that winds around its end vanishes.
}
\end{remark}

\section{Constant mean curvature surfaces in $\esf^3$ and $\esf^2\times\R$.} \label{SectionReferee}

As mentioned in the introduction, this survey's primarily focus is
on past and recent work in the theory of minimal and constant
mean curvature surface that has been done by the authors. However,
in this section we will discuss some key selected results in the
field which are not directly related to our work. These results
are related to constant mean curvature surfaces in $\esf^3$ and $\esf^2\times\R$.

\subsection{Brendle's proof of the Lawson Conjecture.}
We begin by talking about Brendle's proof of the Lawson Conjecture.
See~\cite{bren1} for the actual paper and \cite{bren2} for Brendle's complete survey of this problem
and related questions.
Let $\mathbb S^3$ denote the unit sphere 
in $\mathbb R^4$.
The Clifford Torus is the torus defined by the following set of points
\begin{equation}\label{clifford}
\{(x_1,x_2,x_3,x_4)\in \mathbb R^4 : x_1^2+x_2^2=x_3^2+x_4^2=\frac 12\}.
\end{equation}
It is a minimally embedded torus with zero intrinsic Gaussian
curvature and its principal curvatures are 1 and $-1$.

In 1970, Lawson proposed the following conjecture, see~\cite{la2}.

\begin{conjecture}[Lawson Conjecture]
Let $\Sigma$ be an embedded minimal torus in $\mathbb S^3$,
then $\Sigma$ is congruent to the Clifford Torus.
\end{conjecture}

While, as we pointed out in Section 2, there are no closed minimal
surfaces in $\mathbb R^3$, in $\mathbb S^3$ many examples of closed
minimal surfaces  have been constructed. The simplest example is the equator, namely
\[
\{(x_1,x_2,x_3,x_4)\in \mathbb S^3\subset \mathbb R^4 : x_4=0\}.
\]
The second simplest embedded example is perhaps the Clifford Torus.
Indeed, constructing closed embedded examples which are different
from the equator and the Clifford Torus {turned out}
to be a non-trivial task.
In~\cite{la3} Lawson proved that given any positive genus $g$, there
exists at least one compact embedded minimal surface
{of genus $g$}
in $\mathbb S^3$.
In fact, he also showed that when the genus is not a prime number, then
there are at least two such {(non-congruent)}
surfaces. Note that the conjecture is false if
the surface is not embedded, see~\cite{la5}. After these embedded examples,
more examples were constructed (see for instance~\cite{kps1,kapya}).

The first example of a classification result for minimal surfaces
in $\mathbb S^3$ was given by Almgren. In~\cite{alm3}, Almgren proved the following theorem
{whose proof is based on the fact that the Hopf differential of a CMC surface
in $\esf^3$ is holomorphic.}

\begin{theorem}\label{almgren}
Up to rigid motions of $\mathbb S^3$, the equator is the only minimal
surface in $\mathbb S^3$ of genus zero.
\end{theorem}

This is the analogous of the Hopf Theorem in $\mathbb R^3$ and in fact
the proof uses similar arguments. Lawson Conjecture is the equivalent of
Theorem~\ref{almgren} but instead for the case of tori. In~\cite{ros6} Ros was able to
prove the Lawson Conjecture assuming some additional symmetries of $\Sigma$.
We refer the interested reader to~\cite{bren2} and the references therein for a better
discussion on previous partial results.

We now give a sketch of Brendle's proof of the Lawson Conjecture.  A key
ingredient in Brendle's proof is the following theorem also due to Lawson~\cite{la3}.

\begin{theorem}
\label{thm5.3}
A surface of genus one minimally immersed in $\mathbb S^3$ has no umbilical points.
\end{theorem}

Another key result is {a Simons-type} identity~\cite{sim1}; namely,
if $\Sigma$ is a minimal surface in $\esf^3$, then
\begin{equation}
\label{Simonsid}
{\Delta_\Sigma(|A|)-\frac{|\nabla |A||^2}{|A|}+(|A|^2-2)|A|=0,}
\end{equation}
where $|A| $ is the norm of the second fundamental form of $\S$.

Finally, the proof {by Brendle of the Lawson Conjecture}
relies on applying a maximum principle type argument to
a function depending on two points. Results of this type were first developed
by Huisken in~\cite{hui1} and then extended by Andrews in~\cite{and1}.
In~\cite{hui1}, among other things, Huisken used these techniques to give
a new proof of Grayson's theorem~\cite{gray1}, that says that under the curve
shortening flow, any embedded curve shrinks to a point in finite time and
asymptotically becomes a circle. In~\cite{and1} Andrews  applied these
ideas in the  mean curvature flow setting.

Let $F\colon \Sigma\to \mathbb S^3$ be a minimal immersion of a genus-one
surface $\Sigma$ into $\mathbb S^3$ and let $\nu$ denote a unit normal vector
field. {Since by Theorem~\ref{thm5.3} $\Sigma$ has no umbilical points, then
$\inf_{\Sigma}|A|>0$ and}
the following quantity is finite:
\[
\omega=\sup_{x,y\in \Sigma,\, x\neq y}\sqrt 2\frac{|\langle \nu (x),F(y)\rangle|}{|A|(x)(1-\langle F(x),F(y) \rangle)}.
\]

There are then two cases to consider, $\omega\leq 1$ and $\omega> 1$.
If $\omega\leq 1$ then
one has that
\[
\frac{|A|(x)}{\sqrt 2}(1-\langle F(x),F(y) \rangle)-\langle \nu (x),F(y)\rangle\geq 0.
\]
Using this, a calculation gives that the second fundamental form of $F$ is
parallel and therefore the principal curvatures are constant. Since by the Gauss equation,
\[
K=1+k_1k_2
\]
where $K$ denotes the intrinsic Gaussian curvature and $k_1,k_2$ the principal
curvatures of $\Sigma$, it follows that $K$ is constant and therefore $\Sigma$
is flat. By a result of Lawson, this implies that $\Sigma$ is congruent to the Clifford Torus.

If $\omega>1$, then Brendle considers the function
\[
Z(x,y)=\frac{\omega}{\sqrt 2}{|A|(x)}(1-\langle F(x), F(y)\rangle)+\langle \nu (x),F(y)\rangle.
\]
Note that this function is non-negative and by possibly replacing $\nu$ by $-\nu$,
there exist $\overline x,\overline y\in \Sigma$, $\overline x\neq\overline y$,
such that $Z(\overline x,\overline y)=0$. {Since the function $Z$ attains
its global minimum at $(\overline x,\overline y)$, then $Z_x(\overline x,\overline y)=
Z_y(\overline x,\overline y)=0$ and the Hessian of Z at $(\overline x,\overline y)$ is
 non-negative.}
 Let
\[
\Delta:=\{x\in \Sigma:\, \text{there exists a point } y\in\Sigma\setminus \{x\} \text{ with } Z(x, y)=0\}.
\]

The set $\Delta$ is non-empty. Reasoning as before, if $x\in \Delta $ and $y\in \Sigma \setminus \{x\} $ satisfies
$Z(x, y)=0$, then $Z_x(x,y)=Z_y(x,y)=0$ and $(\nabla ^2Z)(x,y)\geq 0$.
Using equation (\ref{Simonsid}), Brendle proves that $Z$ is a subsolution of a degenerate elliptic equation
in $\Delta $.
An application of Bony's strict maximum principle~\cite{bony} then gives that $\Delta$ is open.
Finally, Brendle shows that for any $x\in \Delta$, $\nabla |A|(x)=0$.
Since $\Omega$ is open, an analytic continuation type argument gives that
for any $x\in \Sigma$, $\nabla |A|(x)=0$. Just like in the case $\omega\leq 1$,
this implies that $\Sigma$ is flat and thus, by a result of Lawson, congruent to the Clifford Torus.

\subsection{Marques and Neves' proof of the Willmore Conjecture}
In this section, we give an idea of Marques and Neves' proof of the Willmore Conjecture.
See~\cite{mane1} for the actual paper and \cite{mane2} for Marques and Neves'
complete survey of this problem
and related questions.    Let $\Sigma\subset \mathbb S^3$ be a
 closed surface. The Willmore energy of $\Sigma$ is
\[
\mathcal W(\Sigma):=\int_\Sigma(1+H^2).
\]
The Willmore energy of the Clifford Torus is $2\pi^2$. The Willmore conjecture~\cite{will1}
states that the Clifford Torus minimizes the Willmore energy among tori,
and thus, in this sense, is the best torus. Namely,
\begin{conjecture}
Let $\Sigma\subset \mathbb S^3$ be a closed surface of genus one. Then
\[
\mathcal W(\Sigma)\geq 2\pi^2.
\]
\end{conjecture}
An important feature of the Willmore energy is that it is
conformally invariant. Namely, given $v\in \B^4$
(unit ball in $\R^4$),
consider the
centered dilation of $\esf^3$ that fixes $v/|v|$ and $-v/|v|$, namely
\begin{equation}\label{confmap}
F_v\colon \mathbb S^3\to\mathbb S^3, \quad F_v(x)=\frac{1-|v|^2}{|x-v|^2}(x-v)-v.
\end{equation}
Then for any $v\in \B^4$,  $\mathcal W(F_v(\Sigma))=\mathcal W(\Sigma)$.

%

In their seminal paper~\cite{mane1}, Marques and Neves proved the Willmore
conjecture. Before their very elegant proof, several results and techniques
have been used to understand the Willmore Conjecture, see for
instance~\cite{baku,br1,br4,chenby,chela,rr1,ros4,shita,si5,top,urb1,will2}.
We refer the interested reader to Section 2 of~\cite{mane2} and the references
therein for a more comprehensive and detailed discussion.

We now give a sketch of their proof. A key tool in their proof
are techniques which come from the min-max principle which is used to find
unstable critical points of a given functional.
See Section~\ref{subsecJacobi} for a discussion on stability and index of a
minimal surface.
 In~\cite{alm4}, Almgren developed a min-max theory for the area functional.
Let $\mathcal Z_2(\mathbb S^3)$ denote the space of
integral 2-currents with zero boundary and let $I^k=[0,1]^k$ denote the $k$-dimensional
cube. Given a continuous function $\Phi\colon I^k\to \mathcal Z_2(\mathbb S^3)$,
let $[\Phi]$ denote the set of all continuous functions from $I^k$ to $\mathcal Z_2(\mathbb S^3)$
that are homotopic to $\Phi$ through homotopies that fix the functions on $\partial I^k$. Let
\[
{\bf L}([\Phi]):=\inf_{\Psi\in[\Phi]} \sup_{x\in I^k} \text{\rm Area}(\Psi(x)).
\]
In~\cite{pi1} Pitts proved the following theorem.

\begin{theorem}
If ${\bf L}([\Phi])>\sup_{x\in I^k} \text{\rm Area}(\Phi(x))$ then there exists a
disjoint collection of smooth, closed, embedded minimal
surfaces $\Sigma_1,\dots,\Sigma_N$ in $\mathbb S^3$ such that
\[
{\bf L}([\Phi])=\sum_{i=1}^N m_i \;
\text{\rm Area}(\Sigma_i),
\]
for some positive integers multiplicities $m_1,\dots, m_N$.
\end{theorem}

Let $F_v$ be the conformal map defined in equation~\eqref{confmap}. Given an
embedded surface $S=\partial \Omega$, where $\Omega$ is a region of $\mathbb S^3$,
and using the ambient distance, let
\[
S_t:=\partial \{x\in \mathbb S^3: d(x,\Omega)\leq t \}\quad \text{if } t\in [0,\pi]
\]
and
\[
S_t:=\partial \{x\in \mathbb S^3: d(x,\mathbb S^3\setminus \Omega)\leq -t \}\quad \text{if } t\in [-\pi ,0].
\]
Given an embedded compact surface $\Sigma\subset\mathbb S^3$, Marques and Neves
define a five parameters deformation of $\Sigma$, $\{\Sigma_{(v,t)}\}_{(v,t)\in \B^4\times [-\pi,\pi]}$,
called the canonical family, where
\[
\Sigma_{(v,t)}:=(F_v(\Sigma))_t\in\mathcal Z_2(\mathbb S^3)
\]

Thanks to a result of Ros~\cite{ros4} that was inspirational to their approach and the fact
that the Willmore energy is conformally invariant, it follows that
\[
\text{\rm Area}(\Sigma_{(v,t)})\leq \mathcal W (F_v(\Sigma))
=\mathcal W(\Sigma), \text{ for all } (v,t)\in \B^4\times [-\pi,\pi].
\]

The idea is to prove that the min-max principle applied to the homotopy
class of the canonical family of a closed surface of genus one produces
the Clifford Torus. If that is the case, then
\[
2\pi^2=\text{\rm Area(Clifford Torus)}\leq
\sup_{(v,t)\in \B^4\times [-\pi,\pi]}\text{\rm Area}(\Sigma_{(v,t)})\leq \mathcal W(\Sigma)
\]
However, the canonical family so defined is discontinuous on
$\partial \B^4=\mathbb S^3$ but Marques and Neves  were able to reparameterize the canonical
family of $\Sigma$ by a continuous map $\Phi_\Sigma\colon I^5\to \mathcal Z_2(\mathbb S^3)$
such that the image $\Phi _{\Sigma }(I^5)$ is equal to the closure of the canonical family in
$\mathcal Z_2(\mathbb S^3)$ and $\Phi_{\Sigma }$ satisfies the following properties:
\begin{itemize}
\item[(A)] $\sup_{x\in I^5} \text{\rm Area}(\Phi_\Sigma(x))\leq \mathcal W(\Sigma)$.
\item[(B)] $\Phi_\Sigma(x,0)=\Phi_\Sigma(x,1)=0$ for all $x\in I^4$.
\item[(C)] For any $x\in \partial I^4$
there exists $Q(x)\in\mathbb S^3$ such
that $\Phi_\Sigma(x,t)$ is a sphere of radius $\pi t$ centered at $Q(x)$ for any $t\in I$.
\item[(D)] The degree of $Q\colon \mathbb S^3\to\mathbb S^3$ is equal to the
genus of $\Sigma$ and hence it is non-zero if the genus is not zero.
\end{itemize}

The definition of $Q$ can be found in the papers~\cite{mane1,mane2}. Without going into the
details, 
when the genus of $\Sigma$ is not zero, Property (D) guarantees that ${\bf L}([\Phi_\Sigma])> 4\pi$,
where ${\bf L}$ and $[\Phi_\Sigma]$ are defined in the previous discussion about the min-max principle.
Given that,  Marques and Neves  apply the min-max argument to prove that there exists a closed minimal
surface $\widehat \Sigma$ with $\text{\rm Area}(\widehat \Sigma)= {\bf L}([\Phi_\Sigma])$.
If $\text{\rm Area}(\widehat \Sigma)\geq 8\pi$ then
 \[
 \mathcal W(\Sigma)\geq \text{\rm Area}(\widehat \Sigma)\geq 8\pi >2\pi ^2
 \]
 and the conjecture holds. Thus  $\text{\rm Area}(\widehat \Sigma)< 8\pi$ which gives that
 $\Sigma$ is connected because the area of a closed minimal surface in $\mathbb S^3$ is at least $4\pi$.

By the nature of the deformation, it is natural to expect that the index of
$\widehat \Sigma$ is 5. If this is the case then using a theorem by Urbano~\cite{urb1} gives that $\widehat \Sigma$ must be a Clifford torus.
It is important to notice that Almgren-Pitts Theory does not give an estimate on the
Morse index of $\widehat \Sigma$. Indeed, Marques and Neves were able to prove that
the index of $\widehat \Sigma$ is 5 and therefore $\widehat \Sigma$ is a Clifford Torus.
Therefore, for any closed surface $\Sigma$ in $\mathbb S^3$ of genus
at least one, we have that
\[
\mathcal W(\Sigma)\geq 2\pi^2.
\]
The fact that when equality holds then $\Sigma$ is a Clifford Torus requires extra work
and we refer the interested reader to Marques and Neves' papers for the argument.

\subsection{Minimal surfaces in $\esf^2\times\R$ foliated by circles: the classification theorem
of Hauswirth, Kilian and Schmidt.}

In~\cite{mr10} Meeks and Rosenberg studied the geometry of complete minimal annuli in
Riemannian manifolds that can be expressed as
the Riemannian product $M\times \R$ of a closed Riemannian surface $M$
with $\R$.  In the case that $M$ is a sphere with a metric of positive Gaussian curvature,
they proved that any complete embedded minimal annulus $\S$ in  $M\times \R$  is properly embedded
with bounded second fundamental form  and
$\S$ intersects each level set sphere $M\times\{ t\}, t\in \R,$ in
a simple closed curve.  In this case they also showed that the moduli space of
such minimal annuli with a fixed vertical flux
is compact, which is a useful property since it implies that
any sequence of vertical translations of $\S$ has a convergent
subsequence.  In particular, if $M=\esf^2$ is the sphere of constant Gaussian curvature
$1$, then $\S\subset \esf^2\times \R$ is properly embedded and intersects each level set
sphere in a round circle. They also described in that paper
a 1-parameter family $\mathcal{A}$ {of periodic minimal annuli in
$ \esf^2\times \R$ and each surface in this family}
intersects each level set sphere of $\esf^2\times \R$  in a circle; this family first described
by Hauswirth in~\cite{hau1} is similar
to the Riemann family of properly embedded minimal planar domains in $\rth$.

In fact
Hauswirth~\cite{hau1} was able to construct a Jacobi function similar to the Shiffman function,
an indispensable tool used by Meeks, P\'erez and Ros in their proof
of Theorem~\ref{classthm}, the proof of which depended upon
methods in the theory of integrable systems.
Also using methods from the theory of integrable systems,  Hauswirth,
Kilian and Schmidt~\cite{hauks1} recently proved the following classification result
for complete embedded minimal annuli in  $\esf^2\times \R$. For some related classification results
for strongly Alexandrov embedded minimal annuli in the 3-sphere
see~Kilian and Schmidt~\cite{KiSch1}.

\begin{theorem}[Hauswirth, Kilian, Schmidt~\cite{hauks1}]
Every complete embedded minimal annulus in   $\esf^2\times \R$ lies in the family
$\mathcal{A}$.  In particular, each such minimal annulus $\S$ intersects
every level set sphere in  $\esf^2\times \R$
in a circle, $\S$ is  invariant under a reflection in
a vertical totally geodesic annulus and $\S$ is periodic under a vertical translation
that is the composition of two  rotational symmetries around  circles that
are great circles in level set spheres.
\end{theorem}

\section{Limits of $H$-surfaces without local area or curvature bounds.}
\label{sec4}
Two central problems in the classical theory of $H$-surfaces are to
understand the possible geometries or shapes of those $H$-surfaces
in $\R^3$ that have finite genus, as well as the structure of limits of sequences of
$H$-surfaces with fixed finite genus. The classical theory deals with these limits when
the sequence has uniform local area and curvature bounds, since in this case
one can reduce the problem of taking limits of $H$-surfaces
to taking limits of $H$-graphs (for
this reduction, one uses the local curvature bound in order to
express the surfaces as local graphs of uniform size, and the local
area bound to constrain locally the number of such graphs to a fixed
finite number). In this graphical framework, the existence and properties of limits is given by
the classical Arzel\`{a}-Ascoli
theorem, see e.g., \cite{pro2}. Hence, we will concentrate here on
the case where we do not have such estimates.

The starting point consists of analyzing local structure of a sequence of compact $0$-surfaces $\Sigma _n$
with fixed finite genus in a fixed extrinsic ball in $\R^3$, which is
an issue first tackled by Colding and Minicozzi in a
series of papers where they study the structure of a sequence of
compact 0-surfaces $\Sigma _n$ with fixed genus but without
area or curvature bounds in balls $\B (R_n)=\B (\vec{0},R_n)\subset \R^3$, whose radii $R_n$ either remain
bounded or tend to infinity as $n\to \infty $ and with boundaries $\partial \Sigma _n\subset
\partial \B _n$, $n\in \N$. These two possibilities on $R_n$ lead to very different
situations for the limit object of (a subsequence of) the $\Sigma _n$, as
we will explain below.  Generalizations these results to the $(H>0)$-setting
by Meeks and Tinaglia will also be discussed.

\subsection{Colding-Minicozzi theory for $0$-surfaces and generalizations by
Meeks and Tinaglia to $H$-surfaces.}
\label{subsecCM}
As we indicated above, a main goal of the Colding-Minicozzi theory,
as adapted by Meeks and Tinaglia, is to
understand the limit objects for a sequence of compact $H_n$-surfaces
$\Sigma _n$ with fixed genus
but not {\it a priori} area or curvature bounds,
each one with compact boundary contained in the boundary sphere of $\B (R_n)$.
Typically, one finds {\it weak $H$-laminations} (see Definition~\ref{deflamination}
for the concept of weak $H$-lamination) possibly with singularities
as limits of subsequences of the $\Sigma _n$. Nevertheless,
we will see in Theorems~\ref{thmlimitlaminCM}
 and~\ref{t:t5.1CM}, that the behavior of the limit lamination changes dramatically
depending on whether $R_n$ diverges to $\infty $ or stays bounded,
among others, in the following two aspects:
\begin{enumerate}[(I)]
\item The limit lamination might develop removable (case $R_n\to \infty $)
or essential singularities (case $R_n$ bounded).
\item The leaves of the lamination are proper (case $R_n\to \infty $) or might fail to have
this property (case $R_n$ bounded).
\end{enumerate}
These two phenomena connect with major open problems in the current state of the theory of $H$-surfaces:
\begin{enumerate}[(I)']
\item Finding removable singularity results for weak $H$-laminations (or equivalently,
finding extension theorems for weak $H$-laminations defined outside of a small set). In this line,
we can mention the work by Meeks, P\'erez and Ros~\cite{mpr21,mpr10},
see also Sections~\ref{LRST} and \ref{sec:CMC} below.
\item Finding conditions under which a complete $H$-surface must be proper
({\it embedded Calabi-Yau problem}), which we will treat further
in Sections~\ref{ttttmr} and \ref{sec:MT}.
\end{enumerate}

Coming back to the Colding-Minicozzi theory, the most important goal is to understand the
shape of a 0-disk $\Sigma $ in $\R^3$ depending on its Gaussian curvature. Roughly speaking,
only two models are possible: if the curvature of $\Sigma $ is everywhere small, then $\Sigma $
is a graph (with the plane as a model); and if the Gaussian
curvature of $\Sigma $ is large at some point, then
$\Sigma $ consists of two multi-valued graphs pieced
together (this is called a {\it double staircase, }
whose model is the helicoid; see~\cite{tin1,tin2} for related generalizations). Multi-valued graphs $\Sigma _g\subset \Sigma $ are subsets such that
every point in $\Sigma _g$ has a
neighborhood which is a graph over its projection
over the plane $\{ z=0\} $ (up to a rotation),
but the global projection from $\Sigma $ to $\{ z=0\} $ fails
to be one-to-one; for a precise definition of a $k$-valued graph, see Definition~\ref{Ngraph}.
For instance, a vertical helicoid can be thought as two $\infty $-valued graphs
joined along the vertical axis.

The study of $H$-multi-valued graphs, i.e., multi-valued graphs with constant mean curvature $H$,
relies heavily on PDE techniques, some of which aspects
we will comment next.
Since the third component of the Gauss map on a $H$-multi-valued
graph $\Sigma $ as in Definition~\ref{Ngraph} is a positive Jacobi function,
then $\Sigma $ is stable, and thus it has
curvature estimates away from its boundary by Theorem~\ref{thmcurvestimstable}.
Also in this case the separation $w$ given by equation (\ref{eq:sep})
is a difference between two solutions of the same mean curvature $H$ equation, thus $w$ satisfies a
second order elliptic, partial differential equation in divergence form:
\begin{equation}
\label{eq:PDEseparat}
\Div (\mathcal{A}\nabla w)=0,
\end{equation}
where $\mathcal{A}$ is a smooth map valued in the space of real symmetric, positive definite
$2\times 2$ matrices. Furthermore,
the eigenvalues $\l _i$ of $\mathcal{A}$ satisfy
\begin{equation}
\label{eq:PDEseparat1}
0<\mu \leq \l_i\leq 1/\mu ,
\end{equation}
where the constant $\mu $ only depends on an upper bound for the gradient $|\nabla u|$
(in particular, (\ref{eq:PDEseparat}) resembles the Laplace equation if $|\nabla u|$
is extremely small). A consequence of (\ref{eq:PDEseparat})-(\ref{eq:PDEseparat1}) is
that $w$ satisfies a Harnack-type inequality:
\[
\sup _{\Omega '}w\leq C\inf _{\Omega '}w
\]
whenever $\Omega '$ has compact closure in the domain of $w$ for some
constant $C>0$ depending solely on
an upper bound for $|\nabla u|$ and on the distance from
$\Omega '$ to the boundary of the domain of $w$.

With these preliminaries at hand, the statement of the so-called
{\it Limit Lamination Theorem for 0-Disks} (Theorem~0.1
of~\cite{cm23}) is easy to understand (see also~\cite{bt1} for related
results on the topology and geometry of leaves of a lamination obtained as
limit of a sequence of 0-disks); to make this statement optimal in this
0-disk setting, we will make use of the result by Meeks~\cite{me25}
(see Theorem~\ref{thmregular} below) that the singular set
$\mathcal{S}$ in Theorem~\ref{thmlimitlaminCM}
is a vertical line instead of a  Lipschitz curve parameterized by its height, as given in~\cite{cm23}.
It is worth mentioning that the proof of Theorem~\ref{thmregular} uses the
uniqueness of the helicoid among properly embedded non-flat $0$-surfaces in $\R^3$,
whose proof in turn depends on the original statement of the Limit Lamination Theorem for 0-Disks with a
Lipschitz curve as singular set $\mathcal{S}$.
Recently, Meeks and Tinaglia~\cite{mt13} generalized the
Limit Lamination Theorem for Disks
to the case that the surfaces are $H$-disks.  Hence, we state
this theorem in the general $H$-setting.

\begin{theorem}[Limit Lamination Theorem for $H$-Disks]
\label{thmlimitlaminCM}
Let $\Sigma _n\subset \B(R_n)$ be a sequence of  $H_n$-disks
with $\vec{0}\in\Sigma _n$,
$\partial \Sigma _n\subset \partial \B(R_n)$, $R_n\to \infty $ and
$\sup |A_{\Sigma _n}(\vec{0})|\to \infty $.
Let $\cS$ denote the $x_3$-axis.
Then, there exists a subsequence
$\Sigma _n$
(denoted in the same way) and numbers  $\wh{R}_n\to \infty$
 such that up to a
rotation of $\R^3$ fixing $\vec{0}$:
\begin{enumerate}[\rm 1.]
    \item Each $\Sigma _n\cap \B(\wh{R}_n)$ consists of exactly two multi-valued
    graphs away from $\mathcal{S}$, which  spiral together.
    \item For each $\a \in (0,1)$, the surfaces $\Sigma _n-\mathcal{S}$ converge
    in the $C^{\a }$-topology to the foliation
${\mathcal F}=\{ x_3=t\} _{t\in \R }$ of $\R^3$  by horizontal planes.
    \item If ${\cal S}(t)=(0,0,t)$ for every $t\in \R $, then given $t\in \R $ and $r>0$,
\[
\sup _{\Sigma _n\cap \B(\mathcal{S}(t),r)}|A_{\Sigma _n}|\to \infty \quad \mbox{as $n\to \infty $.}
\]
\end{enumerate}
\end{theorem}
Items 1 and 2 in the statement of Theorem~\ref{thmlimitlaminCM} mean that for every compact subset
$C\subset \R^3-\mathcal{S}$ and for every
$n\in \N $ sufficiently large depending on $C$, the surface $\Sigma _n\cap C$
consists of multi-valued graphs
over a portion of $\{ x_3=0\} $, and the sequence $\{ \Sigma _n\cap C\} _n$
converges to $\mathcal{F}\cap C$
 as graphs in the $C^{\a }$-topology. Item~3 deals with the behavior of
the sequence along the singular set of convergence.

The basic example to visualize Theorem~\ref{thmlimitlaminCM} is a sequence of rescaled
helicoids $\Sigma _n=\l_nH = \{ \l_nx \ | \ x\in H \}$,
where $H$ is a fixed vertical helicoid with axis the $x_3$-axis
and $\l_n>0$, $\l_n \searrow 0$. The Gaussian curvature of $\{ \Sigma _n\} _n$ blows up
along the $x_3$-axis and the $\Sigma _n$ converge away from the axis to the foliation ${\mathcal F}$
of $\R^3$ by horizontal planes. The $x_3$-axis ${\cal S}$ is the singular set of
 $C^1$-convergence of the $\Sigma _n$ to ${\mathcal F}$;
i.e., the $\Sigma _n$ do not converge $C^1$ to the leaves of ${\mathcal F}$ along the $x_3$-axis.
Finally, each leaf $L$ of ${\mathcal F}$ extends smoothly across $L\cap {\mathcal S}$;
($S$ consists of removable singularities of ${\mathcal F}$).

The proof of Theorem~\ref{thmlimitlaminCM} in the $H=0$ setting is
involved and runs along several highly demanding
papers~\cite{cm26,cm21,cm22,cm24,cm23}. Since a through sketch of this proof
is provided in Chapter~4 of~\cite{mpe10},
we do not provide any further details here. However, before
going on to the next section of this survey,
it is worthwhile to highlight one of the crucial ingredients of this proof when $H=0$,
namely the scale invariant  {\it 1-sided curvature estimate for 0-disks} by Colding, Minicozzi~\cite{cm23}.
We remark that Meeks and Tinaglia~\cite{mt9} have proved a companion {\it 1-sided curvature estimate
for $H$-disks} that is not scale invariant, and that will be stated
in Theorem~\ref{th}.

\begin{theorem}[One-Sided Curvature Estimate for 0-Disks]
\label{thmcurvestimCM}
There exists an $\ve >0$ such that the following
holds. Given $r>0$ and a 0-disk $\Sigma \subset \B(2r)\cap \{ x_3>0\} $
with $\partial \Sigma \subset \partial \B(2r)$, then for any component $\Sigma '$ of
$\Sigma \cap \B(r)$  which intersects $\B(\ve r)$,
\begin{equation}
\label{eq:1-sided}
r\sup _{\Sigma '}|A_{\Sigma }|\leq 1.
\end{equation}
(See Figure~\ref{1-sidedCurvEstim}).
\end{theorem}
\begin{figure}
\begin{center}
  \includegraphics[width=8cm]{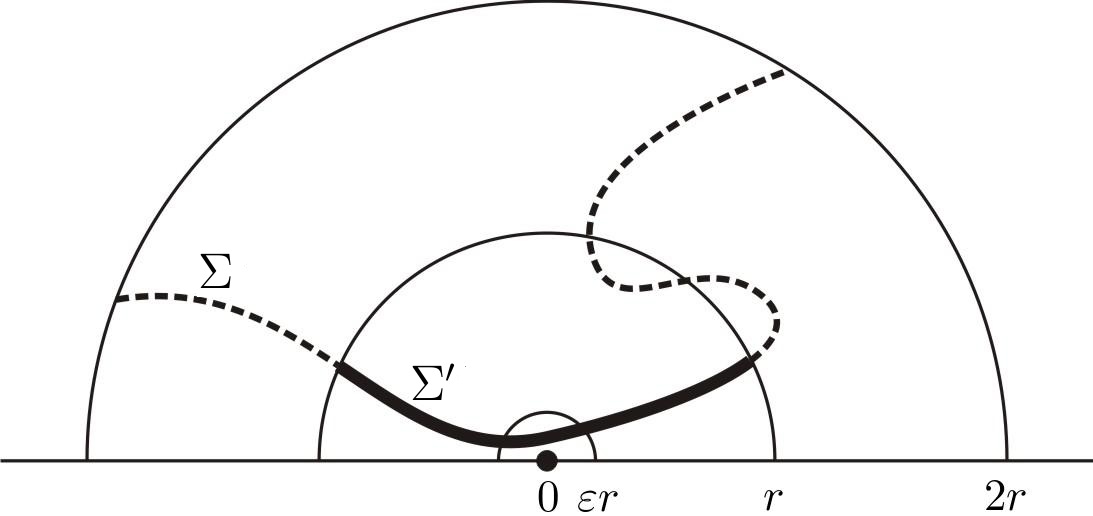}\\
  \caption{The one-sided curvature estimate.}\label{1-sidedCurvEstim}
\end{center}
\end{figure}

\section{Weak $H$-laminations, the Stable Limit Leaf Theorem and the Limit
Lamination Theorem for Finite Genus.} \label{section6}
We have mentioned the importance of understanding limits of
sequences of  $H$-surfaces with fixed (or
bounded) genus but no {\it a priori} curvature or area bounds
in a 3-manifold $N$, a situation of which
Theorem~\ref{thmlimitlaminCM}  is a particular case. This result shows that one must consider
limit objects other than $H$-surfaces, such as minimal foliations or more
generally, $H$-laminations of $N$.

In this section we start by recalling the classical notion of lamination
and discuss some results on the regularity of these objects when the
leaves have constant mean curvature. Then we will enlarge the class to admit
{\it weak laminations} by allowing certain tangential intersections between
the leaves. These weak laminations and foliations will be studied
in subsequent sections.

\begin{definition}
\label{deflamination}
{\rm
 A {\it codimension-one
lamination} of a Riemannian 3-manifold $N$ is the union of a
collection of pairwise disjoint, connected, injectively immersed
surfaces, with a certain local product structure. More precisely, it
is a pair $({\mathcal L},{\mathcal A})$ satisfying:
\begin{enumerate}[1.]
\item ${\mathcal L}$ is a closed subset of $N$;
\item ${\mathcal A}=\{ \varphi _{\be }\colon \D \times (0,1)\to
U_{\be }\} _{\be }$ is an atlas of topological coordinate charts of $N$ (here
$\D $ is the open unit disk in $\R^2$, $(0,1)$ is the open unit
interval in $\R$ and $U_{\be }$ is an open subset of $N$); note that
although $N$ is assumed to be smooth, we
only require that the regularity of the atlas (i.e., that of
its change of coordinates) is of class $C^0$; in other words, ${\cal A}$
is an atlas with respect to the topological structure of $N$.
\item For each $\be $, there exists a closed subset $C_{\be }$ of
$(0,1)$ such that $\varphi _{\be }^{-1}(U_{\be }\cap {\mathcal L})=\D \times
C_{\be}$.
\end{enumerate}

We will simply denote laminations by ${\mathcal L}$, omitting the
charts $\varphi _{\be }$ in ${\mathcal A}$ unless explicitly necessary.
A lamination ${\mathcal L}$ is said to be a {\it foliation of $N$} if ${\mathcal L}=N$.
Every lamination ${\mathcal L}$  decomposes into a
collection of disjoint, connected topological surfaces (locally given
by $\varphi_{\be }(\D \times \{ t\} )$, $t\in C_{\be }$, with the notation
above), called the {\it leaves} of ${\mathcal L}$.
Observe that if $\Delta
\subset {\cal L}$ is any collection of leaves of ${\cal L}$, then
the closure of the union of these leaves has the structure of a
lamination within ${\cal L}$, which we will call a {\it
sublamination.}

A codimension-one lamination ${\cal L}$ of $N$ is called a {\it CMC lamination}
if each of its leaves is smooth and has constant mean curvature, possibly varying
from leaf to leaf.
Given $H\in \R $, an {\it $H$-lamination} of $N$
is a CMC lamination all whose leaves have the same mean curvature
$H$. If $H=0$, the $H$-lamination will  also be called a {\it minimal
lamination}.
 }
\end{definition}

Since the leaves of a lamination ${\cal L}$ are disjoint, it makes sense to consider the
second fundamental form $A_{\cal L}$ as being defined on the union of the leaves.
A natural question to ask is whether or not the
norm $|A_{\cal L}|$  of the second fundamental form of a
(minimal, $H$- or CMC) lamination ${\cal L}$ in a Riemannian
3-manifold is locally bounded. Concerning this question, we
make the following observations.
\begin{enumerate}[(O.1)]
\item If ${\cal L}$ is a minimal lamination, then Theorem~\ref{thmcurvestimCM}
implies that $|A_{\cal L}|$
is locally bounded (to prove this, one only needs to deal
with limit leaves, where the one-sided curvature estimates apply).
\item As a consequence of recent work of Meeks and
Tinaglia~\cite{mt7,mt1,mt9} on curvature estimates for $(H>0$)-disks (Theorem~\ref{th} below),
a CMC lamination ${\cal L}$ of a Riemannian 3-manifold $N$ has
$|A_{\cal L}|$ locally bounded.
\end{enumerate}

Given a
sequence of CMC laminations ${\cal L}_n$ of a Riemannian
3-manifold  $N$ with uniformly bounded second fundamental form
on compact subsets of $N$, a simple application of the uniform graph lemma
for surfaces with constant mean curvature (see Colding and Minicozzi~\cite{cmCourant}
or P\'erez and Ros~\cite{pro2} from where this well-known result can be deduced)
and of the Arzel\`a-Ascoli Theorem, gives that there exists
a limit object of (a subsequence of) the ${\cal L}_n$,
which in general fails to be a CMC lamination since two ``leaves'' of this limit
object could intersect tangentially with mean curvature vectors pointing in opposite directions;
nevertheless, if ${\cal L}_n$ is a 0-lamination for every $n$, then
the maximum principle ensures that the limit object is indeed a 0-lamination,
see Proposition~B1 in~\cite{cm23}. Still, in the general case of
CMC laminations, such a limit object always
satisfies the conditions in the next definition.

\begin{definition}
\label{definition}
{\rm A (codimension-one) {\it weak CMC lamination} ${\cal L}$ of a
Riemannian 3-manifold $N$ is a collection
$\{ L_\alpha\}_{\alpha\in I}$ of (not necessarily
injectively) immersed constant mean curvature surfaces,
called the {\it leaves} of ${\cal L}$,
satisfying the following three properties.
\begin{enumerate}[1.]
\item $\bigcup_{\alpha\in I}L_{\alpha }$ is a closed
subset of $N$.
\item If $p\in N$ is a point where either two leaves of
${\cal L}$ intersect or a leaf of ${\cal L}$ intersects itself, then
each of these local surfaces at $p$ lies at one side of the other
(this cannot happen if both of the intersecting leaves have the same
signed mean curvature as graphs over their common tangent space at
$p$, by the maximum principle).
\item The function $|A_{\cal L}|\colon {\cal L}\to [0,\infty )$ given by
\begin{equation}
\label{eq:sigma}
|A_{\cal L}|(p)=\sup \{ |A_L|(p)\ | \ L \mbox{ is a leaf of ${\cal L}$ with $p\in L$} \} .
\end{equation}
is uniformly bounded on compact sets of~$N$.
\end{enumerate}
Furthermore:
\begin{itemize}
\item If $N=\bigcup _{\alpha } L_{\a }$, then we call ${\cal
L}$ a {\it weak CMC foliation} of $N$.
\item If the leaves of ${\cal L}$ have the same constant mean curvature
$H$, then we call ${\cal L}$ a {\it weak $H$-lamination} of $N$ (or
$H$-foliation, if additionally $N=\bigcup _{\alpha } L_{\a }$).
\end{itemize}
}
\end{definition}

The following proposition follows immediately from the definition of
a weak $H$-lamination and from Theorem~\ref{thmintmaxprinb}.

\begin{proposition}
\label{prop10.2}
Let ${\cal L}$ be a weak $H$-lamination
of a 3-manifold $N$. Then ${\cal L}$ has a local $H$-lamination
structure on the mean convex side of each leaf.  More precisely,
given a leaf $L_{\a }$ of ${\cal L}$ and a small disk $\Delta
\subset L_{\alpha }$, there exists an $\ve >0$ such that if $(q,t)$
denotes the normal coordinates for $\exp _q(t\eta _q)$ (here $\exp $
is the exponential map of $N$ and $\eta $ is the unit normal vector
field to $L_{\a }$ pointing to the mean convex side of $L_{\a }$),
then the exponential map $\exp $ is an injective submersion in
$U(\Delta ,\ve ):= \{ (q,t) \ | \ q\in \mbox{\rm Int}(\Delta ), \, t\in
(-\ve ,\ve )\} $, and the inverse image $\exp^{-1}({\cal L})\cap \{
q\in \mbox{\rm Int}(\Delta ), \,t\in [0,\ve )\} $ is an $H$-lamination of
$U(\Delta ,\ve $) in the pulled back metric, see
Figure~\ref{figHlamin}.
\begin{figure}
\begin{center}
\includegraphics[width=5.1cm,height=4cm]{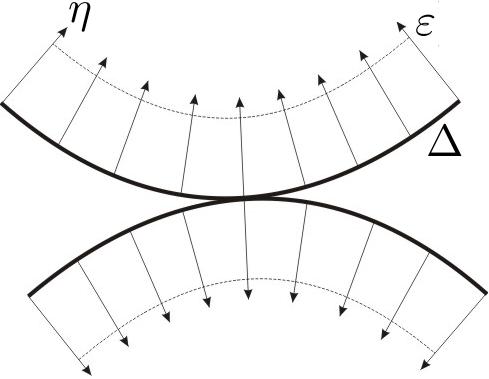}
\caption{The leaves of a weak $H$-lamination with $H\neq 0$ can
intersect each other or themselves, but only tangentially
with opposite mean curvature vectors. Nevertheless,
on the mean convex side of these locally intersecting leaves,
there is a lamination structure.
}
\label{figHlamin}
\end{center}
\end{figure}
\end{proposition}

\begin{definition}
\label{def-limset}
{\rm
Let $M$ be a complete, embedded surface in a Riemannian 3-manifold $N$. A
point $p\in N$ is a {\it limit point} of $M$ if there exists a
sequence $\{p_n\}_n\subset M$ which diverges to infinity in $M$ with
respect to the intrinsic Riemannian topology on $M$ but converges in
$N$ to $p$ as $n\to \infty$. Let lim$(M)$ denote the set of all limit
points of $M$ in $N$; we call this set the {\it limit set of $M$}.
In particular, lim$(M)$ is a closed subset of $N$ and $\overline{M} -M
\subset \lim (M)$, where $\overline{M}$ denotes the closure of~$M$.

The above notion of limit point can be extended to the case of
a lamination ${\cal L}$ of $N$ as follows:
A point $p\in \mathcal{L}$ is a {\it limit point} if there exists a coordinate chart
$\varphi _{\beta }\colon \D \times (0,1)\to U_{\be }$ as in Definition~\ref{deflamination}
such that $p\in U_{\be }$ and $\varphi _{\be }^{-1}(p)=(x,t)$ with
$t$ belonging to the accumulation set of $C_{\be }$.
The notion of limit point can be also extended to the case of a weak $H$-lamination of $N$,
by using that such an weak $H$-lamination has a local lamination structure at the mean convex side
of any of its points, given by Proposition~\ref{prop10.2}.

It is not difficult to show that if $p$ is a limit point of
a lamination ${\cal L}$ (resp. of a weak $H$-lamination), then the leaf $L$ of ${\cal L}$
passing through $p$ consists entirely of limit points of ${\cal L}$;
in this case, $L$ is called a {\it limit leaf} of ${\cal L}$.}
\end{definition}

The following result, called the {\it Stable Limit Leaf Theorem,}
concerns the behavior of limit leaves for a weak $H$-lamination.
\begin{theorem}[Meeks, P\'erez, Ros~\cite{mpr19,mpr18}]
\label{thmstable}
Any limit leaf $L$ of a codimension-one weak $H$-lamination of a Riemannian
manifold is stable for the Jacobi
operator defined in equation~(\ref{Jacobiop}).
More strongly, every two-sided cover of such a limit leaf $L$ is stable. Therefore,
the collection of stable leaves of a weak $H$-lamination ${\mathcal L}$
has the structure of a sublamination containing
all the limit leaves of ${\mathcal L}$.
\end{theorem}

We next return to discuss more aspects related to the Limit Lamination Theorem for $H$-Disks
(Theorem~\ref{thmlimitlaminCM}). The limit object in that result
is an example of a limiting {\it parking garage structure}
on $\rth$ with one column, see the next to last paragraph before Theorem~\ref{thmregular}
for a description of the notion of minimal parking garage structure.
We will find again a limiting parking garage structure
in Theorem~\ref{t:t5.1CM} below, but with two columns instead of one.
In a parking garage structure one can travel quickly up and down
the horizontal levels of the limiting surfaces
only along the (helicoidal\footnote{See Remark~\ref{rem10.3.2}.}) columns,
in much the same way that some parking garages are
configured for traffic flow; hence, the name parking garage structure.
We will study these structures in Section~\ref{subsecregularitysingular}.

Theorem~\ref{thmlimitlaminCM} deals with limits of sequences of $H$-disks, but
it  is also useful when studying more general situations, as
 for instance, locally simply connected sequences of $H$-surfaces, a notion
which we now define.

\begin{definition}
\label{defULSC}
{\rm
Suppose that $\{ M _n\} _n$ is a sequence of
$H_n$-surfaces (possibly with boundary) in an open set $U$ of $\R^3$.
If for any $p\in U$ there exists a number $r(p)>0$ such that $\B (p,r(p))\subset U$ and
for $n$ sufficiently large,
$M_n$ intersects $\B (p,r(p))$ in compact disks whose boundaries lie
on $\partial \B (p,r(p))$, then we say that $\{ M_n\} _n$ is  {\it
locally simply connected in $U$.} If  $\{ M_n\} _n$ is a locally
simply connected sequence in $U=\R^3$ and the positive number $r(p)$
can be chosen independently of $p\in \R^3$, then we say that $\{
M_n\} _n$ is {\it uniformly locally simply connected}.}
\end{definition}

There is a subtle difference between our definition of uniformly
locally simply connected and that of Colding and
Minicozzi~\cite{cm25}, which may lead to some confusion. Colding and
Minicozzi define a sequence
 $\{ M_n\} _n$ to be
uniformly locally simply connected in an open set $U\subset \R^3$ if
for any compact set $K\subset U$, the number $r(p)$ in
Definition~\ref{defULSC} can be chosen independently of $p\in K$. It
is not difficult to check that
 this concept coincides
with our definition of a locally simply connected sequence in $U$.

The Limit Lamination Theorem for $H$-Disks
(Theorem~\ref{thmlimitlaminCM}) admits a generalization to a locally
simply connected sequence of non-simply connected $H_n$-planar domains passing through
the origin and having unbounded curvature at the origin,
which we now
explain since it will be useful for our goal of classifying minimal planar domains.
Instead of the scaled-down limit of the helicoid, the basic example in this case
is an appropriate scaled-down limit of Riemann
minimal examples ${\cal R}_t$, $t>0$. To understand this limit,
normalize each Riemann minimal example ${\cal R}_t$ so that
${\cal R}_t$ is symmetric by reflection in the $(x_1,x_3)$-plane $\Pi $ and
the {\it flux} $F(t)$ of ${\cal R}_t$, which is the flux vector
along any compact horizontal section ${\cal R}_t \cap \{x_3 = {\mathrm{constant}} \}$,
has third component equal to one.
The fact that ${\cal R}_t$ is invariant by reflection in
$\Pi $ forces $F(t)$ to be contained in $\Pi $ for each~$t>0$.
Furthermore, $t>0 \mapsto F=F(t)$
is a bijection {whose inverse map $t=t(F)$}
parameterizes the whole family of
Riemann minimal examples, with $F$ running from
horizontal to vertical (with monotonically increasing slope function).
When $F$ tends to vertical, then
it can be proved that ${\cal R}_{t(F)}$ converges to a vertical
catenoid with waist circle of radius
$\frac{1}{2\pi }$. When $F$ tends to horizontal, then one can shrink ${\cal R}_{t(F)}$
so that $F$ tends to $(4,0,0)$, and in that case the
${\cal R}_{t(F)}$ converge to the foliation of $\R^3$ by horizontal planes,
outside of the two vertical lines $\{ (0,\pm 1,x_3)\ | \ x_3\in \R \} $, along which the
{shrunk}
surface
${\cal R}_{t(F)}$ with $F$ very horizontal approximates two oppositely handed, highly sheeted,
scaled-down vertical helicoids, see Figures~\ref{figRiemannparkgar1} and \ref{ULSC1}.
\begin{figure}
\begin{center}
  \includegraphics[width=13cm]{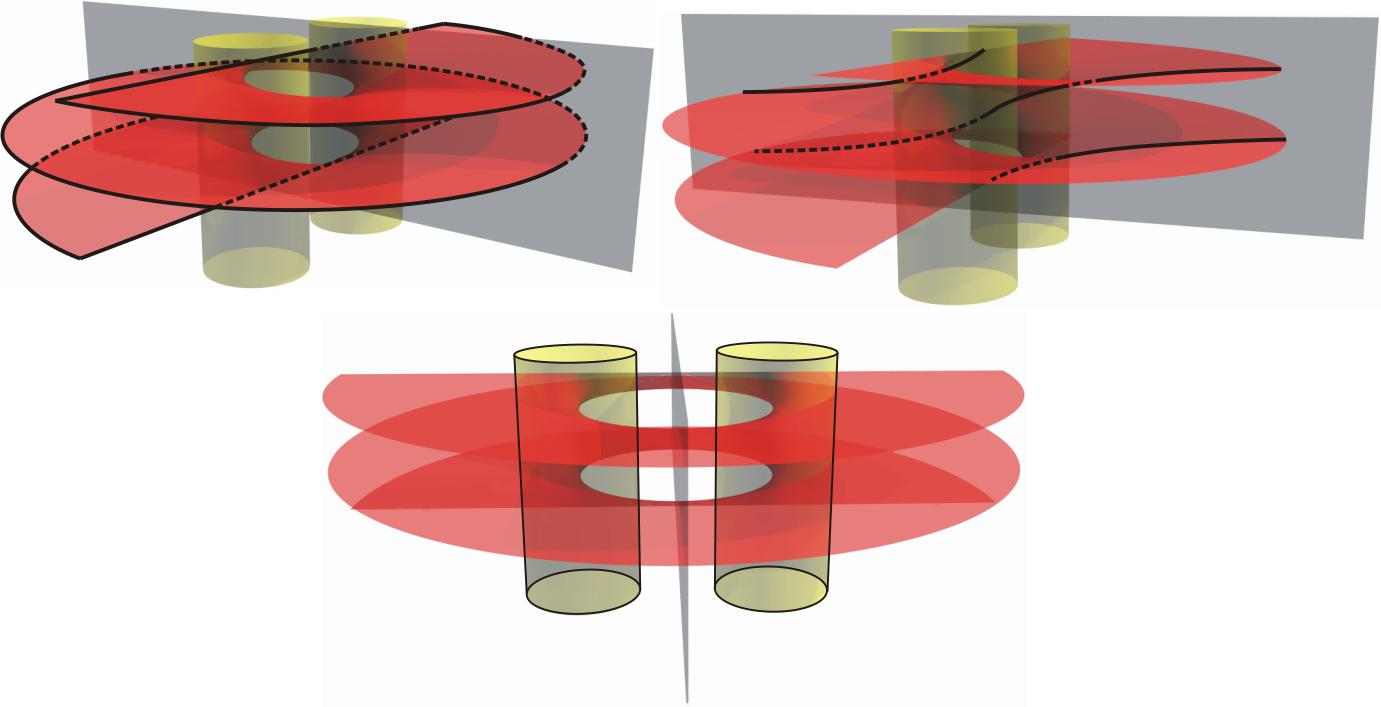}\\
  \caption{Three views of the same Riemann minimal example, with large horizontal flux
  and two oppositely handed vertical helicoids
  forming inside solid almost vertical cylinders,
  one at each side of the vertical plane of symmetry.}
\label{figRiemannparkgar1}
\end{center}
\end{figure}

With this basic family of examples in mind, we state the following
result by Colding and Minicozzi. We refer the reader to the paper~\cite{mt14}
by Meeks and Tinaglia for theorems that describe the limiting object of a
locally simply connected sequence of $H_n$-surfaces of fixed finite
genus in $\rth$ with boundaries diverging to infinity;
in particular, the reader might compare the statement of the next theorem
with the similar statement of Theorem~1.4 in~\cite{mt14}.

\begin{theorem}[Limit Lamination Theorem for $0$-surfaces of Finite Genus~\cite{cm25}]
\label{t:t5.1CM}
 Let $\S _n\subset \B(R_n)$ be a locally simply connected sequence of
$0$-surfaces of finite genus $g$, with $\partial \S_n\subset \partial \B(R_n)$,
$R_n\to \infty $, such that $\S _n\cap \B(2)$ contains a component which is not a disk
for any $n$. If $\sup| A_{\S_n\cap \B(1)}|\to \infty $ as $n\to \infty $,
then there exists a subsequence of the $\S_n$ (denoted in the same way) and
two vertical lines $S_1,S_2$, such that after a rotation in $\R^3$, then following properties hold.
\begin{enumerate}
\item Away from $S_1\cup S_2$, each $\S _n$ consists of exactly
two multi-valued graphs spiraling together.
\item For each $\a \in (0,1)$, the surfaces $\S _n-(S_1\cup S_2)$ converge in the $C^{\a }$-topology
to the foliation ${\mathcal F}$ of $\R^3$ by horizontal planes.
\item Along $S_1$ and $S_2$ the norm of the second fundamental form of the $\S _n$ blows up as $n\to \infty $.
\item The pair of multi-valued graphs appearing in item 1 inside $\S _n$ for $n$ large,
form double spiral staircases with opposite handedness at $S_1$ and $S_2$.
Thus, circling only $S_1$ or only $S_2$ results in going either up or down, while a
path circling both $S_1$ and $S_2$ closes up, see Figure~\ref{ULSC1}.
\begin{figure}
\begin{center}
  \includegraphics[width=12.5cm]{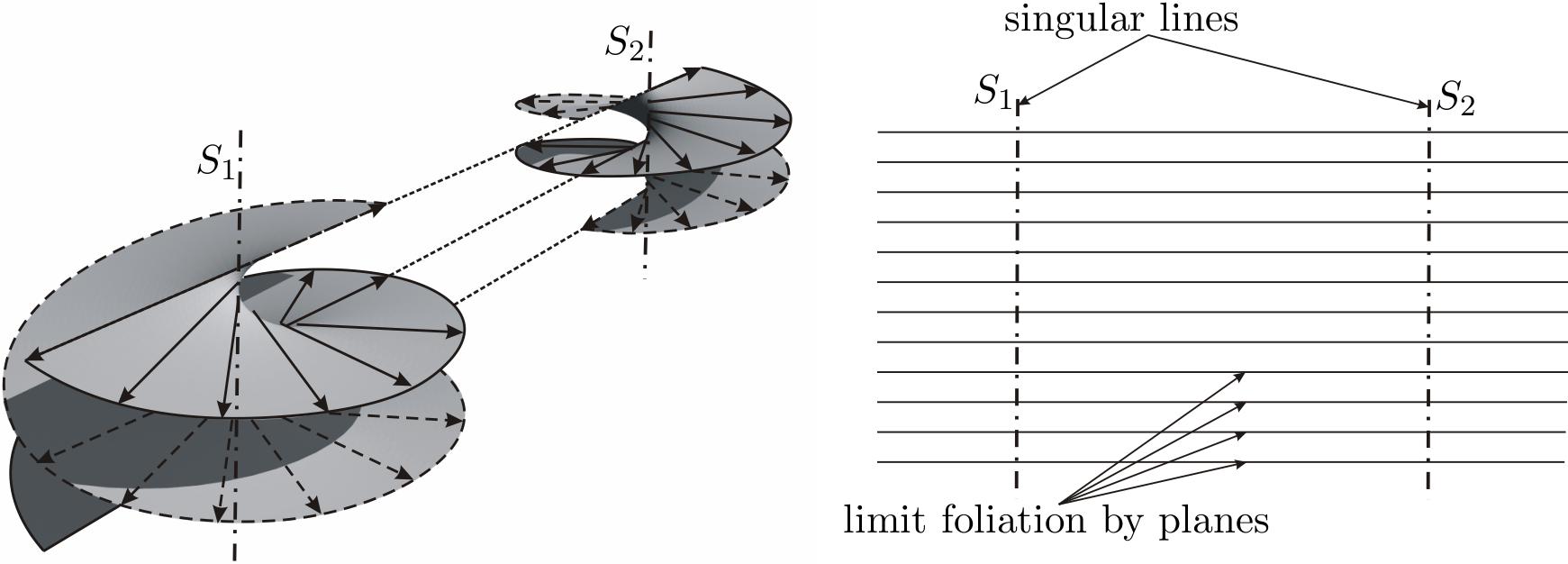}\\
  \caption{Left: Two oppositely handed double spiral staircases.
  Right: The limit foliation by parallel
 planes and the singular set of convergence $S_1\cup S_2$.}
\label{ULSC1}
\end{center}
\end{figure}
\end{enumerate}
\end{theorem}
Theorem~\ref{t:t5.1CM} gives rise to a second example of a limiting parking
garage structure on~$\rth$ (we obtained the
first example in Theorem~\ref{thmlimitlaminCM} above), now with two columns which are
$(+,-)$-handed\footnote{Here, $+$,
(resp. $-$) means that the corresponding forming helicoid or multi-valued
graph is right-handed.}, just like in the case of the Riemann
minimal examples ${\cal R}_t$ discussed before the statement
of Theorem~\ref{t:t5.1CM}. We refer the reader to Section~\ref{subsecregularitysingular}
for more details about parking garage structures on~$\rth$, and to
Theorem~\ref{tthm3introd} for a generalization to the case where there is no bound on the genus
of the surfaces {$\Sigma _n$}.

\section{Properness results for {$0$}-surfaces.}
\label{ttttmr}

In previous sections we have seen that $H$-laminations constitute a key tool in the
understanding of the  global behavior of  $H$-surfaces in $\R^3$.
In this section, we will present some results about the
existence and structure of 0-laminations,
which have deep  consequences in various aspects of 0-surface theory, including the
general 3-manifold setting.
For example, we will give a natural condition under which the closure of a
complete 0-surface in a Riemannian 3-manifold has the structure of a 0-lamination.
We will deduce from this analysis, among other things, that certain complete
0-surfaces are proper, a result which is used to prove
 the uniqueness of the helicoid in the complete setting (Theorem~\ref{helicoid})
by reducing it to the corresponding characterization assuming properness (Theorem~\ref{ttmr}).

In the Introduction we mentioned the result, proved in~\cite{mr8} by
Meeks and Rosenberg, that the closure of a complete 0-surface $M$ with locally bounded
second fundamental form
in a Riemannian 3-manifold $N$, has the structure of a 0-lamination
(they stated this result in~\cite{mr8} in the particular case $N=\R^3$,
but their proof extends to the general case for $N$ as mentioned in~\cite{mr13}).
The same authors have demonstrated that this result still holds true if we
substitute the locally bounded curvature assumption by a weaker hypothesis,
namely that for every $p\in N$, there exists a neighborhood $D_p$ of $p$ in $N$ where
the injectivity radius function $I_M$ of $M$ restricted to $M\cap D_p$ is bounded away from zero.

\begin{definition}
{\rm
Let $\Sigma $ be a complete Riemannian manifold. The {\it injectivity radius function}
$I_{\Sigma }\colon \Sigma \to (0,\infty ]$
is defined at a point $p\in \Sigma $, to be the supremum of the
radii $r>0$ of disks $D(\vec{0},r)\subset T_p\Sigma $,
such that the exponential map $\exp _p\colon T_p\S \to \S $,
restricts to $D(\vec{0},r)$ as a
diffeomorphism onto its image. The {\it injectivity radius of
$\Sigma $} is the infimum of $I_{\Sigma }$.
}
\end{definition}

The next theorem was proved by Meeks and Rosenberg~\cite{mr13}.

\begin{theorem}[Minimal Lamination Closure Theorem]
\label{thmmlct}
Let $M$ be a complete 0-surface of positive injectivity radius in a Riemannian
3-manifold $N$ (not necessarily complete). Then, the closure $\overline{M}$ of $M$ in $N$
has the structure of a 
0-lamination $\lc$, some of whose leaves are the connected components of $M$.
Furthermore:
\begin{enumerate}[1.]
\item If $N$ is homogeneously regular, then there exist $C, \ve>0$ depending on $N$ and
on the injectivity radius of $M$, such that the norm of the second fundamental form of $M$
in the $\ve$-neighborhood of any limit leaf of $\overline{M}$ is less than $C$ (recall that
limit leaves were introduced in Definition~\ref{def-limset}).
\item If $M$ is connected, then exactly one of the following three statements holds for the set
{\rm lim}$(M)\subset \lc$
of limit points of $M$:
\begin{enumerate}[2.a.]
\item $M$ is properly embedded in $N$, and {\rm lim}$(M)=\mbox{\rm \O }$.
\item {\rm lim}$(M)$ is non-empty and
disjoint from $M$, and $M$ is properly embedded in the open set $N-\mbox{\rm lim}(M)$.
\item {\rm lim}$(M)=\lc$ and ${\mathcal L}$ contains an uncountable number of leaves.
\end{enumerate}
\end{enumerate}
\end{theorem}

In the particular case $N=\rth$, more can be said. Suppose $M\subset \R^3$ is a connected,
complete 0-surface
with positive injectivity radius. By Theorem~\ref{thmmlct},  the closure of $M$ has the structure
of a 0-lamination of $\rth$. If item 2.a in Theorem~\ref{thmmlct} does not hold for $M$,
then the sublamination of limit points lim$(M)\subset \overline{M}$ contains some leaf $L$. By
Theorem~\ref{thmstable} $L$ is stable,
hence $L$ is a plane by
Theorem~\ref{thmstablecompleteplane}.
Now Theorem~\ref{thmmlct} insures that $M$ has bounded curvature
in some $\varepsilon$-neighborhood of the plane~$L$,
which contradicts Lemma~1.3 in~\cite{mr8}. This
contradiction proves the following
result.

\begin{corollary}[Meeks, Rosenberg~\cite{mr13}]
\label{corolMLCTR3}
Every connected, complete 0-surface in $\rth$ with positive
injectivity radius is properly embedded.
\end{corollary}

Suppose $M$ is a complete 0-surface of finite
topology in $\R^3$. If the injectivity radius of $M$ is zero, then
there exists a divergent sequence of embedded geodesic loops $\g
_n\subset M$ (i.e., closed geodesics with at most one corner) with
lengths going to zero. Since $M$ has finite topology, we may assume
the $\g _n$ are all contained in a fixed annular end $E$ of $M$. By
the Gauss-Bonnet formula, each $\g _n$ is homotopically non-trivial,
and so, the cycles $\g _n\cup \g _1$, $n\geq 2$, bound compact
annular subdomains in $E$, whose union is a subend of $E$. However,
the Gauss-Bonnet formula implies that the total Gaussian curvature of this
union is finite (greater than $-4\pi $). Hence, $E$ is asymptotic to
an end of a plane or of a half-catenoid, which is absurd. This
argument proves that the following result holds.

\begin{corollary}[Colding-Minicozzi~\cite{cm35}]
\label{corolthmmlctR3}
A complete  0-surface of finite topology in $\rth$ is
properly embedded.
\end{corollary}

\subsection{Regularity of the singular sets of convergence to 0-laminations.}
\label{subsecregularitysingular}
An important technique which is used when dealing with a sequence of
 0-surfaces $M_n$ is to rescale each surface in the sequence
to obtain a well-defined limit after rescaling, from where one
deduces information about the original sequence. An important case
in these rescaling processes is that of blowing up a sequence of
$H$-surfaces {\it on the scale of curvature} (for details, see
Theorem~1.1 in~\cite{mpr20},
and also see the proofs
of Theorem~15 in~\cite{mpe1} or of Corollary~2.2
in~\cite{me29}).
When the surfaces in the sequence are complete and embedded in $\R^3$,
this blowing-up process produces a limit which is a proper
 non-flat $0$-surface with bounded  second fundamental form,
whose genus and rank of homology groups are bounded above by the
ones of the $M_n$. For example, if each $M_n$ is a planar domain, then the
same property holds for the limit surface.

Recall that we defined in {Section}~\ref{subsecCM} the concept of a
locally simply connected sequence of proper 0-surfaces in $\R^3$.
This concept can be easily generalized to a
sequence of proper 0-surfaces
in a Riemannian 3-manifold\footnote{To do this, simply exchange the Euclidean balls
$\B (p,r(p))$ in Definition~\ref{defULSC} by extrinsic balls
$B_N(p,r(p))$ relative to the Riemannian distance function on $N$.}
$N$. For useful applications of the notion of locally simply
connected sequence, it is
essential to consider sequences of proper 0-surfaces
which {\it a priori} may not satisfy the locally simply connected condition, and then
modify them to produce a new sequence which satisfies that
condition. We accomplish this by considering a blow-up argument
on a scale which, in general, is different from
blowing up on the scale of curvature. We call this procedure
{\it blowing up on the scale of topology}. This scale was
defined and used in~\cite{mpr3,mpr4} to prove that any
proper 0-surface in $\rth$ of finite genus
and infinitely many ends has bounded curvature and is recurrent.
We now explain the elements of this new blow-up procedure, which is also
the basis for the proof of   Theorem~\ref{tthm3introd} below in the
general 3-manifold setting.

Suppose $\{ M_n\} _n$ is a sequence of non-simply connected,
proper 0-surfaces in $\R^3$ which is not uniformly
locally simply connected. Note that the Gaussian curvature of the
collection $M_n$ is not uniformly bounded, and so, one could blow up
these surfaces on the scale of curvature to obtain a proper,
 non-flat 0-surface which may or may not be simply
connected. Also note that, after choosing a subsequence, there exist
points $p_n\in \R^3$ such that by defining
\[
r_n(p_n)=\sup \{ r>0 \ | \ \B(p_n,r) \mbox{ intersects $M_n$ in disks} \} ,
\]
then $r_n(p_n)\to 0$ as $n\to \infty $.
Let $\widetilde{p}_n$ be a point in $\B(p_n,1)=\{ x\in \R^3 \ | \
\|x-p_n\| <1\} $ where the function
$x\mapsto d(x,\partial \B(p_n,1)) /  r_n(x)$ attains
its maximum (here $d$ denotes extrinsic distance). Then, the
translated and rescaled surfaces
$\widetilde{M}_n=\frac{1}{r_n(\widetilde{p}_n)}(M_n-\widetilde{p}_n)$
intersect for all~$n$ the closed unit ball $\overline{\B}(\vec{0},1)$
in at least one component which is
not simply connected, and for $n$ large they intersect any
ball of radius less than~1/2 in simply connected components, in particular
the sequence $\{ \widetilde{M}_n\} _n$ is uniformly locally simply
connected (see
the proof of Lemma~8 in~\cite{mpr3} for details).

For the sake of clarity, we now illustrate this blow-up
procedure on certain sequences of Riemann minimal examples,
defined in {Section}~\ref{subsecexamples}. Each of these
surfaces is foliated by circles and straight lines in
horizontal planes, with the $(x_1,x_3)$-plane as a plane of
reflective symmetry. After a translation and a homothety, we can
also assume that these surfaces are normalized so that any ball of
radius less than 1 intersects these surfaces in compact disks, and
the closed unit ball $\overline{\B } (\vec{0},1)$ intersects every
Riemann example in at least one component which is not a disk. Under
this normalization, any sequence of Riemann minimal examples is
uniformly locally simply connected. The flux of each Riemann minimal
example along a compact horizontal section has horizontal and
vertical components which are not zero; hence it makes sense to
consider the ratio $V$ of the norm of its horizontal component over
the vertical one.

As explained before
Theorem~\ref{t:t5.1CM},
$V$ parameterizes the family $\{
{\cal R}(V)\} _V $ of Riemann minimal examples,
with $V\in (0,\infty )$. When $V\to 0$,  the surfaces ${\cal R}(V)$  converge smoothly to the
vertical catenoid centered at the origin with waist circle of radius~1. But for our current
purposes, we are more interested in the limit object of ${\cal R}(V)$ as $V\to \infty $.
In this case, the Riemann minimal examples ${\cal R}(V)$
converge smoothly to a foliation of $\R^3$ by horizontal planes away
from the two vertical lines
passing through $(0,-1,0),(0,1,0)$.
Given a horizontal slab $S\subset \R^3$ of finite width,
the description of ${\cal R}(V)\cap S$ for $V$ large is as
follows, see
Figure~\ref{figRiemannparkgar1}.
\begin{enumerate}[(a)]
\item If $C_1,C_2$ are disjoint vertical cylinders in $S$ with axes $S\cap
[\{ (0,-1)\} \times \R ]$, $S\cap [\{ (0,1)\} \times \R ]$,
then ${\cal R}(V)\cap C_i$ is arbitrarily close to the
intersection of $S$ with a highly sheeted vertical helicoid with axis the axis of $C_i$, $i=1,2$.
Furthermore, the fact that the $(x_1,x_3)$-plane is a plane of reflective symmetry of
${\cal R}(V)$ implies that these limit helicoids are oppositely handed.

\item In $S-(C_1\cup C_2)$, the surface ${\cal R}(V)$ consists of two almost flat,
almost horizontal multi-valued graphs, with number of sheets increasing to $\infty $ as
$V\to \infty $.

\item If $C$ is a vertical cylinder in $S$ containing $C_1\cup C_2$, then ${\cal R}(V)$
intersects $S-C$ in a finite number $n(V)$ of univalent graphs, each one representing a
planar end of ${\cal R}(V)$. Furthermore, $n(V)\to \infty $ as
$V\to \infty $.
\end{enumerate}
This description shows an example of a particular case of what we call a
{\it parking garage structure} on $\R^3$ for the limit of a sequence of 0-surfaces
(we mentioned this notion before Definition~\ref{defULSC}).
Roughly speaking, a parking garage surface with $n$ columns is a smooth
embedded surface in $\rth$ (not necessarily minimal), which in any fixed finite width
horizontal slab $S\subset \R^3$,  can be decomposed into 2 disjoint, almost flat horizontal
multi-valued graphs over the exterior of $n$ disjoint  disks $D_1,\ldots ,D_n$ in the
$(x_1,x_2)$-plane, together $n$ topological strips each one
contained in one of the solid cylinders $S\cap (D_i\times \R )$ (these are the
{\it columns} of the parking garage structure),
such that each strip lies in a small regular neighborhood of
the intersection of a vertical helicoid $H_i$ with $S\cap (D_i\times \R )$. One can
associate to each column a $+$ or $-$ sign, depending on the
handedness of the corresponding helicoid $H_i$. Note that a vertical
helicoid is the basic example of a parking garage surface with 1 column,
and the Riemann minimal examples ${\cal R}(V)$ with $V\to \infty $ have the structure
of a parking garage structure with two oppositely handed columns. Other limiting
parking garage structures with varying
numbers of columns (including the case where $n=\infty$) and
associated signs can be found in Traizet and Weber~\cite{tw1} and in
Meeks, P\'erez and Ros~\cite{mpr14}.

There are interesting cases where the locally simply connected
condition guarantees the convergence of a sequence of 0-surfaces in
$\R^3$ to a limiting parking garage structure.
Typically, one proves that the sequence
converges (up to a subsequence and a rotation) to a foliation of $\R^3$
by horizontal planes, with singular set of convergence consisting of a
locally finite set of Lipschitz curves parameterized by heights. In fact,
these Lipschitz curves are vertical lines by Theorem~\ref{thmregular} below,
and locally around the lines the surfaces in the sequence can be arbitrarily
approximated by highly sheeted vertical helicoids, as follows from
the uniqueness of the helicoid (Theorem~\ref{ttmr}) after a blowing-up process
on the scale of curvature.
By work of Meeks and Tinaglia~\cite{mt1}, the next theorem can be generalized
to a locally simply connected sequence of $H_n$-surfaces in a Riemannian 3-manifold.

\begin{theorem}[$C^{1,1}$-regularity of $S(\mathcal{L})$, \, Meeks~\cite{me25}]
\label{thmregular}
Suppose $\{ M_n\}_n$ is a locally simply connected sequence of proper 0-surfaces in
a Riemannian 3-manifold $N$, that converges $C^{\a }$, $\a \in (0,1)$, to a
0-lamination ${\mathcal L}$ of $N$, outside a locally finite
collection of Lipschitz curves $S({\mathcal L})\subset N$
transverse to ${\mathcal L}$, along which the
Gaussian curvatures of the $M_n$ blow up and the convergence
fails to be $C^{\a }$.
Then, $\lc$ is a 
$0$-foliation in a neighborhood of $S(\lc)$, and
$S(\mathcal{L})$ consists of $C^{1,1}$-curves orthogonal to the leaves
of $\mathcal{L}$.
\end{theorem}
Next we give an idea of the proof of Theorem~\ref{thmregular}. First note that the
nature of this theorem is local, hence it suffices to consider the case of a sequence of
proper 0-disks $M_n$ in the unit ball $\B (1)=\B (\vec{0},1)$ of $\R^3$
(the general case follows from adapting the arguments to a Riemannian 3-manifold).
After passing to a subsequence, one can also assume that the surfaces $M_n$ converge to a
$C^{0,1}$-minimal foliation\footnote{Any codimension-one minimal foliation is of
class $C^{0,1}$ and its unit normal vector field is $C^{0,1}$ as well, see Solomon~\cite{sol1}.}
${\mathcal L}$ of $\B (1)$ and the convergence
is $C^{\a }$, $\a \in (0,1)$, outside of a transverse Lipschitz curve $S({\mathcal L})$ that
passes through the origin. Since unit normal vector
field $N_{\mathcal L}$ to ${\mathcal L}$ is Lipschitz
(Solomon~\cite{sol1}), then the integral curves of
$N_{\mathcal L}$ are of class $C^{1,1}$. Then, the
proof consists of demonstrating that $S({\mathcal L})$
is the integral curve of $N_{\mathcal L}$ passing through the origin. To do this,
one first proves that $S({\mathcal L})$ is of class $C^1$, hence it admits a
continuous tangent field $T$, and then one shows that $T$ is orthogonal
to the leaves of ${\mathcal L}$. These properties
rely on a local analysis of the singular set $S({\mathcal L})$ as a limit of
minimizing geodesics $\gamma _n$ in $M_n$
that join pairs of appropriately chosen {\it points of almost maximal curvature,}
(in a sense similar
to the points $p_n$ in Theorem~\ref{thm3introd} below),
together with the fact that the minimizing
geodesics $\g _n$ converge $C^1$ as $n\to \infty $ to the integral
curve of $N_{\mathcal L}$ passing through the origin.
Crucial in this proof is the uniqueness of the helicoid
(Theorem~\ref{ttmr}), since it gives the local picture of the 0-disks $M_n$ near the origin as
being closely approximated by portions of a highly sheeted helicoid near its axis.

\begin{remark}
\label{rem10.3.2}
{\rm
The local structure of the surfaces $M_n$ for $n$ large near a point in $S({\mathcal L})$
as an approximation by a highly sheeted helicoid is the reason for the parenthetical
comment of the columns being helicoidal for the surfaces limiting to a parking garage
structure, see the paragraph just after Theorem~\ref{thmstable}.
}
\end{remark}
Meeks and Weber~\cite{mwe1} have shown that the
$C^{1,1}$-regularity of $S(\lc)$ given by
Theorem~\ref{thmregular} is the best possible result. They do
this by proving the following theorem, which is based on the
{\it bent helicoids} described in Section~\ref{subsecexamples}, also
see Figure~\ref{figmobius-benthel} Right.

\begin{theorem}  [Meeks, Weber \cite{mwe1}]
\label{MeeksWeber}
Let $\Gamma$ be a properly embedded $C^{1,1}$-curve
in an open set $U$ of $\rth$. Then, $\G $ is the singular set of convergence for some
Colding-Minicozzi type limit foliation of some neighborhood of $\G$  in $U$.
\end{theorem}

\section{Local pictures, local removable singularities and dynamics.}
\label{LRST}

An important application of the Local Removable Singularity Theorem~\ref{tt2} below
is a characterization of all complete 0-surfaces in $\R^3$ of
quadratic decay of Gaussian curvature (Theorem~\ref{thm1introd} below).

Given a 3-manifold $N$ and a point $p\in N$, we will denote by
$d_N$ the distance function in $N$ to $p$ and $B_N(p,r)$ the metric
ball of center $p$ and radius $r>0$. For a lamination  $\lc$ of $N$,
we will denote by $|A_{{\mathcal L}}|$ the norm of the second fundamental form
function on the leaves of ${\mathcal L}$.
Meeks, P\'erez and Ros~\cite{mpr21} obtained the following remarkable local
removable singularity result in any Riemannian 3-manifold $N$
for certain possibly singular weak $H$-laminations.

\begin{theorem} [Local Removable Singularity Theorem]
\label{tt2}
A weak $H$-lamination $\lc$ of a punctured ball $B _N(p,r)-\{ p\} $
in a Riemannian 3-manifold $N$ extends to a weak $H$-lamination
of $B_N(p,r)$ if and only if there exists a positive constant $c$
such that $|A_{{\mathcal L}}|d_N <c$ in some subball.
\end{theorem}

\begin{remark}
{\rm
There is a natural setting where the curvature estimate
$|A_{{\mathcal L}}|d_N <c$ described in the above theorem holds;
namely, if the weak $H$-lamination  $\lc$ is a
sublamination of a CMC foliation of  a punctured ball $B _N(p,r)-\{ p\} $.
This will be discussed further in Section~\ref{sec:CMC}.
}
\end{remark}

Since  stable immersed $H$-surfaces have local curvature estimates which
satisfy the hypothesis of Theorem~\ref{tt2} and every
limit leaf of a $H$-lamination is stable (Theorem~\ref{thmstable}),
we obtain the
next extension result for the sublamination of limit leaves of any
$H$-lamination in a punctured Riemannian 3-manifold.

\begin{corollary}
\label{corrs}
Suppose that $N$ is a Riemannian 3-manifold, which is  not
necessarily complete. If $W\subset N$ is a closed countable subset\,\footnote{An
argument based on the classical
Baire's theorem allows us to pass from the isolated
singularity setting of Theorem~\ref{tt2} to a closed countable
set of singularities, see~\cite{mpr10}.}
and ${\mathcal L}$ is a weak $H$-lamination of $N-W$, then:

\begin{enumerate}[\rm 1.]
\item  The sublamination of ${\mathcal L}$ consisting of the closure
of any collection of its stable leaves extends across $W$ to a
weak $H$-lamination ${\mathcal L}_1$ of $N$. Furthermore, each leaf of
${\mathcal L}_1$ is stable.
\item   The sublamination of ${\mathcal L}$ consisting of its limit
leaves extends across $W$ to a weak $H$-lamination of $N$.
\item If ${\mathcal L}$ is a minimal foliation of $N-W$, then ${\mathcal L}$
extends across $W$ to a minimal foliation of~$N$ (this result will be generalized
in Section~\ref{sec:CMC} to the case of a CMC foliation of $N-W$, provided
that there is a bound on the mean curvature of the leaves of the foliation).
\end{enumerate}
\end{corollary}

Recall that Corollary~\ref{cor.max} ensured that every complete
$H$-surface in~$\R^3$ with bounded Gaussian curvature is properly embedded.
The next theorem by  Meeks, P\'erez and
Ros~\cite{mpr20} shows that any complete $H$-surface in $\R^3$
which is not properly embedded, has natural limits under dilations,
which are properly embedded 0-surfaces.
By {\it dilation,} we mean the composition of a homothety
and a translation.

\begin{theorem}[Local Picture on the Scale of Curvature]\label{thm3introd}

\noindent Suppose $M$ is a complete $H$-surface with unbounded second fundamental
form in a homogeneously
regular 3-manifold $N$. Then, there exists a sequence of points
$p_n\in M$ and positive numbers $\ve _n\to 0$, such that the
following statements hold.
\begin{enumerate}[\rm 1.]
\item  For all $n$, the component $M_n$ of $B _N(p_n,\ve _n)\cap M$ that
contains $p_n$ is compact, with boundary $\partial M_n\subset
\partial B _N(p_n,\ve _n)$.
\item Let $\lambda_{n} = \sqrt{|A_{M_{n}} |(p_{n})}$ (where as
usual, $|A_{\widehat{M}}|$ stands for the norm of the second
fundamental form of a surface $\widehat{M}$).
Then,
$\frac{|A_{M_n}|}{\l _n}\leq 1+\frac{1}{n}$ on $M_n$, with
$\lim_{n\to \infty }\ve_n\l_n=\infty $.
\item The rescaling of the metric balls $B _N(p_n,\ve _n)$ by factor $\l _n$
   converge uniformly to $\R^3$ with its usual metric
(so that we identify $p_n$ with $\vec{0}$ for all $n$), and there exists
a
properly embedded $0$-surface
$M_{\infty}$ in $\R^3$ with
$\vec{0}\in M_{\infty }$,
$|A_{M_{\infty}}|\leq 1$ on $M_{\infty}$ and $|A_{M_{\infty}}|(\vec{0})=1$,
such that for any $k \in \N$, the surfaces $\l _nM_n$ converge $C^k$
on compact subsets of $\rth$  to $M_{\infty}$  with
multiplicity one
as $n \rightarrow \infty$.
\end{enumerate}
\end{theorem}

The above theorem gives a {\it local picture} or description of
the local geometry of an $H$-surface $M$ in an extrinsic
neighborhood of a point $p_n\in M$ of concentrated curvature.
The points $p_n\in M$ appearing in Theorem~\ref{thm3introd}
are called {\it blow-up points on the scale of curvature.}

Now assume that for any positive $\ve $, the intrinsic $\ve $-balls
of such an $H$-surface $M$ are not always disks. Then, the
curvature of $M$ certainly blows up at some points in these
non-simply connected intrinsic $\ve$-balls as $\ve \to 0$. Thus, one
could blow up $M$ on the scale of curvature as in
Theorem~\ref{thm3introd} but this process would create a simply
connected
{
non-flat}
limit, hence a helicoid. Now imagine that we want to avoid
this helicoidal blow-up limit, and note that the injectivity radius
of $M$ is zero, i.e., there exists an intrinsically divergent
sequence of points $p_n \in M$ where the injectivity radius function
of $M$ tends to zero;
If we
choose these points $p_n$ carefully and blow up $M$ around the $p_n$
in a similar way as we did in the above theorem, but exchanging the
former ratio of rescaling (which was  the norm
of the second fundamental form  at $p_n$) by the inverse of the injectivity
radius at these points, then we will obtain a
sequence of $H_n$-surfaces, $H_n\to 0$,  which are
non-simply connected in a fixed ball of space
(after identifying again $p_n$ with $\vec{0}\in \R^n$), and it is
natural to ask about possible limits of such a blow-up sequence.
This is the purpose of the next result.

For a
complete Riemannian manifold $M$, we will let $I_M\colon M\to (0,\infty ]$ be
the injectivity radius function of $M$, and given a subdomain $\Omega \subset
M$, $I_{\Omega }=(I_M)|_{\Omega }$ will stand for the restriction of $I_M$ to $\Omega $.
Recall that
the infimum of $I_M$ is called the {\it injectivity radius} of $M$.

The next theorem  by Meeks, P\' erez and Ros appears in~\cite{mpr14}.

\begin{theorem}[Local Picture on the Scale of Topology]
\label{tthm3introd}
There exists a smooth decreasing function $\de\colon (0,\infty) \to (0,\infty)$
with $\lim_{r\to \infty} r\de(r)=\infty$
such that the following statements hold.
Suppose $M$ is a complete
0-surface with injectivity radius zero in a homogeneously
regular
 3-manifold~$N$. Then, there exists a sequence
of points $p_n\in M$ (called ``points of almost-minimal injectivity radius'')
and positive numbers $\ve _n= n\, I_{M}(p_n)\to 0$ such that:
\begin{description}
\item[{\it 1.}]  For all $n$, the closure $M_n$ of the component
of $M\cap B_N(p_n,\ve _n)$  that contains $p_n$ is  a compact surface with boundary in
$\partial B_N(p_n,\ve _n)$. Furthermore, $M_n$ is contained in the intrinsic open ball
$B_M(p_n,\frac{r_n}{2}I_M(p_n))$, where
$r_n>0$ satisfies $r_n\de (r_n)=n$.
\item[{\it 2.}] Let $\l _n=1/I_{M}(p_n)$.
Then, 
$\l _nI_{M_n}\geq 1-\frac{1}{n}$ on $M_n$.
\item[{\it 3.}]The rescaling of the metric balls $B_N(p_n,\ve_n)$ by factor $\l_n$
converge uniformly as $n\to \infty $ to $\rth$ with
its usual metric (as in Theorem~\ref{thm3introd}, we identify $p_n$ with
$\vec{0}$ for all $n$).
\end{description}
Furthermore, exactly one of the following three possibilities occurs.
\begin{description}
\item[{\it 4.}] The surfaces $\l_nM_n$ have
uniformly bounded Gaussian curvature
on compact subsets\footnote{As $M_n\subset B_N(p_n,\ve _n)$, the convergence
$\{ \l _nB_N(p_n,\ve _n)\} _n\to \R^3$ explained in item~{\it 3} allows us to view
the rescaled surface $\l _nM_n$ as a subset of $\R^3$. The uniformly bounded
property for the Gaussian curvature of the induced metric on $M_n\subset N$ rescaled by
$\l _n$ on compact subsets of $\R^3$ now makes sense.} of
$\rth$ and there exists a connected, properly
embedded  $0$-surface
$M_{\infty}\subset \R^3$
with $\vec{0}\in M_{\infty }$, $I_{M_{\infty}}\geq 1$
and $I_{M_{\infty}}(\vec{0})=1$,
 such that for any
$k \in \N$, the surfaces $\l _nM_n$ converge $C^k$
on compact subsets of $\rth$  to $M_{\infty}$ with
multiplicity one as $n \to \infty$.
\item[{\it 5.}] After a rotation in $\rth$, the surfaces $\l_nM_n$ converge to
a minimal parking garage structure on $\rth$,
consisting of a foliation $\cL$ of $\R^3$ by
horizontal planes, with columns forming a locally finite
set $S(\cL)$ of vertical straight lines (the set $S(\cL)$ is the singular set of
convergence
of $\l _nM_n$ to $\cL$), and:
\begin{enumerate}[(5.1)]
\item $S(\cL)$ contains a line $l_1$  which
passes through the closed ball of radius 1 centered
at the origin, and another line $l_2$ at
distance one from $l_1$, and all of the lines in $S(\cL)$
have distance at least one from each other.
\item There exist oriented, homotopically non-trivial
simple closed curves  $\g_n\subset\l_nM_n$
with lengths converging to~$2$, which converge
to the line segment $\g$ that joins
$(l_1\cup l_2)\cap \{x_3=0\}$ and such that the
integrals of the unit conormal vector
of $\l_nM_n$ along $\g_n$ in the induced exponential
$\rth$-coordinates of $\l _nB_N(p_n,\ve _n)$
converge to a horizontal vector of length $2$ orthogonal to $\g$.
\item If there exists a bound on the genus
of the surfaces $\l_nM_n$, then:
\ben
\item $S(\cL)$ consists  of just two lines
$l_1, \,l_2$ and the associated  limiting
double multigraphs in $\l _nM_n$ are oppositely handed.
\item Given $R>0$, for $n\in \N$ sufficiently large depending on $R$, the surface
$(\l _nM_n)\cap B_{\l _nN}(p,R/\l _n)$ has genus zero.
\een

\end{enumerate}
\item[{\it 6.}]
\label{i6} There exists a non-empty, closed set $\cS\subset
\R^3$ and a 0-lamination $\cL$ of \mbox{$\R^3-\cS$}
such that the surfaces $(\l _nM_n)-\cS$ converge to $\cL$
outside of some singular set of convergence $S(\cL) \subset\cL$,
and $\cL$ has at least one non-flat leaf. Furthermore,
if we let $\Delta (\cL)=\cS \cup S(\cL)$, then, after a rotation of $\rth$:
\begin{enumerate}[(6.1)]
\item Let ${\cal P}$ be the sublamination of flat leaves in ${\cal L}$. Then, $\cP\neq
\mbox{\rm \O}$ and the closure of every such flat leaf is a horizontal plane.
Furthermore, if $L\in \cP$, then the plane
$\overline{L}$ intersects $\Delta (\cL)$ in a set containing at least two points,
each of which are at least distance 1 from each other in $\overline{L}$, and either
$\overline{L}\cap \Delta (\cL)\subset \cS$
or $\overline{L}\cap \Delta (\cL)\subset S(\cL)$.
\item $\Delta (\cL)$ is a closed set of $\rth$ which is contained in the union of planes
$\bigcup_{L \in \cP} \overline{L}$. Furthermore,
every plane in $\R^3$ intersects $\cL$.
\item There exists $R_0>0$ such that the
sequence of surfaces $\left\{ M_n\cap B_M(p_n,\frac{R_0}{\l _n})
\right\} _n$ does not have bounded genus.
\item There exist oriented closed geodesics  $\g_n\subset\l_nM_n$
with uniformly bounded lengths which converge to a line segment $\g$
in the closure of some flat leaf in $\cP$, which
joins two points of $\Delta(\cL)$, and such that the integrals of
$\l_nM_n$ along $\g_n$
in the induced exponential $\rth$-coordinates of $\l _nB_N(p_n,\ve _n)$
converge to
a horizontal vector orthogonal to $\g$ with length $2\, \mbox{\rm Length}(\g)$.
\end{enumerate}
\end{description}
\end{theorem}
The results in the series of papers
\cite{cm23,cm35,cm25} by Colding and Minicozzi
and the minimal lamination closure theorem by Meeks and
Rosenberg~\cite{mr13}  play important roles in
deriving the above compactness result.
The first two authors conjecture that
item~{6} in Theorem~\ref{tthm3introd} does not actually occur.

The local picture theorems on the scales of curvature and topology deal with
limits of a sequence of blow-up rescalings
for a complete 0-surface. Next we will study an
interesting function which is invariant by
rescalings, namely the Gaussian curvature of a surface in $\R^3$
times the squared distance to a given point.
A complete Riemannian surface $M$ is said to have {\em intrinsic
quadratic curvature decay
constant $C>0$ with respect to a point
$p\in M$}, if the absolute curvature function $|K_M|$ of $M$
satisfies
\[
|K_M(q)|\leq \frac{C}{d_M(p,q)^2}\qquad \mbox{for all }q\in M-\{ p\},
\]
where $d_M$ denotes the Riemannian distance
function. Note that
if such a Riemannian surface $M$ is a complete surface
in $\R^3$ with $p=\vec{0}\in M$,
then it also has extrinsic quadratic decay constant $C$ with respect
to the radial distance $R$ to $\vec{0}$,
i.e., $|K_M|R^2 \leq {C}$ on $M$. For
this reason, when we say that a 0-surface in $\rth$ has {\em
quadratic decay of curvature}, we
will always refer to curvature decay with respect to the
extrinsic distance $R$ to $\vec{0}$, independently of whether or not
$M$ passes through or limits to~$\vec{0}$.

\begin{theorem}[Quadratic Curvature Decay Theorem, \, Meeks, P\'erez,
Ros~\cite{mpr10}]
\label{thm1introd}
Let $M\subset \R^3-\{
\vec{0}\} $ be a 0-surface with compact boundary
(possibly empty), which is complete outside the origin $\vec{0}$,
i.e., all divergent paths of finite length on~$M$ limit to
$\vec{0}$. Then, $M$
has quadratic decay of curvature if and only if
its closure in $\R^3$ has finite total curvature.
In particular, every  complete 0-surface $M\subset \R^3$ with compact boundary and
quadratic decay of curvature is properly embedded in~$\R^3$.
Furthermore,
if $C$ is the maximum of the logarithmic growths of the ends of $M$,
then
\[
\lim _{R\to \infty }\sup _{M-{\mbox{\bbsmall B}}(R)}|K_M|R^4=C^2,
\]
where $\B (R)$ denotes the extrinsic ball of radius $R$ centered at
$\vec{0}$.
\end{theorem}

\begin{remark}{\em
The above Quadratic Curvature Decay Theorem
is a crucial tool in understanding the asymptotic behavior of all
properly embedded minimal surfaces in $\rth$ via rescaling arguments.
This application is called the {\it Dynamics Theorem
for Properly Embedded Minimal Surfaces} by Meeks, P\'erez and Ros~\cite{mpr20},
which we will not explain in this article; instead, we will discuss in the next section
a closely related dynamics type theorem for certain $(H>0)$- surfaces in $\rth$.
}
\end{remark}

\section{The Dynamics Theorem for $H$-surfaces in $\rth$.}
\label{sec:dynamics}
We now apply some of the theoretical results in the previous sections
to analyze some aspects of the asymptotic behavior of a given
proper $H$-surface $M$  in $\rth$.
To understand this asymptotic behavior, we consider two
separate cases.

In the case that $M$ has bounded second fundamental form,
the answer to this problem consists of classifying the space $\Te(M)$ of limits of
sequences of the form $\{ M-p_n\} _n$, where $p_n\in M$,
$|p_n|\to \infty $ {(equivalently $\{ p_n\} _n$ is a divergent sequence in $M$,
note that $M$ is proper as follows from Corollary~\ref{cor.max} and Theorem~\ref{cor*}).}
Observe that $\{ M-p_n\} _n$ has area estimates in balls of any fixed radius
by Corollary~\ref{cor.max} {and Theorem~\ref{cor*};}
hence, every divergent sequence of translations of $M$ has a
subsequence that converges on compact subsets of $\rth$ to a possibly immersed $H$-surface.
In fact, the surfaces in $\Te(M)$ are possibly disconnected, strongly Alexandrov embedded $H$-surfaces.

In the case that $M$ is proper but does not have bounded second fundamental form,
a more natural way to understand its asymptotic behavior  is  to consider all proper non-flat surfaces in $\R^3$
that can obtained as a limit of a sequence of the form $\{ \l _n(M-p_n)\} _n$, where
$p_n\in M$ is a sequence of points that diverges in $\R^3$ and
$\l _n=|A_M|(p_n)$ is unbounded (as usual, $|A_M|$ is the norm of the second fundamental
form of $M$). This set of limits by divergent dilations of $M$ was studied by
Meeks, P\'erez and Ros assuming that $M$ is a $0$-surface
(this is their {\it Dynamics Theorem}~\cite{mpr20}, which is an application of their Quadratic
Curvature Decay Theorem~\ref{thm1introd}); also see Chapter~11
in~\cite{mpe10} for further discussion in this minimal case when $|A_M|$ is not bounded.

The material covered here is based on~\cite{mt4} by Meeks and Tinaglia,
which was motivated by the earlier work in~\cite{mpr20} and
we refer the reader to~\cite{mt4} for further details.  We will focus our attention here
on some of the
less technical results in~\cite{mt4} and the basic techniques developed there.

\begin{definition}
\label{def9.1}
 {\rm Suppose that $M\subset \rth$ is a complete,
 non-compact, connected $H$-surface with compact boundary
(possibly empty) and with bounded second fundamental form.

\begin{enumerate}[1.] \item For any divergent sequence of
points $p_n\in M$, a subsequence of the translated surfaces $M-p_n$
converges to a properly immersed $H$-surface which bounds a smooth open subdomain on its mean convex
side. Let $\Te(M)$ denote the collection of all such limit
surfaces.

\item We say that $M$ is {\it chord-arc} if there exists a constant $C>0$ such that
for all $p,q\in M$ with $\| p-q\|\geq 1$, we have $d_M(p,q) \leq C \| p-q\|$.
Note that if $M$ is chord-arc with constant $C$ and $p,q\in M$ with $\| p-q\| <1$,
then $d_{M}(p,q)\leq 5C$ by the following argument.
Let $z\in M\cap \partial B(p,2)$; applying the chord-arc property of $M$ to
$p,z$ and to $z,q$ we obtain $d_M(p,z)\leq 2C$ and $d_M(z,q)\leq C|q-z|\leq 3C$.
Hence, the triangle inequality gives  $d_{M}(p,q)\leq 5C$.
 \end{enumerate}
}
\end{definition}

In order to study $\Te(M)$ it is convenient for technical reasons to study
a closely related space $\TM$ that can be thought of a subset of $\Te(M)$.

\begin{definition} {\rm Suppose $M\subset \rth$ is a
 non-compact, strongly Alexandrov embedded $H$-surface with bounded
second fundamental form.
\begin{enumerate}[1.]
\item We define the set ${\cal T}(M)$ of all connected, strongly
Alexandrov embedded $H$-surfaces $\S \subset \rth$, which are
obtained in the following way.

There exists a divergent sequence of points $p_n\in M$
such that the translated surfaces $M-p_n$
converge $C^2$ on compact sets of $\rth$ to a strongly Alexandrov
embedded $H$-surface $\Sigma'$, and $\Sigma$ is a connected
component of $\Sigma'$ passing through the origin. Actually we
consider the immersed surfaces in ${\cal T}(M)$ to be {\it pointed}
in the sense that if such a surface is not embedded at the origin,
then we consider the surface to represent two different elements in
${\cal T}(M)$ depending on a choice of one of the two preimages of
the origin.
\item $\Delta \subset  {\cal T}(M)$ is called
{\em ${\cal T}$-invariant}, if $\S\in\Delta$
implies ${\cal T}(\S)\subset \Delta$.
\item A non-empty subset $\Delta\subset {\cal T}(M)$ is
called a {\em minimal} ${\cal T}$-invariant
set, if it is ${\cal T}$-invariant and contains no smaller non-empty ${\cal
T}$-invariant sets.
\item If $\S \in {\cal T}(M)$ and $\S$ lies in a minimal ${\cal T}$-invariant
 set of ${\cal T}(M)$, then $\S$ is called a  {\em minimal
element} of ${\cal T}(M)$.
\end{enumerate}}
\end{definition}

${\cal T}(M)$ has a natural compact
metric space topology, that we now describe.
Suppose   that $\S \in {\cal T}(M)$ is embedded at the origin. In this
case, there exists an $\ve>0$ depending only on the bound of the
second fundamental form of $M$, so that there exists a disk
$D(\S)\subset \S\cap \overline{\B}(\ve)$ with $\partial
D(\S)\subset\partial\overline{\B}(\ve)$, $\vec{0} \in D(\S)$
and such that $D(\S)$ is a graph with gradient at most 1 over its
projection to the tangent plane $T_{\vec{0}}D(\S)\subset \rth$.
Given another such $\S'\in {\cal T}(M)$, define
\begin{equation}
\label{dist}
 d_{{\cal T}(M)}(\S,\S')=d_{\cal H}(D(\S),D(\S')),
\end{equation}
where $d_{\cal H}$ is the Hausdorff distance. If $\vec{0}$ is not a
point where $\S$ is embedded, then since we consider $\S$ to
represent one of two different pointed surfaces in ${\cal T}(M)$, we
choose $D(\S)$ to be the disk in $\S\cap\B(\ve)$ corresponding to the
chosen base point. With this modification, the above metric is
well-defined on ${\cal T}(M)$.  Using the curvature and local area estimates
of elements in $\TM$, it is straightforward to prove that $\TM$ is sequentially
compact and therefore it has a compact metric space structure.

 Given a surface
$\S\in {\cal T}(M)$,  it can be shown that ${\cal T}(\S)$ is  a
subset of ${\cal T}(M)$.
In particular, we can consider ${\cal T}$
to represent a function:
$${\cal T}\colon {\cal T}(M) \to {\cal P}({\cal T}(M)),$$
where ${\cal P}({\cal T}(M))$ denotes the power set of ${\cal
T}(M)$. Using the natural compact metric space structure on $\TM$,
we can  obtain classical dynamics type results on
${\cal T}(M)$ with respect to the mapping ${\cal T}$. These
dynamics results include the existence of non-empty minimal ${\cal
T}$-invariant sets in $\TM$.

\begin{theorem}[CMC Dynamics Theorem]
\label{T}
Let $M\subset \rth$ be a connected, non-compact, strongly Alexandrov
embedded $(H>0)$-surface with bounded second fundamental form.
Then the following statements hold:
\begin{enumerate}[1.]
\item \label{n4} ${\cal T}(M)$ is non-empty and ${\cal T}$-invariant.
\item \label{n5} ${\cal T}(M)$ has a natural compact topological space structure
given by the metric $d_{{\cal T}(M)}$ defined in~(\ref{dist}).
\item \label{n8} Every non-empty ${\cal T}$-invariant
set of ${\cal T}(M)$ contains a non-empty minimal ${\cal
T}$-invariant set. In particular, since ${\cal T}(M)$ is itself a
 non-empty ${\cal T}$-invariant set, then ${\cal T}(M)$ contains
 non-empty minimal invariant elements.
\end{enumerate}
\end{theorem}

\begin{definition}
\label{defM(p,R)}
{\em
For any point $p$ in a surface $M\subset \R^3$,
we denote by $M(p,R)$ the
closure of the
connected component of
$M\cap \B(p,R)$ which contains $p$.

If $M$ is a surface satisfying the hypotheses of Theorem~\ref{T} and
$M$ is not embedded at $p$ having two immersed components $M(p,R)$, $M'(p,R)$
that correspond to two pointed immersions, then in what follows we
will consider both of these components separately.
}
\end{definition}

As an application of the above Dynamics Theorem, we have the following result
on the geometry of minimal elements of $\TM$.

\begin{theorem}[Minimal Element Theorem] \label{sp2} Let $M\subset
\rth$ be a complete, non-compact,  $(H>0)$-surface
with possibly empty compact boundary and bounded second fundamental
form. Then, the following statements hold.

\begin{enumerate}[1.]
\item \label{oneend}
If $\S\in \TM$ is a minimal element, then
either every surface in $\Te(\S)$
is the translation of a fixed
Delaunay surface, or every surface in $\Te(\S)$ has one end. In
particular, every surface
in $\Te(\S)$ is connected and, after ignoring base points, $\TS=\Te(\S)$.
\item \label{sca} Minimal elements of $\TM$ are chord-arc, in the sense of
Definition~\ref{def9.1}.
\item \label{inf2} Suppose
$\Sigma$ is a minimal element of ${\cal T}(M)$. Then,
the following statements are equivalent.
\begin{enumerate}[a.]
\item  \label{D} $\Sigma$ is a Delaunay surface.
\item  \label{Afinite1}
$\lim _{R\to \infty }\underline{A}(R)=0$, where $\underline{A}(R)=\inf _{R_1\geq R}
\left( \inf_{p\in \S}({\rm Area}[\S(p,R_1)]\cdot R_1^{-2})\right) $.
\item  \label{Gfinite1}
$\lim _{R\to \infty }\underline{G}(R)=0$, where $\underline{G}(R)=\inf _{R_1\geq R}
\left( \inf_{p\in \S}({\rm Genus}[\S(p,R_1)]\cdot R_1^{-2})\right) $.
\end{enumerate}
\end{enumerate}
\end{theorem}

\section{The curvature and radius estimates
of Meeks-Tinaglia.} \label{sec:MT}
A longstanding  problem in classical surface theory
is to classify the complete, simply connected $H$-surfaces  in $\rth$.
In the case the surface is simply connected and
compact, this classification
follows by work of either Hopf~\cite{hf1} in 1951 or of
Alexandrov~\cite{aa1} in 1956, who gave different proofs that
a round sphere is the only possibility.
In~\cite{mt7}, Meeks and Tinaglia have recently proved that a complete, simply
connected $(H>0)$-surface is compact.
\begin{theorem} \label{round} Complete
simply connected $(H>0)$-surfaces   in
$\rth$  are compact, and thus are round spheres.
\end{theorem}

The two main ingredients in the proof of Theorem~\ref{round} in~\cite{mt7} are
the radius estimates Theorem~\ref{rest} and the curvature estimates
Theorem~\ref{cest}. For the reader's convenience, we restate them here.
%

\begin{theorem}[Radius Estimates, Meeks-Tinaglia~\cite{mt7}]
\label{rest2} There exists an
${\mathcal R}\geq \pi$ such that any $H$-disk in $\rth$ with
$H>0$ has radius less than ${{\mathcal R}}/{H}$.
\end{theorem}

\begin{theorem}[Curvature Estimates, Meeks-Tinaglia~\cite{mt7}]
\label{cest2} Given $\delta$, $\cH>0$,
there exists a  $K(\delta,\cH)\geq\sqrt{2}\cH$ such that
any   $H$-disk $M$  in $\rth$ with
 $H\geq \cH$ satisfies
\[
\sup_{\large \{p\in {M} \, \mid \,
d_{M}(p,\partial M)\geq \delta\}} |A_{M}|\leq  K(\delta,\cH),
\]
where 
$d_{M}$ is the intrinsic distance function of $M$.
\end{theorem}

Since every point on an $H$-surface $M$ of positive injectivity radius ${r_0}$
is the center of a geodesic $H$-disk in $M$ of radius $r_0$,
the curvature estimate in Theorem~\ref{cest2}
has the following immediate consequence.

\begin{corollary} \label{corinj1} If $M$ is a
complete $(H>0)$-surface with positive injectivity radius ${r_0}$, then
 $$\sup_M |A_M|\leq K(r_0,H).$$
\end{corollary}

As complete
$(H>0)$-surfaces of bounded curvature are properly embedded in $\rth$
by  Theorem~\ref{cor*},
then Corollary~\ref{corinj1} implies the next result.

\begin{corollary} \label{corinj2} A
complete $(H>0)$-surface with positive injectivity radius is properly embedded in $\rth$.
\end{corollary}

Since there exists an $\ve>0$ such that for any $C>0$, every complete immersed
surface $\S$ in $\rth$ with $\sup_{\S}|A_\S|<C$ has injectivity radius greater
than $\ve/ C$, then Corollary~\ref{corinj1} also demonstrates
that a necessary and sufficient condition
for an $(H>0)$-surface in $\rth$ to have bounded curvature
is that it has positive injectivity radius.

\begin{corollary} \label{cor10.6}
A complete $(H>0)$-surface has positive injectivity radius
if and only if it has bounded curvature.
\end{corollary}

We now give an outline of Meeks and Tinaglia's
approach to proving Theorems~\ref{rest2}
and \ref{cest2}. \vspace{.2cm}

\noindent {  \bf Step 1:}\, {\em Prove
analogous curvature estimates for $(H>0)$-disks
in terms of extrinsic rather than intrinsic
distances of points to the boundaries of their disks.}

The proofs of this extrinsic
version of Theorem~\ref{cest2} is by contradiction
and relies on an accurate
geometric description of a $1$-disk near interior points where
the norm of the second fundamental form becomes
arbitrarily large. This geometric description
is inspired by the pioneering work of Colding and Minicozzi
in the minimal case~\cite{cm21,cm22,cm23}.

The extrinsic curvature estimates just alluded to are the following ones.

\begin{lemma}[First Extrinsic Curvature Estimate] \label{lem:excest}
Given $\delta>0$ and $H\in (0,\frac 1{2\delta})$,
there exists  $K_0(\delta, H)>0$ such that for any $H$-disk
{$M\subset \R^3$,}
\[
\sup_{\large \{p\in {M} \, \mid \,
d_{\rth}(p,\partial M)\geq \delta\}} |A_{M}|\leq  K_0(\delta,H),
\]
\end{lemma}

The arguments in the proof of the above lemma deal
only with the component $\Delta$ of the
intersection $\B(p,\delta)\cap M$ that contains $p$.
Since the convex hull property does not hold
for $(H>0)$-disks, in principle the topology of
the planar domain $\Delta$ might be arbitrarily
complicated, in the sense that the number $k$ of
boundary components might not have a universal
upper bound. However, this potential problem
is solved by proving, for any $k\in \N$, an extrinsic  curvature
estimate $K_0(\delta, H,k)$ valid when $\Delta $ has at most
$k$ boundary components, and then by demonstrating
the following result on the existence
of an upper bound $N_0$ on the number
of boundary components of $\Delta$:

\begin{proposition}[Proposition 3.1 in~\cite{mt7}] \label{number}
There exists $N_0\in \N$ such that
for any $R\leq{\frac{1}{2}}$ and $H\in [0,1]$, if $M$ is a compact $H$-disk
with $\partial M\subset \rth-\B(R)$
and $M$ is transverse to $\partial \B(R)$, then each component of
$M\cap \B(R)$ has at most $N_0$ boundary components.
\end{proposition}

Since $\Delta$ is a subset of a disk, then every 1-cycle on it has zero flux. Hence,
Lemma~\ref{lem:excest} follows from Proposition ~\ref{number} and the
next extrinsic curvature estimate.

\begin{lemma}[Second Extrinsic Curvature Estimate] \label{lem:excest2}
Given $\delta>0$ and $H\in (0,\frac 1{2\delta})$,
there exists $K_0(\delta, H,k)>0$ such that any $H$-planar
domain $\Delta$ with zero flux and at most
$k$ boundary components satisfies:
\[
\sup_{\large \{p\in {\Delta} \, \mid \,
d_{\rth}(p,\partial \Delta)\geq \delta\}} |A_{\Delta}|\leq  K_1(\delta,H,k),
\]
\end{lemma}

For details on the following  outline
of the proof of Lemma~\ref{lem:excest2},
see~\cite{mt7}. Arguing by contradiction, assume that the lemma fails.
One can easily reduce the proof of Lemma~\ref{lem:excest2} to the following situation.
There exists a sequence  $\{ \Delta(j)\} _j$ of  $1$-planar domains with
zero flux satisfying the following properties for each $j\in \N$:
\ben[1.]
\item $\Delta(j)\subset \B(\delta)$ and $\partial \Delta(j)\subset \partial\B(\delta)$.
\item $\vec{0}$ is a point of almost-maximal curvature on $\Delta(j)$
with $|A_{\Delta (j)}|
{(\vec{0})} >j$,
in the sense that there exists a sequence of positive numbers $\de_j\to 0$
with $ \de_j\cdot |A_{\Delta (j)}|(\vec{0}) \to \infty$ and
$\lim_{j\to \infty}\max\{|A_{\Delta (j)}|(q)\mid q\in \B(\de_j)\}\cdot|A_{\Delta (j)}|(\vec{0})= 1$.
\een

The proof of the  Local Picture Theorem on the Scale of Curvature (Theorem~\ref{thm3introd}) implies,
after replacing by a subsequence, that the homothetically scaled surfaces
$$\S(j)=|A_{\Delta (j)}|{(\vec{0})}\cdot \Delta(j)$$
converge to a {non-flat,}
 proper 0-planar domain $\S_\infty\subset \R^3$ with zero flux.
By the classification Theorem~\ref{classthm} of such 0-planar domains,
$\S_\infty$ must be helicoid that we will assume
has a vertical axis passing through the origin. What this means geometrically is
that on the scale of curvature, vertical helicoids are forming around
the point $\vec{0}\in \Delta(j)$ as $j\to \infty$.

Hence, after replacing by a subsequence,
for any $n\in \N$, there exists an integer $J(n)$ such that for $j>J(n)$, the following
statements hold.
On the scale of curvature and near the origin, there exists in $\Delta(j)$
a pair $\wh{G}^{up}_j,\wh{G}^{down}_j$ of $n$-valued graphs
which correspond to ``large" $n$-valued graphs in the scaled almost-helicoids and that
have  ``small" gradients for their
$n$-valued graphing functions; here the superscripts ``up" and ``down" refer to the
direction of their mean curvature vectors.
The most difficult  part in demonstrating  Step 1 is to prove that, for $n$ sufficiently large,
the $n$-valued graphs  $\wh{G}^{up}_j,\wh{G}^{down}_j$  contain 2-valued
subgraphs that extend to 2-valued graphs ${G}^{up}_j,{G}^{down}_j$
in $\Delta(j)$ on a scale proportional
to $\delta$. Furthermore, it is shown that the
2-valued graph ${G}^{up}_j$ can be chosen to contain a sheet that lies between
the two sheets of ${G}^{down}_j$.   Crucial in the proof of this extension results
is the work
of Colding and Minicozzi~\cite{cm21} on the extension of multi-valued graphs inside of
certain 0-disks in $\rth$.  In our situation, one applies their results to
stable 0-disks $E(n)$ that are shown to  exist in the complement of $\Delta(j)$
in the small ball $\B(\delta)$; see~\cite{mt7} for details.
Finally one obtains a contradiction by showing that one can
choose the multigraphs ${G}^{up}_j,{G}^{down}_j$, so that as $j\to \infty$, they collapse to a
$1$-graph over an annulus in the$(x_1,x_2)$-plane, which is impossible since
${G}^{up}_j,{G}^{down}_j$ have oppositely signed mean curvatures. \vspace{.3cm}

\noindent {  \bf Step 2:}\, {\em Relate the existence of
an extrinsic curvature estimate in Step 1 to
the existence of  extrinsic radius estimates for $H$-disks.}

Arguing again by contradiction,
we may assume that $E(n)$ is a sequence of 1-disks with $\vec{0}\in E(n)$ and
$\partial E(n) \subset \rth-\B (n+1)$. By Step 1, the 1-planar domains $E(n)\cap \B(n)$
have bounded  norm of their second fundamental
forms. By rather standard arguments like those used in the proof
of the Dynamics Theorem~\ref{T}, a subsequence of the $E(n)\cap \B(n)$ converges to
a strongly Alexandrov embedded 1-surface $M$ in $\rth$ with zero flux.  But
item~\ref{inf2} in the Minimal Element
Theorem~\ref{sp2} implies that there exists a Delaunay surface which
is a limit of a sequence of translations of subdomains
in $M$. This is contradiction, since a Delaunay surface has non-zero flux.
\vspace{.3cm}

\noindent {  \bf Step 3:}\, {\em Prove  the following one-sided curvature
estimate for $H$-disks. }

\begin{theorem}[One-sided curvature estimate for $H$-disks, Meeks, Tinaglia \cite{mt9}] \label{th}
There exist $\ve\in(0,\frac{1}{2})$ and
$C \geq 2 \sqrt{2}$ such that for any $R>0$, the following holds.
Let $M\subset \R^3$ be an $H$-disk such that $M\cap \B(R)\cap\{x_3=0\}
=\O$ and $\partial M\cap \B(R)\cap\{x_3>0\}=\O.$
Then,
\begin{equation} \label{eq1}
\sup _{x\in M\cap \B(\ve R)\cap\{x_3>0\}} |A_M|(x)\leq \frac{C}{R}.
\end{equation} In particular,
if $M\cap \B(\ve R)\cap\{x_3>0\}\neq\O$, then $H\leq \frac{C}{R}$.
\end{theorem}

A key technical result that is needed in the proof of the above one-sided curvature
estimate is the existence of the extrinsic curvature estimates in Step 1.
Perhaps even more important in the
proof of Theorem~\ref{th} are the extension results for multi-valued graphs
in the surfaces $\Delta(j)$ described in the sketch of that proof  of Step~1,
and some rather technical results on 0-laminations of $\rth$ with a finite
number of singularities. Some of the tools used in the proof of Step~3
include the one-sided curvature estimates for
0-disks in Theorem~\ref{thmcurvestimCM}, the Stable Limit Leaf Theorem~\ref{thmstable},
the Local Removable Singularity Theorem~\ref{tt2} and
the Local Picture Theorem~\ref{tthm3introd} on the Scale of Topology.

\vspace{.2cm}
\noindent {  \bf Step 4:}\, {\em Relate the existence of
an extrinsic curvature estimate in Step 1 to
the existence of  an intrinsic curvature
estimate via the following weak chord-arc result for $H$-disks.}

Recall from Definition~\ref{defM(p,R)} that given a point $p$ in a surface
$\Sigma\subset \rth$, $\S (p,R)$ denotes the closure of the component of
$\Sigma \cap \B(p,R)$ passing through $p$.

\begin{theorem}[Weak Chord Arc Estimate, Theorem 1.2
in~\cite{mt8}]
\label{thm1.1}
 There exists a $\delta_1 \in (0,
\frac{1}{2})$  such that the following holds.
Given an   $H$-disk in $\S \subset \rth$ and an
intrinsic closed ball $\ov{B}_\S(x,R)$ which is contained in $\S-
\partial \S$, we have
\ben[1.]
\item $\S (x,\delta_1 R)$ is a disk with
$\partial \Sigma(\vec{0},\delta_1 R)\subset \partial \B(\de_1R)$.
\item $ \S (x, \delta_1 R) \subset B_\S (x, \frac{R}{2}).$
\een
\end{theorem}

Theorem~\ref{cest} is a straightforward consequence
of the Theorem~\ref{thm1.1}. Once one has obtained the 1-sided
curvature estimate in Step 3, the strategy of the proof of
Theorem~\ref{thm1.1} is similar to the strategy of the
proof of Proposition~1.1 in~\cite{cm35} by Colding and Minicozzi.\vspace{.1cm}

\vspace{.2cm}

\noindent {  \bf Step 5:}\, {\em Relate the existence of
an intrinsic curvature estimate in Step 4 to
the existence of  radius estimates for $H$-disks.}
This final step of the proof of Theorem~\ref{cest2} is similar to that
of Step~2, which completes our sketch of the proofs of Theorems~\ref{rest2} and \ref{cest2}.
\vspace{.4cm}

In~\cite{mt7}, Meeks and Tinaglia also  obtain curvature estimates
for $(H>0)$-annuli. However, while these estimates
are analogous to
the curvature estimates in Theorem~\ref{cest2} for $H$-disks,
they necessarily must depend on the length of the
flux vector of the generator
of the first homology group of the given annulus.
An immediate consequence of these curvature estimates for $H$-annuli is
the next Theorem~\ref{annulus} on the
properness of complete $H$-surfaces of finite topology.
The next theorem is what allows us to obtain
the properness and curvature estimates for classical finite topology
$(H>0)$-surfaces described in Section~\ref{subsecends}.

\begin{theorem} \label{annulus} A complete $H$-surface  with smooth
compact boundary (possibly empty) and finite topology has
bounded curvature and  is properly embedded in $\rth$.
\end{theorem}

The  theory developed
in this manuscript also provides key tools for understanding the geometry
of $(H>0)$-disks in a Riemannian 3-manifold, especially in the case that
the manifold is complete and locally homogeneous.  These
generalizations and applications of the work presented here
will appear in ~\cite{mt1}, and we mention
 two of them here and refer the reader
to~\cite{mt11} for details.

First of all, one has the next generalization of Theorem~\ref{cest2}.

\begin{theorem}[Curvature Estimates, Meeks-Tinaglia~\cite{mt1}]
\label{cest3}
 Let $X$ be a homogeneous 3-manifold. Given $\delta$, $\cH>0$,
there exists $K(\delta,\cH,X)\geq\sqrt{2}\cH$ such that
any   $H$-disk $M$  in $X$ with
 $H\geq \cH$ satisfies
\[
\sup_{\large \{p\in {M} \, \mid \,
d_{M}(p,\partial M)\geq \delta\}} |A_{M}|\leq  K(\delta,\cH,X).
\]
\end{theorem}

The main difficulty in generalizing the curvature estimate in Theorem~\ref{cest2}
to  the setting of homogeneous 3-manifolds
is that one does not have a corresponding result like Proposition~\ref{number}.
But this problem can be solved by applying and adapting the manifold techniques applied
by Meeks and Rosenberg~\cite{mr13}, in their proof of the
Minimal Lamination Closure Theorem~\ref{thmmlct}.  One obtains a modified version of
Theorem~\ref{cest2}  in the setting of homogeneously regular 3-manifolds for
any compact $(H>0)$-surface whose injectivity radius function is bounded away from zero outside
of a small neighborhood of its boundary and in the small neighborhood this function is
equal to the distance to the boundary.
This approach gives the following generalization of Corollary~\ref{cor10.6}.

\begin{theorem}
 A complete $(H>0)$-surface in a homogeneously regular 3-manifold
 has positive injectivity radius
 if and only if it has bounded second fundamental form.
\end{theorem}

As an application of these  results, one can prove the following
theorem.

\begin{theorem}[Meeks, Tinaglia \cite{mt11}]
Let $H\geq 1$.
Then, any complete $H$-surface of finite topology in a complete
hyperbolic 3-manifold is proper.
\end{theorem}

This result is sharp when the complete hyperbolic 3-manifold is $\HH^3$:

\begin{theorem}[Coskunuzer, Meeks, Tinaglia\, \cite{cmt1}]
\label{mainBaris}
For every $H\in [0,1)$, there
exists a complete, stable simply-connected $H$-surface
in $\mathbb{H}^3$  that is not proper.
\end{theorem}

See also~\cite{cosk1} and~\cite{rodt1} for examples of non-properly
embedded complete simply-connected $0$-surfaces in $\mathbb H^3$ and
$\mathbb H^2\times\mathbb R$ respectively.

\section{Calabi-Yau problems.}
\label{seccy}
The Calabi-Yau problems or conjectures refer to a series of
questions concerning the non-existence of a complete, 0-immersion
$f \colon M \rightarrow \rth$ whose image $f(M)$ is
constrained to lie in a particular region of $\rth$ (see~\cite{ca4}, page 212
in~\cite{che4}, problem 91 in~\cite{yau1} and
page 360 in~\cite{yau2}).  Calabi's original conjecture states that
a complete non-flat minimal surface cannot be contained either in
the unit ball $\B(1)$ or in a slab. The first important negative
result on the Calabi-Yau problem was given by Jorge and
Xavier~\cite{jx1}, who proved the existence of a complete
0-surface contained in an open slab of $\rth$. In 1996,
Nadirashvili~\cite{na1} constructed a complete minimal disk in
$\B(1)$; {such a}
minimal disk cannot be embedded by the
Colding-Minicozzi Theorem~\ref{thmCM}. A clever refinement of the
ideas used by Nadirashvili, allowed Morales~\cite{moral1} to
construct a conformal 0-immersion of the open unit disk that
is proper in $\R^3$. These same techniques were then applied by
Mart\'{i}n and Morales~\cite{marmor2} to prove that if ${\mathcal D}
\subset \rth$ is either a smooth open bounded domain or a possibly
 non-smooth open convex domain, then there exists a complete,
properly immersed 0-disk in ${\mathcal D}$;
{again, }
this disk cannot be
embedded by Theorem~\ref{thmCM}. In fact, embeddedness creates a
dichotomy in results concerning the Calabi-Yau questions, as we have
already seen in Theorems~\ref{thmCM} and \ref{thmmlct}.

In contrast to the existence results described in the previous
paragraph, Mart\'\i n and Meeks  have shown that there
exist many bounded non-smooth domains in $\rth$ which do not admit
any complete, properly immersed surfaces with bounded absolute mean curvature function
and at least one
annular end.
This generalized  their previous
joint work with  Nadirashvili~\cite{mmn}
in the minimal setting.

\begin{theorem}[Mart\'\i n and Meeks~\cite{mm1})]
\label{MMN}
Given any bounded domain ${\mathcal D}' \subset \rth$, there
exists a proper family ${\mathcal F}$ of horizontal simple closed curves
in ${\mathcal D}'$ such that the bounded domain ${\mathcal D} = {\mathcal D}' -
\bigcup {\mathcal F}$ does not admit any complete, connected properly immersed
surfaces with compact (possibly empty) boundary, an annular end and  bounded
absolute mean curvature function.
\end{theorem}

Ferrer, Mart\'\i n and  Meeks have given the following
general result on the classical Calabi-Yau
problem.

\begin{theorem} [Ferrer, Mart\'\i n and  Meeks \cite{fmm1}]
\label{calabi}
Let $M$ be an open, connected orientable surface
and let ${\mathcal D}$ be a domain in $\rth$ which is either convex or
bounded and smooth. Then, there exists a complete, proper minimal
immersion $f\colon M\to {\mathcal D}$.
\end{theorem}

The following conjecture and the earlier stated
Conjecture~\ref{CYconj1} are the two most important problems
related to the properness of complete $H$-surfaces
in $\rth$. In relation to the following conjecture for
complete 0-surfaces, one can ask whether there
{exists}
a complete, bounded non-compact $(H>0)$-surface in $\rth$.
The next problem is largely motivated and suggested by the work of  Mart\'\i n, Meeks,
Nadirashvili, P\'erez and  Ros.

\begin{conjecture}[Embedded Calabi-Yau Conjectures]
\label{CY1} \par .
\vspace{-.2cm}
\begin{enumerate}[1.]
\item A necessary and sufficient condition for a
connected, open topological surface $M$ to admit a
complete bounded minimal embedding in $\R^3$ is that every end of
$M$ has infinite genus.
\item A necessary and sufficient condition for a
connected, open topological surface $M$ to admit a proper minimal
embedding in every smooth bounded domain ${\mathcal D}\subset \R^3$ as a complete
surface is that $M$ is orientable and every end of $M$
has infinite genus.
\item A necessary and sufficient condition for a
connected, non-orientable open topological surface $M$ to admit a
proper minimal embedding in some bounded domain
${\mathcal D}\subset \R^3$ as a complete surface is that every end of
$M$ has infinite genus.
\end{enumerate}
\end{conjecture}

\section{The Hopf Uniqueness Problem.} \label{sec:Hopf}
There are two highly
influential results on the classification and geometric
description of $H$-spheres in
homogeneous 3-manifolds, see \cite{AbRo1,AbRo2,hf1}:

\begin{theorem}[Hopf Theorem]
An immersed $H$-sphere  in a complete, simply
connected 3-dimensional manifold $\Q^3(c)$ of constant sectional
curvature $c$ is a round sphere.
 \end{theorem}

\begin{theorem}[Abresch-Rosenberg Theorem] \label{abro}
An immersed $H$-sphere in a simply connected
homogeneous 3-manifold with a four-dimensional isometry group is
a rotationally symmetric immersed sphere.
\end{theorem}
When the ambient space is an arbitrary homogeneous 3-manifold
$X$, the type of description of immersed $H$-spheres
given by the above theorems is no longer possible, due to the lack of
continuous families of ambient rotations in $X$. Because of this, one
natural way to describe  immersed $H$-spheres in this general setting is
to parameterize explicitly the moduli space of these spheres up to ambient
isometries, and to determine their most important geometric properties.

In this section we describe  a theoretical framework for  studying
 immersed $H$-surfaces in  any simply connected
homogeneous 3-manifold $X$ that is not diffeomorphic to
$\esf^2\times \R $; in $\esf^2\times \R $, there is a unique
immersed $H$-sphere for each value of the mean curvature $H\in [0,\infty )$,
and each such immersed $H$-sphere is embedded as consequence of
Theorem~\ref{abro}.

The common framework for every simply connected
homogeneous 3-manifold $X$ not diffeomorphic to $\esf^2\times \R $ is
that such a $X$ is isometric to a metric Lie group, i.e., to a 3-dimensional
Lie group equipped with a left invariant metric.
For background material
on the classification
and geometry of 3-dimensional metric Lie groups,
the reader can consult the introductory textbook-style
article~\cite{mpe11} by the first two authors, and
for further details on the proofs outlined in this section, we refer
the reader to the papers~\cite{mmpr1,mmpr4,mmpr2}
by Meeks, Mira, P\' erez and Ros.

We wish to explain here how this general
theory leads to the classification and geometric
study of  immersed $H$-spheres when $X$ is compact.
Specifically, Theorem~\ref{main5} below gives a classification of
 immersed $H$-spheres in any homogeneous 3-manifold
diffeomorphic to $\esf^3$, and determines the essential
properties of such spheres with respect to their existence,
uniqueness, moduli space, symmetries, embeddedness and stability.
Since we will refer to smooth families of oriented $H$-spheres parameterized by the
values $H$ of their constant mean curvature, in this section we will allow
$H$ to be any real number.

\begin{theorem}[Meeks, Mira, P\'erez, Ros\, \cite{mmpr4}]
\label{main5} Let $X$ be a compact, simply connected homogeneous
3-manifold. Then:
\ben[1.]
\item \label{exist}
For every $H\in \R$,
there exists an immersed oriented sphere $S_H$ in $X$ of constant
mean curvature $H$.
\item \label{uniq}
 Up to ambient isometry, $S_H$
is the unique immersed sphere in $X$ with constant mean curvature
$H$.
\item \label{familySt}
There exists
a well-defined point in $X$ called
the  \emph{center of symmetry} of $S_H$
such that the isometries of $X$ that fix this point also leave
$S_H$  invariant.
\item \label{alex-embed}
$S_H$ is Alexandrov embedded,
in the sense that the immersion
$f\colon S_H\looparrowright X$ of $S_H$ in $X$ can be extended to an
isometric immersion $F\colon B\to X$ of a Riemannian 3-ball
 such that $\partial B=S_H$ is mean convex.
\item \label{index1}
$S_H$ has index one and nullity three for the Jacobi operator.
\end{enumerate}
Moreover, let $\mathcal{M}_X$ be the set of oriented  immersed
$H$-spheres  in $X$ whose center of symmetry
is a given point $e\in X$. Then, ${\mathcal M}_X$ is an analytic
family $\{ S(t) \mid t \in \R\} $ parameterized by the mean
curvature value $t$ of $S(t)$.
\end{theorem}

Every compact,  simply connected homogeneous 3-manifold is
isometric to the Lie group $\su $ given by (\ref{eq:su2}),
endowed with a left invariant
metric. There exists a 3-dimensional family of such homogeneous
manifolds, which includes the 3-spheres $\esf^3(c)$ of constant
sectional curvature $c>0$ and the two-dimensional family of
rotationally symmetric \emph{Berger spheres}, each of which has a
four-dimensional isometry group. Apart from these two more symmetric
families, any other  left invariant metric on $\su$ has a
3-dimensional isometry group, with the isotropy group of every
point being isomorphic to $\Z_2\times\Z_2$. Item~\ref{familySt} in
Theorem~\ref{main5} provides the natural generalization of the
theorems by Hopf and Abresch-Rosenberg to this more general context,
since it implies that any  immersed $H$-sphere $S_H$ in
such a space inherits all the ambient isometries fixing some point;
in particular, $S_H$ is round in $\esf^3(c)$
and rotationally symmetric in the Berger spheres.

Items~\ref{exist} and \ref{uniq} together with the last
statement of Theorem~\ref{main5} provide an explicit description of
the moduli space of  immersed $H$-spheres in any compact,
simply connected homogeneous 3-manifold $X$.
Items~\ref{alex-embed} and \ref{index1} in Theorem~\ref{main5}
describe general embeddedness and stability type properties of
immersed $H$-spheres in $X$ which are essentially sharp,
as we explain next. In $\esf^3(c)$,  immersed $H$-spheres
are round, embedded and weakly
stable (see Definition~\ref{defwstable} for the notion of weak
stability).
 However, for a general homogeneous $X$ diffeomorphic to
$\esf^3$,  immersed $H$-spheres need not be embedded
(Torralbo~\cite{tor1} for certain ambient Berger spheres) or weakly
stable (Torralbo and Urbano~\cite{tou1} for certain ambient Berger
spheres, see also Souam~\cite{so3}), and they are not geodesic
spheres if $X$ is not isometric to some  $\esf^3(c)$.  Nonetheless,
item~\ref{alex-embed} in Theorem~\ref{main5} shows that any
 immersed $H$-sphere in a general $X$ is {Alexandrov
embedded,}
a weaker notion of embeddedness,
while item~\ref{index1} describes the index and the
dimension of the kernel of the Jacobi operator of an  immersed $H$-sphere.

Just as in the classical case of $\rth$, the left invariant
Gauss map of an oriented  surface $\S$ in a metric Lie group
$X$ (not necessarily compact) takes values
in the unit sphere of the Lie algebra of $X$ and
contains
essential information on the geometry of the
surface, especially when $\S$ is an immersed $H$-surface.

\begin{definition}
\label{defG} {\rm Given an oriented immersed surface $f\colon \Sigma
\looparrowright X$ with unit normal vector field $N\colon \Sigma \to
TX$ (here $TX$ refers to the tangent bundle of $X$), we define the
{\it left invariant Gauss map} of the immersed surface to be the map
$G\colon \Sigma \to \esf^2\subset T_eX$ that assigns to each $p\in
\Sigma $ the unit tangent vector to $X$ at the identity element $e$
given by $(dl_{f(p)})_e(G(p))=N_p$. }
\end{definition}

An additional  property of the immersed $H$-spheres in $X$ that is not
listed in the statement of Theorem~\ref{main5} is that, after
identifying $X$ with the Lie group $\su$ endowed with a left
invariant metric, the left invariant Gauss map of every immersed $H$-sphere
 in $X$ is a diffeomorphism to $\esf^2$; this diffeomorphism
property is crucial in the proof of Theorem~\ref{main5} and follows from the
next more general result.

\begin{theorem}[Theorem~4.1 in \cite{mmpr4}]
\label{thm:index1} Any index-one $H$-sphere $S_H$ in a
3-dimensional, simply-connected metric Lie group $X$ satisfies:
\ben[(1)]
\item The left invariant Gauss map of $S_H$ is an orientation-preserving diffeomorphism
to $\esf^2$.

\item   $S_H$ is unique up to left translations among
$H$-spheres in $X$.

\item $S_H$ lies inside a real-analytic family $\{S_{H'} \mid H'\in
(H-\varepsilon , H+\varepsilon )\}$ of index-one spheres in $X$ for
some $\ve>0$, where $S_{H'}$ has constant mean curvature of value
$H'$.
\een\end{theorem}


As an application of Theorem~\ref{main5}, Meeks, Mira, P\'erez and Ros
provide  a more detailed description of the special
geometry of  immersed  0-spheres in a general compact $X$.

\begin{theorem}[Theorem~7.1 in \cite{mmpr4}]\label{embed:su2}
For $X$ as in Theorem~\ref{main5}, the unique
(up to left translations)
immersed 0-sphere  $S_0$ in $X$ is embedded. Furthermore, since
the stabilizer of any point in a left invariant
metric on $\su$ contains $\Z_2\times \Z_2$ and
$S_0$ is also invariant under the antipodal map $A \mapsto -A$,
then by item~\ref{familySt} {of Theorem~\ref{main5}}
the related group of
ambient isometries $G=\Z_2\times \Z_2\times \Z_2$ leaves
 $S_0$ invariant; in fact when the isometry group of $X$ is
 3-dimensional, then $G$ is the subgroup of
 ambient isometries of $X$ that leaves $S_0$ invariant.
\end{theorem}

Since it is well-known~\cite{smith1} that for every Riemannian
metric on $\esf^3$, there exists an
embedded minimal
sphere, then one could apply the uniqueness statement in
item~\ref{uniq} of Theorem~\ref{main5}
to give an alternative proof that $S_0$ is embedded.

The theorems of Hopf and Abresch-Rosenberg rely on the existence of
a holomorphic quadratic differential for  immersed $H$-surfaces
in homogeneous 3-manifolds with isometry group of
dimension at least four. This approach using holomorphic quadratic
differentials does not seem to work when the isometry group of the
homogeneous 3-manifold has dimension three. Instead, the
approach to proving Theorem~\ref{main5} is inspired by two recent
works on  immersed $H$-spheres in the Thurston geometry
$\sol$, i.e., in the solvable Lie group Sol$_3$ equipped with its
standard left invariant metric. One of these works is the local
parameterization by Daniel and Mira~\cite{dm2} of the space
$\mathcal{M}_{\mbox{\footnotesize Sol}_3}^1$ of index-one,  immersed
$H$-spheres in $\sol$ equipped with its standard metric,
via the left invariant Gauss map and the uniqueness of such spheres.
The other one is  Meeks'~\cite{me34}  area
estimates for the subfamily of spheres in
$\mathcal{M}_{\mbox{\footnotesize Sol}_3}^1$ whose mean curvatures
are bounded from below by any fixed positive constant; these two
results lead to a complete description of the  immersed $H$-spheres
in $\sol$ endowed with its standard metric. However,
the proof of Theorem~\ref{main5} for a general compact, simply
connected homogeneous 3-manifold $X$ requires the development of
new techniques and theory, {which}
are needed to prove that the left invariant Gauss map of an index-one
 immersed $H$-sphere  in $X$ is a diffeomorphism, that
 immersed $H$-spheres in $X$ are {Alexandrov embedded}
and have a center of symmetry, and that there exist a priori area
estimates for the family of index-one  immersed $H$-spheres in $X$.

Here is a brief outline of the proof of Theorem~\ref{main5}. One first
identifies the compact, simply connected homogenous 3-manifold $X$
isometrically with $\su$ endowed with a left invariant metric. Next,
one shows that any index-one  immersed $H$-sphere $S_H$
in $X$ has the property that any other immersed sphere of the same
constant mean curvature $H$ in $X$ is a left translation of $S_H$.
The next step in the proof is to show that the set $\cH$ of values
$H\in \R$ for which there exists an index-one  immersed $H$-sphere
in $X$ is non-empty, open and closed in $\R$
(hence, $\cH=\R$). That $\cH$ is non-empty follows from the
existence of isoperimetric spheres in $X$ of small volume. Openness of $\cH$
follows from an application of the implicit function theorem,
an argument that
also proves that the space of index-one  immersed $H$-spheres
in $X$ modulo left translations is an analytic
one-dimensional manifold. By elliptic theory, closedness of $\cH$
can be reduced to obtaining a priori area and curvature estimates
for index-one,  immersed $H$-spheres with any fixed upper
bound on their mean curvatures. The existence of these curvature
estimates is obtained by a rescaling argument. The most delicate
part of the proof of Theorem~\ref{main5} is obtaining a priori  area
estimates; for this, one first shows that the non-existence of an
upper bound on the areas of all  immersed $H$-spheres in
$X$ implies the existence of a complete, stable, constant mean
curvature surface in $X$ that can be seen to be the lift via a
certain fibration $\Pi\colon X\to \esf^2$  of an immersed curve in
$\esf^2$, and then one proves that such a surface cannot be stable to
obtain a contradiction. This contradiction completes the proof of the fact
that index-one  immersed $H$-spheres exist for all values
of $H$, and so, they are the unique  immersed $H$-spheres
in $X$. The {Alexandrov embeddedness}
 of  immersed $H$-spheres
follows from a deformation argument, using the
smoothness of the family of  immersed $H$-spheres in $X$
and the maximum principle
for $H$-surfaces in Theorem~\ref{thmintmaxprin}.
Finally, the existence of a center of symmetry for any  immersed
$H$-sphere in $X$ is deduced from the
{Alexandrov embeddedness}
and the uniqueness up to left translations of the sphere.

We next describe  some key results and definitions that are essential
in pushing forward and generalizing the arguments
for classifying immersed $H$-spheres described above in the compact case to
the setting where the metric Lie
group is diffeomorphic to $\rth$.

\begin{definition}
\label{Crit-Cheg} {\em
Let $Y$ be a complete homogeneous 3-manifold.
\ben
\item The {\em critical mean curvature} $H(Y)$ of $Y$ is defined as
\[
H(Y)=\inf \{\max|H_{M}| :
M \mbox{ is an immersed closed surface in }Y \},
\]
where  $\max|H_{M}| $ denotes  the maximum of the
absolute mean curvature function $H_{M}$.
\item The {\em Cheeger constant} Ch$(Y)$ of $Y$ is defined as
\[
\mbox{Ch}(Y)=\inf_{\stackrel{K \subset Y}{\mathrm{ \; compact}}}
 \frac{{\rm Area}(\partial K)}{{\rm Volume}(K)}.
\]
\een
}
\end{definition}

The strategy of Meeks, Mira, P\'erez and Ros in~\cite{mmpr1}
to generalize Theorem~\ref{main5} to the case
where $X$ is diffeomorphic to $\rth$ is to obtain
a  result similar to Theorem~\ref{main5}  except that in this case, index-one $H$-spheres
in $X$ exist  precisely for the values $H\in (H(X),\infty)$; note that the definition
of $H(X)$ only permits immersed $H$-spheres to occur in $X$ if $H\geq H(X)$ and the case
of $H=H(X)$ is also easily ruled out.
The expected proof of this generalization of  Theorem~\ref{main5}
follows the same general reasoning as
the outline given above. At the present moment, the main difficulty
in completing the proof of this final result on the Hopf Uniqueness Problem
is to obtain the following area estimates for immersed $H$-spheres in $X$:
\begin{quote}
$(\star )$\ {\em For any $\ve>0$, there exists $A(X,\ve)>0$ such that every index-one
$H$-sphere in $X$ with $H\in (H(X)+\ve,\infty)$
has  area less than $A(X,\ve)$.}
\end{quote}
Meeks, Mira, P\'erez and Ros in~\cite{mmpr1}
are presently writing up the proof of
these area estimates in an essentially
case-by-case study of the possible metric Lie groups  $X$ that can occur
when $X$ is diffeomorphic to $\rth$.

We end this section with some comments about the geometry of
solutions of the isoperimetric problem in
metric Lie groups diffeomorphic to $\rth$ and the relationship between the two
constants $H(X), \mbox{\rm Ch}(X)$ in Definition~\ref{Crit-Cheg}.
In this subject there are still many important open problems concerning $H$-surfaces
in simply connected
homogeneous 3-manifolds, and we refer the interested reader to the last section
of~\cite{mpe11} for a long list of them; however,
we mention some of our favorite ones below related to the isoperimetric problem in
metric Lie groups diffeomorphic to $\rth$.

\begin{conjecture}[Isoperimetric Domains Conjecture] \label{iso:conj}
Let  ${ X}$ denote a metric Lie group diffeomorphic to $\rth$. Then:
 \begin{enumerate}[1.]
\item  Isoperimetric domains in ${ X}$ are topological balls. More generally, closed
Alexandrov embedded $H$-surfaces in
$X$ are spheres.
\item Immersed $H$-spheres in $X$ are embedded, and the balls that
they bound are isoperimetric domains.
\item For each fixed volume ${ V}_0$, solutions to
the isoperimetric problem in  ${X}$
for volume ${ V}_0$ are unique up to left translations in ${ X}$.
\end{enumerate}
\end{conjecture}

There seems to be no direct method for computing the
critical mean curvature $H(X)$ of a metric Lie group $X$
diffeomorphic to $\rth$,
whereas when $X$ is not isomorphic to the universal
cover $\sl$ of the special linear group
SL$(2,\R )$, it is straightforward to compute
the more familiar Cheeger constant of $X$ directly
from its metric Lie algebra, see~\cite{mpe11} for this computation.
Since  the validity of Conjecture~\ref{iso:conj} would imply that
$H$-spheres in $X$ are the boundaries of
isoperimetric domains, then it is perhaps not too surprising
that one has the following result.

\begin{theorem}[Meeks, Mira, P\'erez, Ros \cite{mmpr2}]
If \,$Y$ is a simply connected homogeneous 3-manifold,
then $2\, H(Y)=\mbox{\rm Ch}(Y)$.
Furthermore, if $\Delta_n$ is a sequence of isoperimetric
domains in $X$ with diverging volumes, then,
as $n\to \infty$, the mean curvatures of their boundary surfaces
converge to $H(X)$ and
the radii $R_n$ of  $\Delta_n$
converge to infinity.
\end{theorem}

\begin{remark} \label{remark-integration}
{\em
Another theoretical tool developed by Meeks, Mira, P\'erez and  Ros
in~\cite{mmpr4} is a conformal PDE that the stereographic projection $g$
of the left invariant Gauss map of
an immersed $H$-surface in a simply connected 3-dimensional Lie group $X$ must
satisfy; this PDE depends on the value of $H$ and
invariants of its metric Lie algebra of $X$. Conversely, it follows
from the representation  Theorem~3.7 in~\cite{mmpr4} that
any function $g\colon M\to \C\cup \{\infty\}$
on a simply connected Riemann surface satisfying this PDE,
can be integrated
to obtain a conformal $H$-immersion of $M$ into $X$ with $g$
as its stereographically projected left invariant Gauss map.
}
\end{remark}

\section{CMC foliations.} \label{sec:CMC}
{This section}
is devoted to results
on the existence and geometry of
CMC foliations of Riemannian $n$-manifolds.

\subsection{The classification of singular CMC foliations of $\rth$.}

The  following classification  theorem is stated for weak CMC foliations of $\rth$,
which are similar to weak $H$-laminations
in that the leaves can intersect non-transversely, and where two such leaves
intersect at a point, then locally they lie on one side of the other one near this
point; for the definition of the more general notion of
a weak CMC lamination, see Definition~\ref{definition}.
Critical to its proof are the existence of curvature estimates
given in Theorem~\ref{thm5.7} for weak CMC foliations
of any Riemannian 3-manifold with bounded absolute sectional curvature.
The following result
generalizes the classical theorem of Meeks~\cite{me17} that the only CMC foliations of $\rth$
are foliations by parallel planes.

\begin{theorem}[Meeks, P\'erez, Ros \cite{mpr21}]
\label{thmspheres}
Suppose that ${ \cal F}$ is
a weak  CMC foliation of $\rth$ with a  closed
countable set $ \cal S$ of singularities (these are the points
where the weak CMC structure of ${\cal F}$ cannot be extended).  Then,
each leaf of ${\cal F}$ is contained in either a plane or a round sphere,
and ${\cal S}$ contains at most 2 points. Furthermore
if $\cal S$ is empty, then
$\cal F$ is a foliation by planes.
\end{theorem}

The simplest examples of weak CMC foliations of $\rth$ with a closed
countable set of singularities are families of
parallel planes or concentric spheres around a given point. A
slightly more complicated example appears when considering a
family of pairwise disjoint planes and spheres as in
Figure~\ref{figspheres}, where the set ${\cal S}$ consists of two points.
\begin{figure}
\begin{center}
\includegraphics[width=10cm]{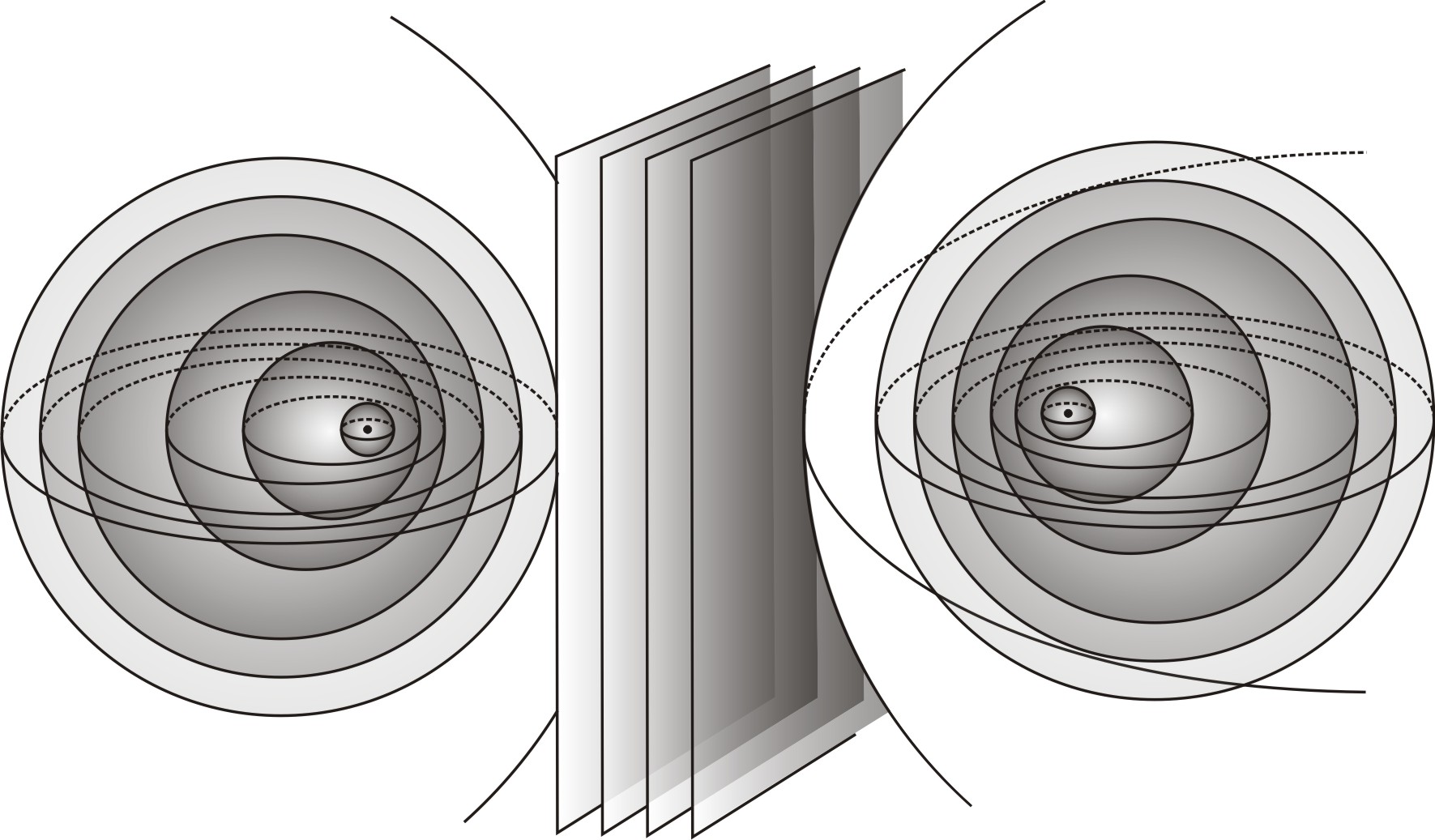}
\caption{A foliation of $\R^3$  by spheres and planes with two
singularities. }
\label{figspheres}
\end{center}
\end{figure}

In the case of the unit 3-sphere $\esf^3\subset \R^4$ with its
constant 1 sectional curvature, we obtain a similar result:
\par
\vspace{.2cm}
{\it The leaves of every weak CMC foliation   of $\esf^3$ with a closed
countable set $ \cal S$ of singularities are contained in round spheres,
and ${\cal S}$ consists of 1 or 2 points.
}
\par
\vspace{.2cm}

We note that in the
statement of the above theorem, we made no assumption on the
regularity of the foliation ${\cal F}$. However, the proofs
require that ${\cal F}$ has bounded second fundamental
form on compact sets of $N=\R^3$ or $\esf^3$ minus the singular set
${\cal S}$. This bounded curvature assumption always holds
for a topological CMC foliation by
recent work of Meeks and Tinaglia~\cite{mt7,mt1} on curvature estimates for embedded, non-zero constant
mean curvature disks and a related 1-sided curvature estimate for embedded surfaces
of any constant mean curvature (see Theorem~\ref{th} in the $\rth$-setting and observation
{(O.2)}
in Section~\ref{section6});  in
the case that all of the leaves of the lamination of a 3-manifold are
minimal, this 1-sided curvature estimate was given earlier by
Colding and Minicozzi~\cite{cm23} which also holds in the 3-manifold setting.

Consider a foliation ${\cal F}$ of a Riemannian 3-manifold $N$
with leaves having constant absolute mean curvature, with this constant possibly depending on
the given leaf.
After possibly passing to a four-sheeted cover, we can assume $X$ is oriented and that all
leaves of ${\cal F}$ are oriented consistently, in the sense that
there exists a continuous, nowhere zero vector field in $X$ which is
transversal to the leaves of ${\cal F}$. In this situation, the mean
curvature function of the leaves of ${\cal F}$ is well-defined and so $\cF$ is a CMC foliation.
Therefore, when analyzing the structure of such a CMC foliation ${\cal F}$,
it is natural to consider for each $H\in \R $,
the subset ${\cal F}(H)$ of ${\cal F}$ of those leaves that have
mean curvature $H$. Such a subset ${\cal F}(H)$ 
is closed since the mean curvature function is continuous on
${\cal F}$; ${\cal F}(H)$ is an example of an $H$-lamination. A
cornerstone in proving Theorem~\ref{thmspheres} is to analyze the
structure of an $H$-lamination ${\cal L}$ (or more generally, a weak
$H$-lamination, see Definition~\ref{definition}) of a punctured ball
in a Riemannian 3-manifold, in a small neighborhood of the
puncture. This local problem can be viewed as a desingularization
problem, see Theorem~\ref{tt2}.

Besides Theorem~\ref{tt2}, a second key ingredient is needed in
the proof of Theorem~\ref{thmspheres}: a
universal scale-invariant curvature estimate valid
for any weak CMC foliation of a compact
Riemannian 3-manifold with boundary, solely in terms
of an upper bound for its sectional curvature. The next result is inspired by
previous curvature estimates described in Section~\ref{subsecJacobi}
for stable constant mean curvature surfaces.

\begin{theorem}[Curvature Estimates for CMC foliations, \,\cite{mpr21}]
\label{thm5.7}
There exists a constant $C>0$ such that the following statement holds.
Given $\Lambda \geq 0$, a compact Riemannian 3-manifold $X$ with boundary
whose absolute sectional curvature is at most $\Lambda $, a weak $CMC$ foliation
${\cal F}$ of $X$ and a point $p\in \rm Int(X)$, we have
\[
|A_{\cal F}|(p)\leq \frac{C}{\min\{ d_X(p,\partial X),
\frac{\pi }{\sqrt{\Lambda }}\}},
\]<
where $|A_{\cal F}|\colon X\to [0,\infty )$ is the function that assigns to each $p\in X$
the supremum of the norms of the second fundamental forms of leaves of ${\cal F}$ passing through
$p$,
and $d_X$ is the Riemannian distance in $X$.
\end{theorem}

If $\cF$ were a non-flat weak CMC foliation of $\rth$, then the norms
of the second fundamental forms of foliations obtained by scaling $\cF$ by $\frac1n$, $n\in \N$,
are not uniformly bounded which contradicts the conclusions
of Theorem~\ref{thm5.7}. This contradiction
proves that
the only CMC foliations of $\rth$
are foliations by parallel planes.

The above curvature estimate is also an essential tool for
analyzing the structure of a weak CMC
foliation of a small geodesic Riemannian 3-ball
punctured at its center.  Among other things, in~\cite{mpr21}
Meeks, P\'erez and Ros proved that if the mean curvatures
of the leaves of such a weak CMC foliation  are bounded in a neighborhood
of the puncture, then  the weak CMC foliation extends
across the puncture to a weak CMC foliation of the ball.
Theorem~\ref{thmspheres} and a blow-up argument lead to a
model for the structure of a weak CMC foliation of a punctured ball
in any Riemannian 3-manifold. From here, one can deduce
that a compact, orientable Riemannian 3-manifold not diffeomorphic to
the 3-sphere $\esf^3$
does not admit any weak (transversely oriented) CMC foliation with a non-empty
countable closed set of singularities; see~\cite{mpe14}
for this and other related results.

\subsection{CMC foliations of closed $n$-manifolds.} \label{sec:closedH}
By the next  theorem by Meeks and P\'erez, the vanishing of the
Euler characteristic of a  closed
$n$-manifold $X$ is equivalent to  the existence of a CMC foliation of $X$ with respect
to some Riemannian metric.  In the case $X$ is orientable, this theorem
was proved by Oshikiri~\cite{osh3}; we emphasize  that the proof of  Theorem~\ref{mainCMC} below
by Meeks and P\'erez in~\cite{mpe13} does not use Oshikiri's
results.
Furthermore, when $n\geq 3$ the CMC foliations $\cF$ that we construct
on $X$ with vanishing Euler characteristic
satisfy that there are a finite number of components of the complement
of the sublamination of minimal leaves
in $\cF$ such that each of these
foliated components is diffeomorphic to the product of an open $(n-1)$-disk $D$
and a circle $\esf^1$, with isometry group containing $SO(n-1)\times \esf^1$;
furthermore, the universal cover $D\times \R$ of each such component together with its
lifted foliation and metric are equivalent to a rather explicit CMC
foliation $\cF_n$ on  $D\times \R$ with a product metric $g_n$, such that this structure is
invariant under the action of $SO(n-1)\times \R$ and depends only on the dimension $n$;
see Figure~\ref{Reeb}.
\begin{figure}
\begin{center}
\includegraphics[width=2in]{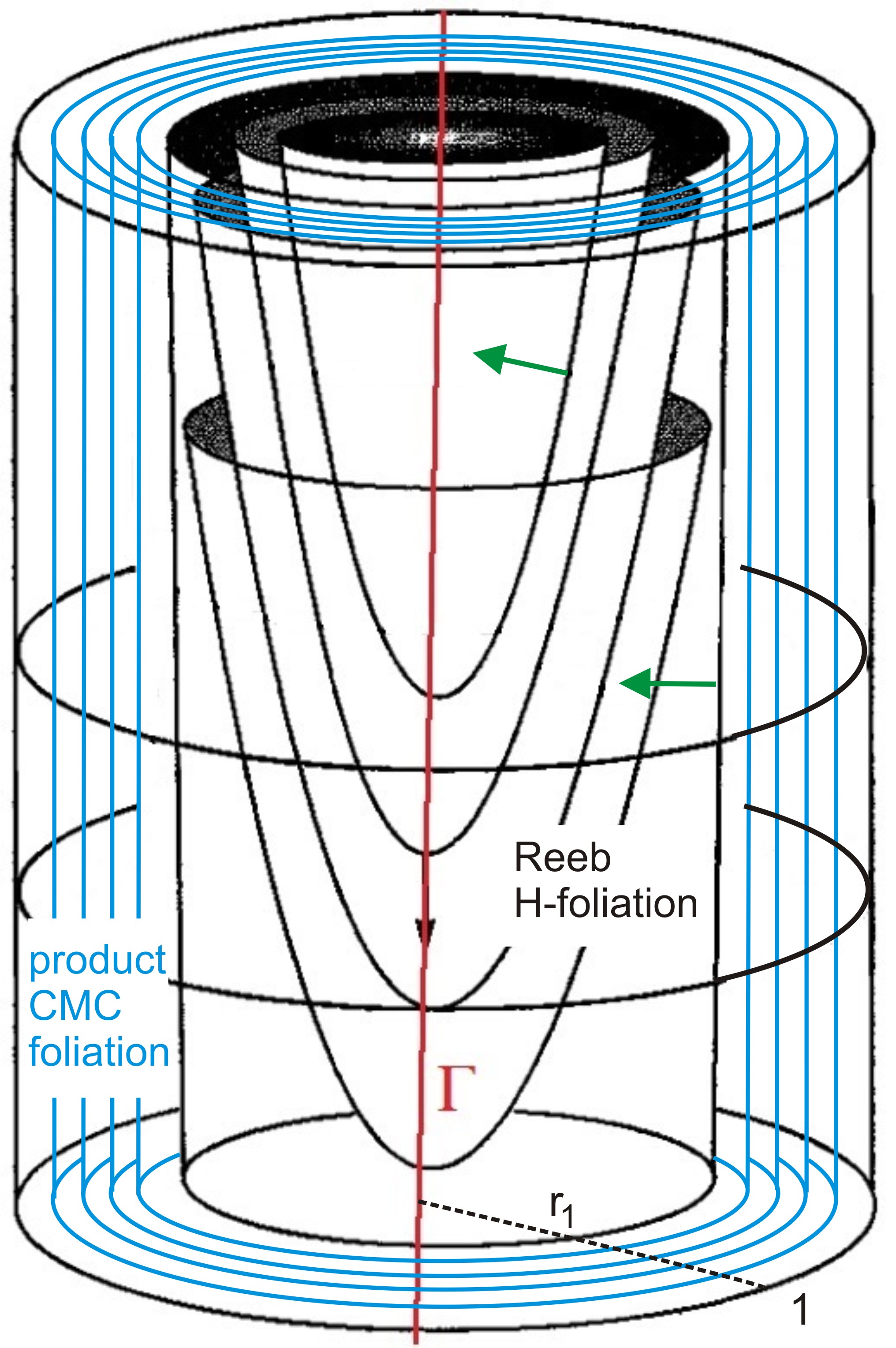}
\caption{This figure depicts the structure of a generalized  Reeb foliation
on $D\times \R$ with
leaves of constant mean curvature $H$.}
\label{Reeb}
\end{center}
\end{figure}

Recall that by definition, a CMC foliation is necessarily smooth.

\begin{theorem}[Existence Theorem for CMC Foliations] \label{mainCMC}
A closed   $n$-manifold admits a CMC foliation
for some Riemannian metric 
if and only if its Euler characteristic is zero. When $n\geq 2$,
the CMC foliation can be taken to be non-minimal.
\end{theorem}

Since closed (topological) 3-manifolds admit smooth structures and
the Euler characteristic of any closed manifold
of odd dimension is zero, the 
previous theorem has the following corollary.

\begin{corollary} \label{cor1.2}
Every closed topological 3-manifold 
admits a smooth structure together with
a  Riemannian metric and a non-minimal CMC
foliation.
\end{corollary}

The proof of Theorem~\ref{mainCMC} is motivated by two seminal
works. The first one, due to Thurston
(Theorem 1(a) in~\cite{th3}), shows
that a necessary and sufficient condition for a smooth closed
$n$-manifold $X$ to admit a
smooth, codimension-one foliation $\cF$ is for its Euler characteristic to vanish;
for our applications, $\cF$ can be chosen to be  transversely oriented.
The second one is the result by
Sullivan (Corollary~3 in~\cite{sul1}) that given such a pair
$(X,\cF)$ where $\cF$ is orientable (this means that the subbundle
of the tangent bundle to $X$ which is tangent
to the foliation is orientable),
then $X$ admits a smooth Riemannian metric
$g_X$ for which $\cF$ is a minimal foliation (this is called $\cF$
is {\it geometrically taut}) if and only if
for every compact leaf $L$ of $\cF$ there exists a closed transversal that intersects $L$
(called $\cF$ is {\it homologically taut}); in the proof of
Theorem~\ref{mainCMC}, it is needed the generalization
of the implication `homologically taut $\Rightarrow $ geometrically taut' without
Sullivan's hypothesis that the foliation $\cF$ be orientable.

In the case when $n=2$, Theorem~\ref{mainCMC} follows by giving
explicit examples.
Consider the curve
$\a=\{(t, 3 +\cos t )\mid t\in \R\}$ in the $(x_1,x_2)$-plane and let $C$ in $\rth$
be the surface obtained by revolving $\a$ around the $x_1$-axis. Let $\cF$ be the
foliation of $C$ by circles contained in planes orthogonal to the $x_1$-axis,
whose leaves have constant geodesic curvature, see Figure~\ref{revolution}.
\begin{figure}
\begin{center}
\includegraphics[height=6.5cm]{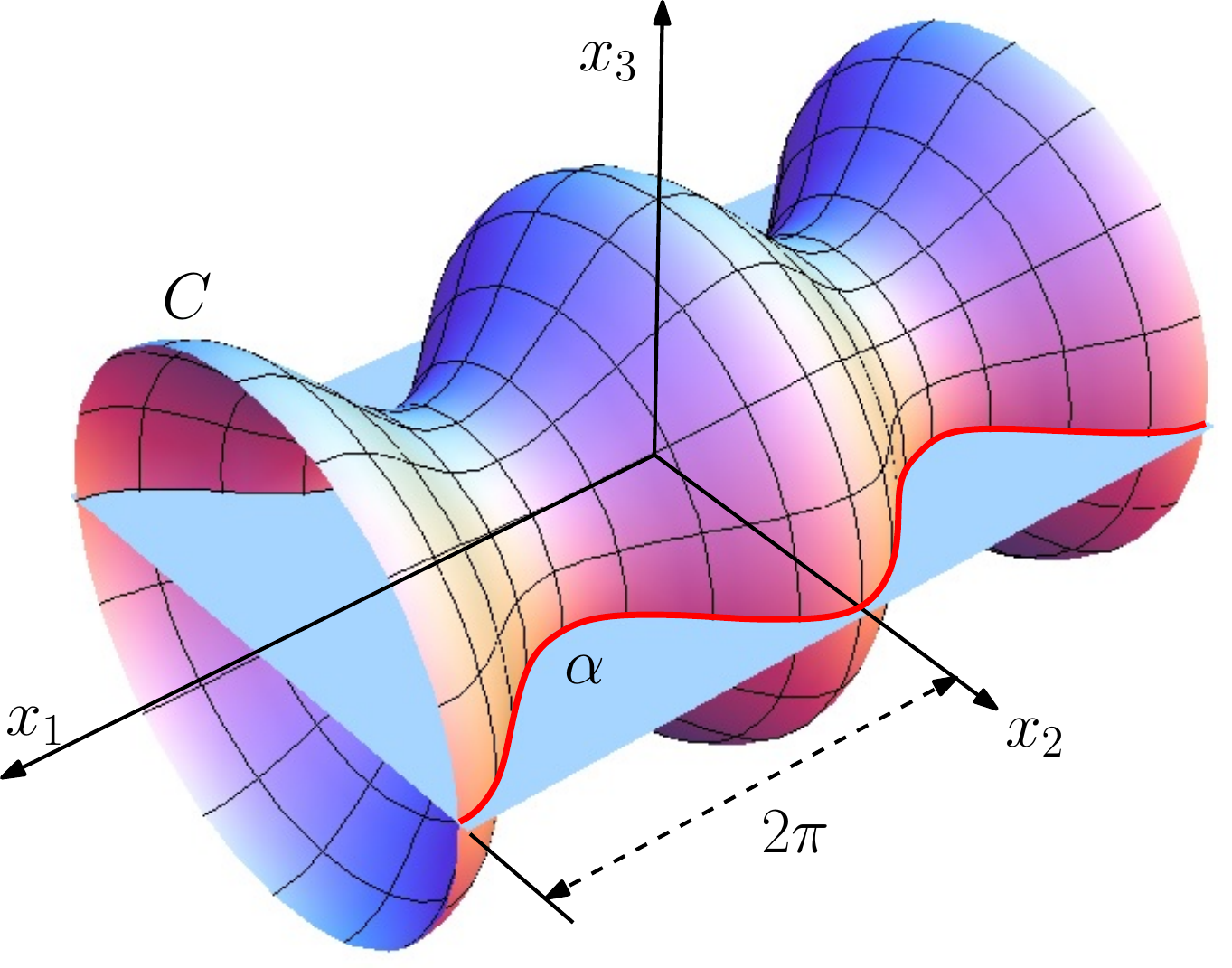}
\caption{The circles in $C$ foliate the surface by curves of constant
geodesic curvature. The symmetry $R$ is
the composition of the reflection in the $(x_1,x_2)$-plane
(depicted in the figure) with the translation by the vector $(2\pi ,0,0)$.}
\label{revolution}
\end{center}
\end{figure}
$\cF$ is transversely oriented by the normal vectors to the circles in $C$
that have positive inner product in $\rth$ with $\partial_{x_1}$.
Since the map  $R(x_1,x_2,x_3)= (2\pi +x_1,x_2,-x_3)$
preserves the transverse orientation of the CMC foliation,
then $\cF$ descends to a CMC foliation
of the Klein bottle $C/R$ or to the torus $C/(R^2)$.
By classification of closed surfaces, a closed surface with Euler characteristic zero
must be a torus or a Klein bottle.
Thus, Theorem~\ref{mainCMC} trivially holds when $n=2$.

So  assume $n\geq 3$ and we will give a sketch of the proof of
Theorem~\ref{mainCMC} in this case.
One first studies the existence of
codimension-one,  $(SO(n-1)\times \R)$-invariant CMC
foliations $\cR_{n-1} $ of the Riemannian product of the
real number line $\R$ with the closed unit $(n-1)$-disk $\ov{\D}(1)\subset \R^{n-1}$
with respect to a certain $SO(n-1)$-invariant metric,
see Figure~\ref{Reeb}. The leaves of this
foliation $\cR_{n-1} $ are of one of two types:
those leaves that intersect $\D(r_1)\times \R $
(here $\D(r_1)=\{ x\in \R ^{n-1} \ | \ |x| <r_1\} $ and
$0<r_1<1)$ are rotationally symmetric
hypersurfaces which are graphical
over $\D (r_1)\times \{ 0\} $ and asymptotic to the vertical
$(n-1)$-cylinder $\esf ^{n-2}(r_1)\times \R $;
the remaining leaves of $\cR_{n-1} $ are the
vertical cylinders $\esf ^{n-2}(r)\times \R $, $r\in [r_1,1]$.
All leaves of $\cR_{n-1} $
in $\D^{n-1}(r_1)\times \R $
are vertical translates of a single such leaf (in particular,
they all have the same constant mean
curvature, equal to the constant value of the mean curvature
of $\esf ^{n-2}(r_1)\times \R $),
while the (constant) mean curvature values of the
cylinders $\esf ^{n-2}(r)\times \R $, $r\in [r_1,1]$,
vary from leaf to leaf. This foliation $\cR_{n-1} $
gives rise under the quotient action
of $\Z\subset \R$ to what we call an {\it enlarged
foliated Reeb component} $\cR_{n-1}/\Z$,
that is diffeomorphic to $\overline{\D}(1)\times \esf^1$.

The sufficient implication in Theorem~\ref{mainCMC} follows directly
from the Poincar\'e-Hopf index theorem.
As for the necessary
implication, the results in~\cite{th3} imply that
a smooth, closed   $n$-manifold $X$ with Euler
characteristic zero admits a smooth,
transversely oriented foliation $\cF'$ of codimension one.
After a simple modification of $\cF'$ along some smooth
simple closed curve $\G$ transverse
to the foliation by the classical technique of turbularization, $\cF'$
can be assumed to have at least one non-compact leaf.  Recall that in this
process one modifies the previous foliation in a small tubular neighborhood of $\G$ and
one ends up with a new foliation where we  have introduced what we called
in the previous paragraph a generalized Reeb component
centered along $\G$. Then one proves the existence of a finite collection
$\Delta=\{\g_1,\ldots,\g_k\}$ of pairwise
disjoint, compact embedded arcs in $X$ that are transverse to the leaves of $\cF'$
and such that every  compact leaf of the foliation intersects at least one of
these arcs; this existence result  follows from work of
Haefliger~\cite{hae2} on the compactness of the set of compact leaves
of any codimension-one foliation of~$X$. The next step consists of modifying
$\cF'$ using again turbularization by introducing pairs of
enlarged Reeb components, one pair for each $\g_i\in \Delta$.
These modifications give rise to a new transversely oriented foliation $\cF$.
By a careful application of a generalization
of Sullivan's theorem to the case of non-orientable codimension-one
foliations, one can check that in the complement
of the sewn in generalized Reeb components in $\cF$,
the resulting manifold $\wh{X}$ with boundary admits a metric so that all of the leaves
of the restricted foliation are minimal and in a neighborhood
of $\partial \wh{X}$ where the foliation is a product
foliation, the metric is also a product metric.
Then one proceeds by extending this minimal metric on $\wh{X}$ to $X$ and so that
the regions modified by  turbularization now have the rotationally invariant metrics
mentioned in the previous paragraph. Crucial in
obtaining this smooth metric on $X$ so that all
of the leaves of $\cF$ have constant mean curvature, is the application  of
the classical Theorem~2 in Moser~\cite{moser1}. This smooth ``gluing" result of Moser
is closely related to his following well-known classical result:
If $g_1,g_2$ are two metrics with respective volume forms $dV_1,dV_2$
on a closed orientable Riemannian $n$-manifold~$Y$ with the same total volume,
then there exists an orientation preserving diffeomorphism $f\colon Y\to Y$ such that
$f^*(dV_1)=dV_2$; furthermore, $f$ can be chosen isotopic to the identity.

We finish this article mentioning another result from~\cite{mpe13},
which is the Structure Theorem~\ref{mainCMC2}
given below on the geometry and topology of non-minimal CMC foliations of
a closed   $n$-manifold. Before stating this theorem, we fix some notation
for a CMC foliation $\cF$ of a (connected) closed  Riemannian
$n$-manifold $X$:
\bit
\item $N_\cF$ denotes the unit normal vector field
to $\cF$ whose direction coincides
with the given transverse orientation.
\item  $H_{\cF}\colon X\to \R$ stands for the
{\em mean curvature function} of $\cF$ with respect to $N_{\cF}$.
\item $H_\cF (X)=[\min H_{\cF},\max H_{\cF}]$ is the image of $H_\cF$.
\item $\cC_\cF$ denotes the union of the compact
leaves in $\cF$, which is a compact
subset of $X$ by the aforementioned result of Haefliger~\cite{hae2}.
\eit

\begin{theorem}[Structure Theorem for CMC Foliations] \label{mainCMC2}
Let $(X,g)$ be
a closed, connected Riemannian $n$-manifold
which admits a non-minimal CMC foliation $\cF$. Then:
\ben[1.]
\item  \label{it1}  $\int_X H_{\cF} \,dV=0$ and so,
$H_{\cF}$ changes sign (here $dV$ denotes the volume element with respect to $g$).
\item  \label{it2} For $H$ a regular value of $H_\cF$,  $H_{\cF}^{-1}(H)$ consists
of a finite number of compact leaves of $\cF$ contained in $\Int(\cC_\cF)$.
\item  \label{it3} $X-\cC_\cF$ consists of a countable number of open
components and the leaves in each such
component $\Delta $ have the same mean curvature
as the finite positive number of compact leaves in $\partial \Delta $;
furthermore, every leaf in the closure of $X-\cC_\cF$ is stable.
In particular, except for a countable subset of $H_\cF(X)$, every leaf of $\cF$
with mean curvature $H$ is compact, and for every $H\in H_\cF(X)$,
there exists at least one compact leaf of $\cF$ with mean curvature $H$.
\item \label{it4}
\ben[a.] \item Suppose that $L$ is a leaf of $\cF$ that contains
a regular point of $H_\cF$. Then $L$ is compact, it consists
entirely of regular points of $H_\cF$ and lies
in $\Int(\cC_\cF)$. Furthermore, $L$ has 
index zero if and only if the function $g( \nabla H_\cF, N_\cF)=N_{\cF}(H_{\cF})$ is
negative along $L$, and if the index of  $L$
is zero, then it also has nullity zero.
\item  Suppose that $L$ is a leaf of $\cF$ that
is disjoint from the regular points
of $H_\cF$. Then the index of $L$ is zero, and if $L$ is a limit
leaf\,\footnote{See Definition~\ref{def-limset} for
the definition of a limit leaf of a {codimension-one}
lamination.}
of the CMC lamination of $X$ consisting of the compact leaves of $\cF$,
then $L$ is compact with nullity one.
\een
\item \label{it5} Any leaf of $\cF$ with mean curvature equal
to $\min H_\cF$ or $\max H_\cF$
is stable and such a leaf can  be chosen to be compact with nullity one.
\een
\end{theorem}

\section{Outstanding problems and conjectures.}  \label{sec:conj}
In this last section, we present many of the  fundamental conjectures
in minimal and constant mean curvature surface theory.
In the statement of most of these conjectures we have
listed the principal researchers to whom the given conjecture might
be attributed and/or those individuals who have made important
progress in its solution.

\subsection{Conjectures in the classical case of $\rth$.}

The classical Euclidean isoperimetric inequality states that the inequality
$4\pi A\leq L^2$ holds for the area $A$ of a compact subdomain of $\R^2$
with boundary length $L$, with equality if and only if $\Omega $ is a round
disk. The same inequality is known to hold for compact minimal surfaces with
boundary in $\R^3$ with at most two boundary components
(Reid~\cite{red1}, Osserman and Schiffer~\cite{osSchi1},
Li, Schoen and Yau~\cite{LSY1}, see also Osserman's survey paper~\cite{os5}).
\begin{conjecture} [Isoperimetric Inequality Conjecture]
\label{isop}
Every connected, compact minimal surface $\Omega $ with boundary in $\R^3$ satisfies
\begin{equation} \label{isoper}
4\pi A\leq L^2,
\end{equation}
where $A$ is the area of $\Omega $ and $L$ is the length of its boundary. Furthermore,
equality holds if and only if $\Omega $ is a planar round disk.

More generally,  if $\Omega $ is a connected, compact surface with
boundary in a Hadamard 3-manifold $N$ with sectional curvature at most $-a^2$,
and the absolute mean curvature function of $\Omega $ is at most $|a|\geq0$, then
equation~\eqref{isoper} is satisfied and equality holds if and
only if $\Omega $ is a planar round disk in a flat totally umbilic simply-connected
hypersurface of constant mean curvature $|a|$.
\end{conjecture}

Gulliver and Lawson~\cite{gl1} proved that if $\Sigma$ is an orientable, stable
minimal surface with compact boundary that is properly
embedded in the punctured unit ball $\B -\{ \vec{0}\} $ of $\rth$, then its closure is
a compact, embedded  minimal surface.
If $\Sigma$ is not stable, then the
corresponding result is not known.
Meeks, P\'erez and
Ros~\cite{mpr10,mpr4} proved that a properly
embedded minimal surface $M$ in $\B -\{ \vec{0}\} $ with $\partial
M\subset \esf^2$ extends across the origin if and only if the function
$K|R|^2$ is bounded on $M$, where $K$ is the Gaussian curvature
function of $M$ and $R^2=x_1^2+x_2^2+x_3^2$ (Theorem~\ref{thmcurvestimstable}
implies that if $M$ is stable, then $K|R|^2$ is bounded).
In fact, this removable singularity result holds true if we replace $\R^3$
by an arbitrary Riemannian 3-manifold (Theorem~\ref{tt2}).
The following conjecture can be proven to hold for any
minimal surface of finite topology (in fact, with finite genus, see Corollary~2.4 in~\cite{mpr11}).

\begin{conjecture} {\bf (Isolated Singularities Conjecture, \ Gulliver-Lawson)}
\label{conjisolsing}
\mbox{}
\newline
The closure of a properly embedded minimal surface with compact boundary
in the punctured ball $\B-\{\vec{0}\}$ is a compact, embedded minimal surface.
\end{conjecture}


The most ambitious conjecture about removable singularities for
minimal surfaces is the following one, which deals with laminations
instead of with surfaces.

\begin{conjecture}\label{FSC} {\bf (Fundamental Singularity Conjecture, \ Meeks-P\'erez-Ros)}
If $A \subset \rth$ is a closed set with zero $1$-dimensional Hausdorff
measure and $\lc$ is a minimal lamination of $\rth -A$, then $\lc$ extends
to a minimal lamination of $\rth$.
\end{conjecture}

In Section~\ref{LRST}, we saw how the Local Removable Singularity Theorem~\ref{tt2} is
a cornerstone for the proof of the Quadratic Curvature Decay Theorem~\ref{thm1introd} and the
Dynamics Theorem in~\cite{mpr20}, which illustrates the usefulness of removable singularities results.

In the discussion of the conjectures that follow, it is helpful
to fix some notation for certain classes
of complete embedded minimal surfaces  in $\rth$.
\begin{itemize}
\item Let ${\mathcal C}$ be the space of connected, ${\mathcal C}$omplete,
 embedded minimal surfaces.
\item Let ${\mathcal P}\subset {\mathcal C}$ be the subspace
of ${\mathcal P}$roperly embedded surfaces.
\item Let ${\mathcal M} \subset {\mathcal P}$ be the subspace of
surfaces with ${\mathcal M}$ore than one end.
\end{itemize}

In what follows we will freely {use}
the properness result of
complete 0-surfaces of finite topology given in
Corollary~\ref{corolthmmlctR3} and Collin's Theorem ~\cite{col1}
that properly embedded minimal surfaces of finite topology
with more than 1 end have finite total curvature.

\begin{conjecture}[Finite Topology Conjecture I, \ Hoffman-Meeks]
\label{conj15.3}
 An orientable surface $M$ of finite topology
  with genus $g$ and $r$ ends, $r \neq 0,2$, occurs as a topological
  type of a surface in ${\mathcal C}$
 if and only
  if $r \leq g+2$.
\end{conjecture}

The main theorem in~\cite{mpr8} insures that for each positive genus $g$, there exists
an upper bound $e(g)$ on the number of ends of an $M \in {\mathcal M}$
with finite topology and genus $g$. Hence, the non-existence implication in
Conjecture~\ref{conj15.3} will be proved if one
can show that $e(g)$ can be taken as $g+2$. Concerning the case $r=2$, the
classification result of  Schoen~\cite{sc1}  implies that the only examples in ${\mathcal M}$
with finite topology and two ends are catenoids.

On the other hand,
Theorem~\ref{helicoid} characterizes the helicoid among complete,
embedded, non-flat minimal surfaces in $\R^3$ with genus zero and
one end. Concerning one-ended surfaces in ${\mathcal C}$ with finite
positive genus, first note that all these surfaces are proper by
Theorem~\ref{thmCM}. Furthermore, every example $M \in {\mathcal P}$ of
finite positive genus and one end has a special analytic
representation on a once punctured compact Riemann surface, as
follows from the works of Bernstein and Breiner~\cite{bb2} and Meeks
and P\'erez~\cite{mpe3}, see Theorems~\ref{ttmr}
{and~\ref{th9.8}.}
In fact, these
authors showed that any such minimal surface has {\it finite
type}\footnote{See Definition~\ref{deffinitetype} for the concept of
minimal surface of finite type.} and is asymptotic to a helicoid.

 All these facts motivate the next
conjecture, which appeared in print for the first time in the
paper~\cite{mr8} by Meeks and Rosenberg, although several versions
of it as questions were around a long time before appearing
in~\cite{mr8}.
The finite type condition and work of Colding and Minicozzi
were applied  by Hoffman, Traizet and White~\cite{htw1,htw2} to prove of the existence
implication of the next conjecture. A step in the proof of the uniqueness statement
of the next conjecture in the
case of genus one is the result of Bernstein and
Breiner~\cite{bb3} that states that every genus $1$ helicoid has an
axis of rotational symmetry; here uniqueness means up to the composition of
an ambient isometry and a homothety.

\begin{conjecture}  {\bf (Finite Topology Conjecture II, \ Meeks-Rosenberg)}
\label{conj15}
For every non-negative integer $g$,
there exists a unique
  non-planar $M \in {\mathcal C}$ with genus $g$ and one end.
\end{conjecture}

The Finite Topology Conjectures I and II together propose the
precise topological conditions under which a non-compact orientable
surface of finite topology can be properly minimally embedded
in $\rth$.  What about the case where the
non-compact orientable surface $M$ has infinite topology? In this case, either
$M$ has infinite genus or $M$ has an infinite number of ends.
Results of Collin, Kusner, Meeks and Rosenberg imply
such an $M$ must have at most two limit ends.
Meeks, P\'erez and Ros proved in~\cite{mpr4} that such an $M$ cannot have one limit end and
finite genus.  The claim is that these restrictions are the
only ones.

\begin{conjecture} [Infinite Topology Conjecture, \ Meeks] \label{conj:meekstop}
A non-compact, orientable surface of infinite topology
occurs as a topological type of a surface in ${\mathcal P}$
if and only if it has at most one or two
limit ends, and when it has one limit end, then its limit end has infinite
genus. \end{conjecture}

Traizet~\cite{tra8} constructed a properly embedded
minimal surface with infinite genus and one limit end, all whose
simple ends are annuli and whose Gaussian curvature function is
unbounded. In a closely related paper, Morabito and
Traizet~\cite{morTra1} constructed a properly embedded minimal
surface with two limit ends, one of which has genus zero and the
other with infinite genus, such that all of its middle ends are annuli.
These results represent progress on Conjecture~\ref{conj:meekstop}.

If $M \in {\mathcal C}$ has finite topology, then $M$ has finite total
curvature or is asymptotic to a helicoid by
Theorems~\ref{thmCM} and \ref{ttmr}. It follows that for any such surface $M$,
there exists a constant $C_M>0$ such that the injectivity radius function
$I_M \colon M \rightarrow (0, \infty]$ satisfies
\[
I_M (p) \geq C_M\| p\| ,\quad p\in M.
\]
Work of Meeks, P\'erez and Ros in~\cite{mpr14,mpr10} indicates that this linear growth property
of the injectivity radius function should characterize the examples in ${\mathcal C}$ with finite topology,
in a similar manner that the inequality $K_M(p)\| p\|^2\leq C_M$ characterizes finite total
curvature for a surface $M\in {\mathcal C}$ (Theorem~\ref{thm1introd}, here $K_M$ denotes the
Gaussian curvature function of $M$).

\begin{conjecture}\label{conjinjradius} {\bf (Injectivity Radius Growth Conjecture,  Meeks-P\'erez-Ros)}

\noindent A surface $M \in {\mathcal C}$ has finite topology if and only if
its injectivity radius function grows at least linearly with respect to the extrinsic
distance from the origin.
\end{conjecture}
The results in~\cite{mpr14,mpr10} and the earlier described
Theorems~\ref{thmlimitlaminCM} and~\ref{t:t5.1CM} also motivated
several conjectures concerning the limits  of locally simply
connected sequences of minimal surfaces in  $\rth$, like the
following one, which in the case that $M$
is allowed to have compact boundary represents the
necessary implication in the embedded Calabi-Yau Conjecture~\ref{CY} below.
\begin{conjecture}\label{conjFTPC} {\bf (Finite Genus Properness Conjecture, \
Meeks-P\'erez-Ros)}

\noindent If  $M \in {\mathcal C}$ and $M$ has finite genus, then $M \in {\mathcal P}$.
\end{conjecture}

In~\cite{mpr9}, Meeks, P\'erez and Ros proved Conjecture~\ref{conjFTPC}
under the additional hypothesis that $M$ has a countable number of ends
(this assumption is necessary for $M$ to be proper
in $\R^3$ by work in~\cite{ckmr1}).

 Conjecture~\ref{conjFTPC}
can be shown to follow from the next beautiful structure conjecture,
which we include here in spite of the fact that it is stated when the ambient space is a general
Riemannian manifold.

\begin{conjecture}\label{conjlocfing} {\bf (Finite Genus Conjecture in 3-manifolds,
 Meeks-P\'erez-Ros)}

\noindent Suppose $M$ is a connected, complete, embedded minimal surface with empty boundary
and finite genus in a Riemannian 3-manifold $N$. Let $\overline{M}=M\cup {\mbox{\rm lim}(M)}$,
where $\mbox{\rm lim}(M)$ is the set of limit points\footnote{See the paragraph just before Theorem~\ref{thmmlct}
for the definition of $\mbox{\rm lim}(M)$.}
of $M$. Then, one of the following possibilities holds.
\begin{enumerate}[\rm 1.]
\item $\overline{M}$ has the structure of a minimal lamination of $N$.
\item $\overline{M}$ fails to have a minimal lamination  structure,
$\mbox{\rm lim}(M)$ is a non-empty minimal lamination of $N$ consisting of stable leaves
and $M$ is properly embedded in $N-\mbox{\rm lim}(M)$.
\end{enumerate}
\end{conjecture}

The next conjecture is a more ambitious version of the previously stated Conjecture~\ref{CY1}.

\begin{conjecture}\label{CY} {\bf (Embedded Calabi-Yau Conjectures, \ Mart\'\i n, Meeks,
Nadirashvili, P\'erez, Ros)} \vspace{-.2cm}
\begin{enumerate}[1.]
\item There exists an $M \in {\mathcal C}$ contained in a bounded domain in
$\rth$.
In particular, ${\mathcal P}\neq {\mathcal C}$.
\item There exists an $M \in {\mathcal C}$ whose closure $\overline{M}$
has the structure of a minimal lamination of a slab, with $M$ as a leaf and with two planes as
limit leaves.
\item A necessary and sufficient condition for a
connected, open topological surface $M$ to admit a
complete bounded minimal embedding in $\R^3$ is that every end of
$M$ has infinite genus.
\item A necessary and sufficient condition for a
connected, open topological surface $M$ to admit a proper minimal
embedding in every smooth bounded domain ${\mathcal D}\subset \R^3$ as a complete
surface is that $M$ is orientable and every end of $M$
has infinite genus.
\item A necessary and sufficient condition for a
connected, non-orientable open topological surface $M$ to admit a
proper minimal embedding in some bounded domain
${\mathcal D}\subset \R^3$ as a complete surface is that every end of
$M$ has infinite genus.
\end{enumerate}
\end{conjecture}

We now discuss two conjectures related
to the underlying conformal structure of a minimal surface.
\begin{conjecture} [Liouville Conjecture, \ Meeks-Sullivan]
\label{conjLiouv}
If $M \in {\mathcal P}$ and $h\colon M \rightarrow{\mathbb{R}}$
is a positive harmonic function, then $h$ is constant.
\end{conjecture}

The above conjecture is closely related to work in~\cite{ckmr1,mpr13,mr7}.  We also remark
that Neel~\cite{nee1} proved that if a surface $M\in {\mathcal P}$
has bounded Gaussian curvature, then $M$ does not admit non-constant bounded
harmonic functions. A related conjecture is the following one:

\begin{conjecture} {\bf (Multiple-End Recurrency Conjecture,  Meeks)}
\label{conjmultipleendrec}
If $M \in {\mathcal M}$, then $M$ is recurrent for Brownian motion.
\end{conjecture}

Assuming that one can prove the last conjecture,
the proof of the Liouville Conjecture would reduce to the case where
$M\in {\mathcal P}$ has
infinite genus and one end. Note that in this setting, a surface
could satisfy Conjecture~\ref{conjLiouv} while at the same time
being transient. For example, by work of Meeks, P\'erez and Ros~\cite{mpr13}
every doubly or triply-periodic minimal
surface with finite topology quotient satisfies the Liouville Conjecture,
and these minimal surfaces are never recurrent.
On the other hand, every doubly or triply-periodic minimal
surface has exactly one end (Callahan, Hoffman and Meeks~\cite{chm3}), which implies that
the assumption in Conjecture~\ref{conjmultipleendrec} that $M \in {\mathcal M}$,
not merely $M \in {\mathcal P}$, is a
necessary one. It should be also noted that the previous two conjectures
need the hypothesis of global embeddedness, since there exist properly
immersed minimal surfaces with two embedded ends and which admit bounded
non-constant harmonic functions~\cite{ckmr1}.

\begin{conjecture}  {\bf (Scherk Uniqueness Conjecture,  Meeks-Wolf)}
\label{conjScherk}
 If $M$ is a connected, properly immersed minimal surface in $\rth$ and
$\mbox{\rm Area}(M \cap \B(R)) \leq 2\pi R^2$ holds in extrinsic balls $\B(R)$ of radius $R$,
then $M$ is a
plane, a catenoid or one of the singly-periodic Scherk minimal surfaces.
\end{conjecture}

By the Monotonicity Formula
(see e.g., \cite{cmCourant}),
any connected, properly immersed minimal surface in $\R^3$ with
\[
\lim_{R\to \infty }R^{-2}\mbox{Area}(M\cap \B (R))\leq 2\pi ,
\]
 is actually embedded.
A related conjecture on the uniqueness of the doubly-periodic Scherk
minimal surfaces was solved by Lazard-Holly and Meeks~\cite{lm2}; they proved that if
$M \in {\mathcal P}$ is doubly-periodic and its quotient surface has genus zero, then
$M$ is one of the doubly-periodic Scherk minimal surfaces.  The basic approach used in \cite{lm2}
was adapted later on by Meeks and Wolf~\cite{mrw1} to prove that Conjecture~\ref{conjScherk}
holds under the assumption that the surface is singly-periodic. We recall that Meeks and Wolf's proof
uses that the Unique Limit Tangent Cone Conjecture below holds
in their periodic setting;
this approach suggests that a good way to solve the general Conjecture~\ref{conjScherk} is
first to prove Conjecture~\ref{conjlimittgtcone} on the uniqueness of the limit tangent cone of $M$,
from which it follows (unpublished work of Meeks and Ros) that $M$ has
two Alexandrov-type planes of symmetry.  Once $M$ is known to have these planes of
symmetry, one can describe the Weierstrass representation of $M$,
which Meeks and Wolf (unpublished) claim would be sufficient to complete
the proof of the
conjecture.

\begin{conjecture}  {\bf (Unique Limit Tangent Cone at Infinity Conjecture, \ Meeks)}
\label{conjlimittgtcone}
 If $M\in {\mathcal P}$ is not a plane and has
extrinsic quadratic area growth, then
  $\lim_{t \rightarrow \infty}
  \frac{1}{t} M$ exists and is a minimal, possibly non-smooth cone over a finite
  balanced configuration
of geodesic arcs in the unit sphere, with common ends points and integer multiplicities.
Furthermore, if $M$ has  area not greater than $2\pi
  R^2$ in
extrinsic balls of radius $R$, then the limit tangent cone of $M$ is either the union of
two planes or consists of a single
  plane with multiplicity two passing through the origin.
 \end{conjecture}

By unpublished work of Meeks and Wolf, the above conjecture is closely
tied to the validity of the next classical one.

\begin{conjecture}{\bf (Unique Limit Tangent Cone at Punctures Conjecture)}
\label{uniqueness}
Let $f\colon M\to \B -\{ \vec{0}\} $ be a proper immersion of a surface
with compact boundary in the punctured unit ball, such that
$f(\partial M)\subset \partial \B $ and whose mean curvature function
is bounded. Then, $f(M)$ has a {\em unique} limit tangent cone at the origin
under homothetic expansions.
\end{conjecture}

 A classical
result of Fujimoto~\cite{fu1} establishes that the
 Gauss map of any orientable, complete, non-flat,  immersed $0$-surface in $\rth$
cannot exclude more than 4 points, which improved the earlier result of Xavier~\cite{xa1} that the Gauss
map of such a surface cannot miss more than 6 points.
If one assumes that a surface $M\in {\mathcal C}$ is periodic
with finite topology quotient, then Meeks, P\'erez and Ros solved the  first item in the
next conjecture~\cite{mpr16}.
Also see Kawakami, Kobayashi and Miyaoka~\cite{KaKoMi1} for related results on this problem,
including some partial results on the conjecture of Osserman that states
that the Gauss map of an orientable, complete, non-flat, immersed $0$-surface
with finite total curvature in $\R^3$ cannot miss 3 points of $\esf^2$.

\begin{conjecture} [Four Point Conjecture, \ Meeks, P\'erez, Ros]
\label{conjFP}
\mbox{}\newline
Suppose $M \in {\mathcal C}$. If the Gauss map of $M$ omits 4 points on $\esf^2$, then $M$ is a
singly or doubly-periodic Scherk minimal surface.
\end{conjecture}

We next deal with the question of when a surface $M\in {\mathcal C}$
has strictly negative Gaussian curvature.
Suppose again that a surface $M\in {\mathcal C}$ has finite topology, and so, $M$
either has finite total curvature or is a helicoid with handles.
It is straightforward to check that
such a surface has negative curvature if and only if it is a catenoid
or a helicoid (note that if $g\colon M\to \C \cup \{ \infty \} $ is the
stereographically projected Gauss map of $M$, then after a suitable
rotation of $M$ in $\R^3$, the meromorphic
differential $\frac{dg}{g}$ vanishes exactly at the zeros of the
Gaussian curvature of $M$; from here one deduces easily that
if $M$ has finite topology and strictly negative Gaussian curvature, then the genus
of $M$ is zero). More generally, if we allow
a surface $M\in {\mathcal C}$ to be invariant under a proper discontinuous group $G$ of isometries of $\R^3$,
with $M/G$
having finite topology, then $M/G$ is properly embedded in $\R^3/G$ by an elementary application of the
Minimal Lamination
Closure Theorem (see Proposition~1.3 in~\cite{PeTra1}). Hence, in this case $M/G$ has finite total
curvature by a result of Meeks and Rosenberg~\cite{mr3,mr2}.
Suppose additionally that $M/G$ has negative curvature, and we will discuss which surfaces are possible.
If the ends of $M/G$ are helicoidal or planar, then
a similar argument using $\frac{dg}{g}$ gives that $M$ has genus zero, and so,
it is a helicoid. If $M/G$ is doubly-periodic,
then $M$ is a Scherk minimal surface, see~\cite{mpr16}.
In the case $M/G$ is singly-periodic,
then $M$ must have Scherk-type ends but we still do not
know if the surface must be a Scherk singly-periodic minimal surface.
These considerations motivate the following conjecture.

\begin{conjecture}
\label{NCC} {\bf (Negative Curvature Conjecture, \, Meeks, P\'erez, Ros)}
If $M\in {\mathcal C}$ has negative curvature, then $M$ is a catenoid, a helicoid
or one of the singly or doubly-periodic Scherk minimal surfaces.
\end{conjecture}

We end this section of conjectures about $H$-surfaces in $\R^3$ by reminding the reader the already stated
Conjecture~\ref{CYconj1} about properness of complete $H$-surfaces in $\R^3$ with finite genus.

\subsection{Open problems in homogeneous 3-manifolds.}
%
%

In all of the conjectures below,  $X$ will denote a simply-connected,
3-dimensional metric Lie group.


\begin{conjecture}[Isoperimetric Domains Conjecture] \label{iso:conj2}
Let  ${ X}$ denote a metric Lie group diffeomorphic to $\rth$. Then:
 \begin{enumerate}[1.]
\item  {Isoperimetric domains (resp. surfaces) in ${ X}$ are topological balls
(resp. spheres).}
\item Immersed $H$-spheres in $X$ are embedded, and the balls that
they bound are isoperimetric domains.
\item For each fixed volume ${ V}_0$, solutions to
the isoperimetric problem in  ${X}$
for volume ${ V}_0$ are unique up to left translations in ${ X}$.
\end{enumerate}
\end{conjecture}

In reference to the following open problems and conjectures,
the reader should note that Meeks, Mira, P\'erez and Ros are in
the final stages of completing paper~\cite{mmpr1} that solves
some of them;
this work should give complete solutions to
Conjectures~\ref{conj3.16*} and~\ref{conj3.19*} below. Their claimed
results would also demonstrate that every $H$-sphere in $X$ has
index one (see the first statement of Conjecture~\ref{conj3.17*}).
In~\cite{mmpr2}, in the case that
$X$ is diffeomorphic to $\rth$, it is shown that as the volumes
of isoperimetric domains in $X$ go to infinity,
their radii\footnote{The radius  of a compact
Riemannian manifold $M$ with boundary is the maximum
distance of points in $M$ to its boundary.} go to infinity  and the
mean curvatures of their boundaries converge to
the critical mean curvature $H(X)$ of $X$ (introduced in Definition~\ref{Crit-Cheg}).
We expect that by the time this survey is
published, \cite{mmpr1} will be available and
consequentially, some parts of this section on open problems
should be updated by the reader to include these new results.


\begin{conjecture}[Hopf Uniqueness Conjecture, Meeks-Mira-P\'erez-Ros]
\label{conj3.16*} \mbox{}\\
 For every $H\geq 0$, any two $H$-spheres
immersed in $X$ differ by a left translation of~$X$.
\end{conjecture}

It is easy to see that the index of an $H$-sphere
$S_H$ immersed in $X$ is at least one; indeed, if $\{F_1,F_2,F_3\}$ denotes a
basis of right invariant vector fields of $X$ (that are Killing vector fields for
the left invariant metric of $X$), then the functions $u_i=\langle F_i,N\rangle$, $i=1,2,3$, are
Jacobi functions on $S_H$ (see Definition~\ref{defJacobif}, here $N$ is a unit normal vector field to $S_H$).
Since right invariant vector fields on $X$ are identically zero or never
zero and spheres do not admit a nowhere zero tangent vector field, then
the functions $u_1,u_2,u_3$ are linearly
independent. Hence, $0$ is an eigenvalue of the Jacobi operator of $S_H$ of multiplicity
at least three. As the first eigenvalue is simple, then $0$ is not
the first eigenvalue of the Jacobi operator and thus, the index of $S_H$ is at
least one. Moreover, if the index of $S_H$ is exactly one, then it
follows from Theorem~3.4 in Cheng~\cite{cheng1} (see
also~\cite{dm2,ross2}) that the nullity of $S_H$ is three. Finally,
recall that every weakly stable compact $H$-surface has index at most one
(two eigenfunctions associated to different negative eigenvalues are $L^2$-orthogonal,
and thus produce a linear combination with zero mean, that contradicts weak stability).
Therefore, a weakly stable $H$-sphere in $X$ has index one and nullity three.
The next conjecture claims that this index-nullity property
does not need the hypothesis on weak stability, and that weak stability holds whenever
$X$ is non-compact.


\begin{conjecture}[Index-one Conjecture, Meeks-Mira-P\'erez-Ros]
\label{conj3.17*}\mbox{}\\ Every $H$-sphere in $X$ has index one.
Furthermore, when $X$ is diffeomorphic to $\rth$,
then every $H$-sphere in $X$ is
{weakly}
 stable.
\end{conjecture}
Note that by Theorem~\ref{main5}, the first statement in
Conjecture~\ref{conj3.17*} holds in the case $X$ is $\mbox{\rm SU}(2)$ with a
left invariant metric. Also note that the hypothesis that $X$ is
diffeomorphic to $\R^3$ in the second statement of
Conjecture~\ref{conj3.17*} is necessary since the second statement
fails to hold in certain Berger spheres, see Torralbo and
Urbano~\cite{tou1}. By Theorem~4.1 in~\cite{mmpr4}, the validity of
the first statement in Conjecture~\ref{conj3.17*} implies
Conjecture~\ref{conj3.16*} holds as well.

Hopf~\cite{hf1} proved that the moduli space of
non-congruent $H$-spheres in $\R^3$ is the interval $(0,\infty)$
(parameterized by their mean curvatures $H$) and all of these
$H$-spheres are embedded and stable, hence of index one; these
results and arguments of Hopf readily extend to the case of $\Hip
^3$ with the interval being $(1,\infty )$ and $\esf^3$ with interval
$[0,\infty )$, both $\Hip ^3$ and $\esf^3$ endowed with their
standard metrics; see Chern~\cite{che5}. By Theorem~\ref{main5},
if $X$ is a metric Lie group diffeomorphic to $\esf^3$, then the
moduli space of non-congruent $H$-spheres in $X$ is the interval
$[0,\infty)$, again parameterized by their mean curvatures $H$.
However, Torralbo~\cite{tor1} proved that some $H$-spheres fail to
be embedded in certain Berger spheres. These results motivate the
next two conjectures.

\begin{conjecture}[Hopf Moduli Space Conjecture, Meeks-Mira-P\'erez-Ros]
\label{conj3.19*} \mbox{}\\
When $X$ is diffeomorphic to $\rth$, then the moduli space of
non-congruent $H$-spheres in $X$ is the interval $ (H(X),\infty)$,
which is parameterized by their mean curvatures $H$. In particular,
every $H$-sphere in $X$ is Alexandrov embedded and $H(X)$ is the
infimum of the mean curvatures of $H$-spheres in $X$.
\end{conjecture}

The results of Abresch and Rosenberg~\cite{AbRo1,AbRo2} and previous
classification results for rotationally symmetric $H$-spheres demonstrate
that Conjecture~\ref{conj3.19*} holds when $X$ is some
$\E(\kappa,\tau)$-space
{(see e.g., \cite{dhm1} for a description of these spaces).
Work of Daniel and
Mira~\cite{dm2} and of Meeks~\cite{me34} imply that
{Conjectures~\ref{iso:conj2}, \ref{conj3.16*}, \ref{conj3.17*} and~\ref{conj3.19*}
}
hold for $\sol$ with its standard metric.

%
%

The next conjecture is
known to hold in the flat $\R^3$ as proved by Alexandrov~\cite{aa1}
and subsequently extended to $\Hip ^3$ and to a hemisphere of
$\esf^3$.

\begin{conjecture}[Alexandrov Uniqueness Conjecture] \label{conjAlex}
If $X$ is diffeomorphic to $\R^3$, then the only
compact, Alexandrov embedded $H$-surfaces in $X$ are topologically
spheres.
\end{conjecture}
In the case
there exist two orthogonal foliations of $X$ by planes of
reflectional symmetry, as is the case of $\sol$ with
 its standard metric,
then using the Alexandrov reflection method, the last
conjecture is known to hold;
{see~\cite{dm2} for
details in the special case of the standard metric on Sol$_3$.
}

If Conjecture~\ref{conjAlex} holds, then the unique
compact $H$-surfaces which bound regions in $X$ are
constant mean curvature spheres.
In particular, one would have the validity of
items~1 and 3 of Conjecture~\ref{iso:conj2}.

Although we do not state it as a conjecture, it is generally
believed that for any value of $H>H(X)$ and $g\in \N$, there exist
compact, genus-$g$, immersed, non-Alexandrov embedded $H$-surfaces
in $X$, as is the case in classical $\R^3$ setting (Wente~\cite{we1}
and Kapouleas~\cite{kap2}).

\begin{conjecture}[Stability Conjecture for $\mbox{\rm SU}(2)$, Meeks-P\'erez-Ros]
\label{conjstab} \mbox{}\\
 If $ X$ is diffeomorphic to $\esf^3$, then $X$ contains no strongly
 stable (the 2-sided cover admits a positive Jacobi function) complete $H$-surfaces.
\end{conjecture}
Conjecture~\ref{conjstab} holds when the metric Lie group $X$ is in one of the following
two cases:
\begin{itemize}
\item $X$ is a Berger sphere with non-negative scalar curvature (see item~(5) of Corollary~9.6 in
Meeks, P\'erez and Ros~\cite{mpr19}).
\item $X$ is SU$(2)$ endowed with a left invariant metric of positive scalar curvature
(by item~(1) of Theorem~2.13 in~\cite{mpr19}, a complete
stable $H$-surface $\Sigma $
in $X$ must be compact, in fact must be topologically a two-sphere or a projective plane; hence one
could find a right invariant Killing field on $X$ which is not
tangent to $\Sigma$ at some point of $\Sigma $, thereby inducing
a Jacobi function which changes sign on $\Sigma $, a contradiction).
\end{itemize}
It is also proved in~\cite{mpr19} that
if $Y$ is a 3-sphere with a Riemannian metric (not necessarily a
left invariant metric) such that it admits no  stable
complete minimal surfaces, then for each integer $g\in \N\cup\{0\}$,
the space of compact embedded minimal surfaces of genus $g$ in $Y$
is compact, a result which is known to hold for Riemannian metrics
on $\esf^3$ with positive Ricci curvature (Choi and
Schoen~\cite{cs1}).

\begin{conjecture}[Stability Conjecture, Meeks-Mira-P\'erez-Ros]
\label{conj:stable}
Suppose $X$ is diffeomorphic to $\R^3$. Then
\begin{equation}
\label{eq:conj:stable}
H(X)=\sup \{ \mbox{mean curvatures of complete  stable
$H$-surfaces in $X$}\} .
\end{equation}
\end{conjecture}

Regarding Conjecture~\ref{conj:stable}, define $\widehat{H}(X)$ to be the
supremum in the right-hand-side of (\ref{eq:conj:stable}).
By Theorem~1.5 in~\cite{mmpr2}, there  exists a
properly embedded, complete stable $H(X)$-surface in $X$
that is part of an $H(X)$-foliation of $X$. Thus, $H(X)\leq \widehat{H}(X)$.
}

 Remember from
the discussion after Definition~\ref{Crit-Cheg} that the main difficulty in completing
the proof of the generalization of Theorem~\ref{main5} to the case that $X$ is
diffeomorphic to $\R^3$, is to obtain the area estimates $(\star)$ for index-one
$H$-spheres in $X$. We next explain why the validity of Conjecture~\ref{conj:stable}
would imply that the area estimates $(\star )$ hold. To see this, consider
a sequence of index-one spheres $S_{H_n}$ immersed in $X$
with $H_n\searrow H_{\infty }\geq 0$ and with areas diverging to infinity.
In~\cite{mmpr4} it is proved that one can produce an appropriate limit of
left translations of $S_{H_n}$ which is a  stable $H_{\infty }$-surface
in $X$. Therefore, $H_{\infty }\leq \widehat{H}(X)$. As by definition $H(X)\leq
H_n$ for all $n\in \N$, then $H(X)\leq H_{\infty }$. As $H(X)=\widehat{H}(X)$
because we are assuming the validity of Conjecture~\ref{conj:stable}, then
we conclude that $H_{\infty }=H(X)$,
which proves the area estimates $(\star )$. This argument also shows that
the validity of Conjecture~\ref{conj:stable} would imply that
both Conjecture~\ref{conj3.16*}
and the first statement in Conjecture~\ref{conj3.17*} hold.


\begin{conjecture}[CMC Product Foliation Conjecture, Meeks-Mira-P\'erez-Ros]
\label{conj:product}
\mbox{} \vspace{-.6cm}
\ben
\item
If $X$ is diffeomorphic to $\R^3$, then given $p\in X$ there exists
a smooth CMC product foliation of $X-\{ p \} $ by spheres.
\item Let $\mathcal{F}$ be a CMC foliation of $X$, i.e., a foliation all
whose leaves have constant mean curvature (possibly varying from
leaf to leaf). Then
$\mathcal{F}$ is a product foliation by topological planes with absolute
mean curvature function bounded from above by $H(X)$.
\end{enumerate}
\end{conjecture}
Since spheres of radius $R$ in $\rth$ or in $\Hip^3$ have constant
mean curvature, item~(1) of the above conjecture holds in these
spaces. In fact it can be shown that the conjecture holds if
the isometry group of $X$ is at least four-dimensional.


Regarding item~(2) of Conjecture~\ref{conj:product}, we remark that the existence of
a CMC foliation in $X$ implies that $X$ is diffeomorphic to $\R^3$. To see this, we argue
by contradiction: suppose that $\mathcal{F}$ is a CMC foliation of a
metric Lie group diffeomorphic to $\esf^3$. Novikov~\cite{nov1}
proved that any foliation of $\esf^3$ by surfaces has a Reeb
component $C$, which is topologically a solid doughnut with a
boundary torus leaf $\partial C$ and the other leaves of
$\mathcal{F}$ in $C$  all have $\partial C$ as their limits sets.
Hence, all of leaves of $\mathcal{F}$ in $C$ have the same mean
curvature as $\partial C$. By the Stable Limit Leaf Theorem for
$H$-laminations, $\partial C$ is  stable. But an embedded
compact, two-sided $H$-surface in SU$(2)$ is never  stable,
since some right invariant Killing field induces a Jacobi function
which changes sign on the surface.

Suppose for the moment that item~(1) in
Conjecture~\ref{conj:product} holds and we will point out some
important consequences. Suppose $\mathcal{F}$ is a smooth CMC
product foliation of $X-\{ p\}$ by spheres, $p$ being a point in
$X$. Parameterize the space of leaves of $\mathcal{F}$ by their mean
curvature; this can be done by the maximum principle for
$H$-surfaces, which shows that the spheres in $\mathcal{F}$ decrease
their positive mean curvatures at the same time that the volume of
the enclosed balls by these spheres increases. Thus, the mean
curvature parameter for the leaves of $\mathcal{F}$ decreases from
$\infty $ (at $p$) to some value $H_0\geq 0$.
We claim that
\begin{quote}
{\it $H_0=H(X)$ and every compact $H$-surface in $X$
satisfies $H>H(X)$.}
\end{quote}
To see the claim, we argue by contradiction. Suppose that
there exists an immersed closed surface $M$ in $X$ such that the maximum value of the
absolute mean curvature function of $M$ is less than or equal to
$H_0$. Since $M$ is compact, then $M$ is contained in the ball
enclosed by some leaf $\Sigma $ of $\mathcal{F}$. By left
translating $M$ until touching $\Sigma $ at a first time, we obtain a
contradiction to the usual comparison principle for the mean
curvature, which finishes the proof of the claim. With this property in mind,
we now list some consequences of item~(1) in
Conjecture~\ref{conj:product}.
\begin{enumerate}
\item
All leaves of $\mathcal{F}$ have index one. This is because
the leaves of $\mathcal{F}$ bounding balls of small volume
have this property and as the volume increases, the multiplicity of
zero as an eigenvalue of the Jacobi operator of the corresponding
boundary sphere cannot exceed three by Cheng's theorem~\cite{cheng1}.
\item
All leaves of $\mathcal{F}$ are weakly stable. To see this,
note that every function $\phi $ in the nullity of a leaf $\Sigma $
of $\mathcal{F}$ is induced by a right invariant Killing field on
$X$ (this is explained in the paragraph just before Conjecture~\ref{conj3.17*}),
and hence, $\int _{\Sigma }\phi =0$ by the Divergence Theorem
applied to the ball enclosed by $\Sigma $. In this situation,
Koiso~\cite{ko1} proved that the weak stability of $\Sigma $ is
characterized by the non-negativity of the integral $\int _{\Sigma}
u$, where $u$ is any smooth function on $\Sigma$ such that $Lu=1$ on
$\Sigma $ (see also Souam~\cite{so3}). Since the leaves of
$\mathcal{F}$ can be parameterized by their mean curvatures, the
corresponding normal part $u$ of the associated variational field satisfies
$u>0$ on $\Sigma $, $Lu=1$ and $\int _{\Sigma }u>0$. Therefore,
$\Sigma $ is weakly stable.

\item The leaves of $\mathcal{F}$ are the unique
$H$-spheres in $X$ (up to left translations), by
Theorem~\ref{thm:index1}.
\end{enumerate}

If additionally the Alexandrov Uniqueness Conjecture~\ref{conjAlex}
holds, then the constant mean curvature spheres in $\mathcal{F}$ are
the unique (up to left translations) compact $H$-surfaces in $X$
which bound regions. As explained in the second paragraph just after Conjecture~\ref{conjAlex},
this would also imply that the leaves of ${\cal F}$ are the unique (up to left translations)
solutions to the isoperimetric problem in $X$.
%
%


The next conjecture is motivated by the isoperimetric inequality of
White~\cite{wh13}.

\begin{conjecture}[Isoperimetric Inequality Conjecture, Meeks-Mira-P\'erez-Ros]
\label{conj:isoper}
Suppose that $X$ is diffeomorphic to $\rth$. Given any $\ve>0$ and $L_0>0$, there
exists  $C(\ve,L_0)$ such that for any compact immersed surface $\Sigma$
in $X$ with connected boundary of length at most $L_0$ and which is minimal or has
absolute mean curvature function bounded from above by $\mathrm{Ch}(X)-\ve$, then
$$\mathrm{Area}(\Sigma)\leq C(\ve,L_0).$$
\end{conjecture}

The next conjecture exemplifies another aspect of the special role
that the critical mean curvature $H(X)$  of $X$ might play in the
geometry of $H$-surfaces in $X$.

\begin{conjecture}[Stability Conjecture, Meeks-Mira-P\'erez-Ros]
 A complete,  stable $H$-surface $\Sigma $ in $X$ with
$H=H(X)$ is
an entire graph with respect to some Killing field $V$,
i.e., every integral curve of $V$ intersects exactly once to $\Sigma $ (transversely).
In particular, if $H(X)=0$, then any
complete,  stable minimal surface $\Sigma$ in $X$ is a leaf
of a minimal foliation of $X$ and so $\Sigma$ is actually
homologically area-minimizing in $X$.
\end{conjecture}

The previous conjecture is closely related to the next conjecture,
which in turn is closely tied to recent work of Daniel, Meeks and
Rosenberg~\cite{dmr2,dmr1} on halfspace-type theorems in
simply-connected, 3-dimensional metric semidirect products.

\begin{conjecture}[Strong-Halfspace Conjecture in Nil$_3$, Daniel-Meeks-Rosenberg]
A complete,  stable minimal surface in {\em Nil}$_3$ is
{either an entire}
graph with respect to the Riemannian submersion $\Pi \colon
\mbox{\em Nil}_3 \to \R^2$ or a vertical plane $\Pi ^{-1}(l)$,
where $l$ is a line in $\R^2$. In particular, by the results
in~\cite{dmr1}, any two properly immersed disjoint minimal surfaces
in {\em Nil}$_3$ are parallel vertical planes or they are entire
graphs $F_1, F_2$ over $\R^2$, where $F_2$ is  a vertical
translation of $F_1$.
\end{conjecture}

\begin{conjecture}[Positive Injectivity Radius,
Meeks-P\'erez-Tinaglia] \label{conj:positive}
A complete embedded $H$-surface of finite topology in $X$  has
positive injectivity radius.
\end{conjecture}

Conjecture~\ref{conj:positive} is motivated by the partial result of
Meeks and P\'erez~\cite{mpe17}
that the injectivity radius of a
complete, embedded {\it minimal} surface of finite topology in a
homogeneous 3-manifold is positive (hence Conjecture~\ref{conj:positive}
holds for $H=0$).
A related result of Meeks and Per\'ez~\cite{mpe17}
when $H=0$ and of Meeks and Tinaglia (unpublished) when
$H>0$, is that if $Y$ is a complete locally homogeneous
3-manifold {with positive injectivity radius}
and $\Sigma$ is a complete embedded $H$-surface in
$Y$ with finite topology, then the injectivity radius function of
$\Sigma$ is {bounded away from zero on compact domains in $Y$.
Meeks and Tinaglia
(unpublished) have also shown that Conjecture~\ref{conj:positive}
holds for complete embedded $H$-surfaces of finite topology
in metric Lie groups $X$ with four or six-dimensional isometry
group.

\begin{conjecture}[Bounded Curvature Conjecture,
Meeks-P\'erez-Tinaglia] A complete embedded $H$-surface of finite
topology in $X$ with $H>0$ has bounded second
fundamental form.
\end{conjecture}

The previous two conjectures are related as follows. Curvature
estimates of Meeks and Tinaglia ~\cite{mt1} for embedded $H$-disks
imply that every complete embedded $H$-surface with $H>0$ in a
homogeneously regular 3-manifold has bounded second fundamental
form if and only if it has positive injectivity radius.

\begin{conjecture}[Calabi-Yau Properness Problem, Meeks-P\'erez-Tinaglia]
 \label{conj:CY}
\mbox{}\\
A complete, connected, embedded $H$-surface of positive injectivity
radius in $X$ with $H\geq H(X)$ is always proper.
\end{conjecture}

In the classical setting of $X=\rth$, where $H(X)=0$,
Conjecture~\ref{conj:CY} was proved by Meeks and
Rosenberg~\cite{mr13} for the case $H=0$. This result was based on
work of Colding and Minicozzi~\cite{cm35} who demonstrated that
complete embedded minimal surfaces in $\rth$ with  finite topology are
proper, thereby proving what is usually referred to as the {\it
classical embedded Calabi-Yau problem} for finite topology minimal
surfaces. Recently, Meeks and Tinaglia~\cite{mt7} proved
Conjecture~\ref{conj:CY}  in the case $X=\rth$  and $H>0$, which
completes the proof of the conjecture in the classical setting.

As we have already mentioned, Meeks and P\'erez~\cite{mpe17} have
shown that every complete embedded minimal surface $M$ of finite
topology in  $X$ has positive injectivity radius;  hence $M$ would
be proper whenever $H(X)=0$ and Conjecture~\ref{conj:CY}  holds for
$X$. Meeks and Tinaglia~\cite{mt11} have shown that any complete
embedded $H$-surface $M$ in a complete 3-manifold $Y$ with
constant sectional curvature $-1$  is proper provided that $H\geq 1$
and $M$ has injectivity radius function bounded away from zero on
compact domains of in $Y$;  they also proved that any complete,
embedded, finite topology $H$-surface in such a $Y$  has bounded
second fundamental form. In particular, for $X=\Hip^3$ with its
usual metric, an annular end of any complete, embedded, finite
topology $H$-surface in $X$ with $H\geq H(X)=1$ is asymptotic to an
annulus of revolution by the classical results of Korevaar, Kusner,
Meeks and Solomon~\cite{kkms1} when $H>1$ and of Collin, Hauswirth
and Rosenberg~\cite{chr1} when $H=1$.

The next conjecture is motivated by the classical results of Meeks
and Yau~\cite{my5} and of Frohman and Meeks~\cite{fme4} on the
topological uniqueness of minimal surfaces in $\rth$ and partial
unpublished results by Meeks.

\begin{conjecture}[Topological Uniqueness Conjecture, Meeks]
\label{CT}
If $M_1,M_2$ are two  diffeomorphic, connected, complete embedded
$H$-surfaces of finite topology in $X$ with $H=H(X)$, then there
exists a diffeomorphism $f\colon X\to X$ such that $f(M_1)=M_2$.
\end{conjecture}

We recall that Lawson~\cite{la2} proved a beautiful unknottedness
result for minimal surfaces in $\esf^3$ equipped with a metric of
positive Ricci curvature. He demonstrated that whenever $M_1, M_2$
are compact, embedded, diffeomorphic minimal surfaces in such a
Riemannian 3-sphere, then $M_1$ and $M_2$ are ambiently
isotopic. His result was generalized by Meeks, Simon and
Yau~\cite{msy1} to the case of metrics of non-negative scalar
curvature on $\esf^3$.
Meeks and P\'erez  proved the above conjecture
in the case that $X$ is diffeomorphic to $\esf^3$; see Corollary~4.19 in~\cite{mpe11}.

The next conjecture is motivated by the classical case of $X=\rth$,
where it was proved by Meeks~\cite{me17}, and in the case of
$X=\Hip^3$ with its standard constant $-1$ curvature metric, where
it was proved by Meeks and Tinaglia~\cite{mt11}.

\begin{conjecture}[One-end / Two-ends Conjecture, Meeks-Tinaglia]
\mbox{}\\
Suppose that $M$ is  a connected, non-compact, properly embedded
$H$-surface  of finite topology in $X$ with $H> H(X)$.  Then:
\ben
\item $M$ has more than one end.
\item If $M$ has two ends, then $M$ is an annulus.
\een
\end{conjecture}

The previous conjecture also motivates the next one.

\begin{conjecture}[Topological Existence Conjecture, Meeks]
\label{conj:topconj}
\mbox{}\\
Suppose $X$ is diffeomorphic to $\R^3$. Then for every $H>H(X)$, $X$
admits connected properly embedded $H$-surfaces of every possible
orientable topology, except for connected finite genus surfaces with
one end or connected finite positive genus surfaces with 2 ends which
it never admits.
\end{conjecture}

Conjecture~\ref{conj:topconj} is probably known in the classical
settings of $X=\rth \mbox{ and }\Hip^3$ but the authors do not have
a reference of this result for either of these two ambient spaces.
For the non-existence results alluded to in this conjecture in these
classical settings see~\cite{kkms1,kks1,me17,mt11}.  The existence
part of the conjecture should follow from gluing constructions
applied to  collections of non-transversely intersecting
embedded $H$-spheres appropriately placed in $X$, as in the
constructions of Kapouleas~\cite{kap1} in the case of $X=\rth$.

We end our discussion of open problems in $X$ with the following
generalization of the classical properly embedded Calabi-Yau problem
in $\rth$; see {item~3 of}
Conjecture~\ref{CY}.
Variations of this conjecture can be attributed  to many people but
in the formulation below, it is primarily due to Mart\'\i n, Meeks,
Nadirashvili, P\'erez and Ros and their related work.

\begin{conjecture}[Embedded Calabi-Yau Problem]
\label{conn:CY}
Suppose $X$ is diffeomorphic to $\rth$ and $\Sigma$ is a connected,
non-compact surface. A necessary and sufficient condition for
$\Sigma$ to be diffeomorphic to some complete, embedded bounded
minimal surface in $X$ is that every end of $\Sigma$ has infinite
genus.
\end{conjecture}

In the case of $X=\rth$ with its usual metric, the non-existence
implication in the last conjecture was proved by Colding and
Minicozzi~\cite{cm35} for complete embedded minimal surfaces with an
annular end; also see the related more general results of Meeks and
Rosenberg~\cite{mr13} and of Meeks, Per\'ez and Ros~\cite{mpr9}.
%


\center{William H. Meeks, III at profmeeks@gmail.com\\
Mathematics Department, University of Massachusetts, Amherst, MA 01003}
\center{Joaqu\'\i n P\'{e}rez at jperez@ugr.es\\
Department of Geometry and Topology and Institute of Mathematics
(IEMath-GR), University of Granada, 18071 Granada, Spain}
\center{Giuseppe Tinaglia at  giuseppe.tinaglia@kcl.ac.uk\\
Department of Mathematics, Kings College London, London, WC2R 2LS,
U.K.}

\bibliographystyle{plain}

\bibliography{bill}

\end{document}